\documentclass[11pt]{article}

\usepackage{amssymb,commath,
graphics,amsmath,amsthm,amsopn,amstext,amsfonts,hyperref,fullpage,graphicx,
fancybox,verbatim}

\usepackage{fancyhdr,color,mathtools, algorithmic,cite}
\usepackage{xcolor}
\usepackage{verbatim}
\def\doubleunderline#1{\underline{\underline{#1}}}
\def\doubletilde#1{\widetilde{\widetilde{#1}}}

\def\doubleunderline#1{\underline{\underline{#1}}}
\def\doubletilde#1{\widetilde{\widetilde{#1}}}

\DeclareFontFamily{U}{mathx}{\hyphenchar\font45}
\DeclareFontShape{U}{mathx}{m}{n}{
      <5> <6> <7> <8> <9> <10>
      <10.95> <12> <14.4> <17.28> <20.74> <24.88>
      mathx10
      }{}
\DeclareSymbolFont{mathx}{U}{mathx}{m}{n}
\DeclareFontSubstitution{U}{mathx}{m}{n}
\DeclareMathAccent{\widecheck}{0}{mathx}{"71}

\newtheorem{theorem}{Theorem}
\newtheorem{proposition}[theorem]{Proposition}
\newtheorem{lemma}[theorem]{Lemma}
\newtheorem{corollary}[theorem]{Corollary}
\newtheorem{definition}[theorem]{Definition}
\newtheorem{example}[theorem]{Example}
\newtheorem{remark}[theorem]{Remark}

\newcommand{\gap}{\vspace{0.1in}}

\newcommand{\epc}{\hspace{1pc}}

\newcommand{\onebld}{{\bf 1}}
\newcommand{\wt}{\widetilde}
\newcommand{\wh}{\widehat}

\title{
Nonconvex and Nonsmooth Approaches for \\
Affine Chance-Constrained Stochastic Programs}

\author{
Ying Cui\footnote{Department of Industrial and Systems Engineering, University
of Minnesota, Minneapolis, U.S.A.\ 55455. {\tt Email: yingcui@umn.edu.}} \and
Junyi Liu\footnote{Department of Industrial Engineering, Tsinghua University,
Beijing, China 100084.  The work of this author was initiated when she was affiliated
with the Daniel J.\ Epstein
Department of Industrial and Systems Engineering at the University of Southern
California.  {\tt Email: junyiliu@tsinghua.edu.cn.}}
\and Jong-Shi Pang\footnote{The
Daniel J.\ Epstein Department of Industrial and Systems Engineering, University of
Southern California, Los Angeles, U.S.A.\ 90089.
This work was based on research supported by the  U.S.\ Air Force Office of
Sponsored Research under grant FA9550-18-1-0382. {\tt Email: jongship@usc.edu.}}
}

\date{Original: May 2021; previous revision January 2022; second revision \today}

\begin{document}

\maketitle

\begin{abstract}
\noindent Chance-constrained programs (CCPs) constitute a difficult class of
stochastic programs due to its {possible} nondifferentiability and nonconvexity
even with simple linear random functionals. Existing approaches for solving the CCPs
mainly deal with convex random functionals within the probability function.
In the present paper, we consider two generalizations of the class of chance constraints
commonly studied in the literature;
one generalization involves probabilities of disjunctive nonconvex functional events
and the other generalization
involves mixed-signed affine combinations of the resulting probabilities; together,
we coin the term
{\sl affine chance constraint (ACC) system} for these generalized chance constraints.
Our proposed treatment of such an ACC system involves the fusion of several individually
known ideas: (a) parameterized upper and lower
approximations of the indicator function in the
expectation formulation of probability; (b) external (i.e., fixed) versus internal
(i.e.,\ sequential) sampling-based approximation of the expectation operator;
(c) constraint penalization as relaxations of feasibility; and (d) convexification
of nonconvexity and nondifferentiability via surrogation.
The integration of these techniques for solving the affine chance-constrained
stochastic program (ACC-SP) is the main contribution of this paper.  Indeed,
combined together, these ideas lead to several algorithmic strategies with various
degrees of practicality and computational efforts for the nonconvex ACC-SP.
In an external sampling scheme, a given sample batch (presumably large) is applied to
a penalty formulation of a fixed-accuracy approximation
of the chance constraints of the problem via their expectation formulation.
This results in a sample average approximation scheme, whose almost-sure convergence
under a directional derivative condition to a Clarke stationary solution of the
expectation constrained-SP as the sample sizes tend to infinity is established.
In contrast, sequential sampling, along with surrogation leads to a sequential
convex programming based algorithm whose asymptotic convergence for fixed- and
diminishing-accuracy approximations of the indicator function can be established
under prescribed increments of the sample sizes.

\end{abstract}


\section{Introduction}

Chance constrained programs (CCPs) are a class of stochastic optimization problems that
restrict the likelihood of undesirable outcomes
from a system within a prescribed tolerance.
\textcolor{black}{The focus of our study is the following stochastic program
with {\sl affine chance constraints} (ACCs):
\begin{equation} \label{eq:focus sp_dc_constraint}
\begin{array}{ll}
\displaystyle{
\operatornamewithlimits{\mbox{\bf minimize}}_{x \, \in \, X}} \quad &
\bar c_0(x) \, \triangleq \, \mathbb{E}[ \,c_0(x,\tilde{z} )\, ] \\ [0.1in]
\mbox{\bf subject to} \quad & \displaystyle{
\sum_{\ell=1}^L
} \, e_{k \ell} \, \mathbb{P}\left( \, {\cal Z}_{\ell}(x,\tilde{z} ) \, \geq \, 0 \,
\right) \, \leq \, \zeta_k, \epc k \, \in \, [ \, K \, ] \, \triangleq \,
\{ \, 1, \cdots, K \, \},
\end{array}
\end{equation}
where $X$ is a deterministic constraint set contained in the
open set ${\cal O} \subseteq \mathbb{R}^n$;
$c_0:{\cal O} \times \Xi \to \mathbb{R}$ is a random functional,
$\tilde{z} : \Omega \to \Xi$ is
a random vector (i.e., a measurable function) defined on the sample space $\Omega$
with values in $\Xi \subseteq \mathbb{R}^d$
whose realizations we write without the tilde (i.e.,
$z = \tilde{z}(\omega) \in \Xi$ for
$\omega \in \Omega$); $\mathbb{P}$ is the probability measure
defined on the sigma algebra ${\cal F}$ that is generated by subsets of $\Omega$;
each $e_{k\ell}$ is a scalar with the signed decomposition
$e_{k\ell} = e_{k\ell}^+ - e_{k\ell}^-$, where $e_{k\ell}^{\pm} \geq 0$ are
the nonnegative and nonpositive parts of $e_{k\ell}$, respectively;
each ${\cal Z}_{\ell} :{\cal O} \times \Xi \to \mathbb{R}$ for
$\ell \in [ L ] \triangleq \{ \, 1, \cdots, L \, \}$ is a bivariate function to be
specified in Section~\ref{sec:piecewise probabilistic constraints}; and each
$\zeta_k$ is a given threshold.}
A special case of (\ref{eq:focus sp_dc_constraint}) is the simplified form
(with $\zeta_k \in ( 0,1 ]$)
\begin{equation}\label{general ccp}
\displaystyle{
\operatornamewithlimits{\mbox{\bf minimize}}_{x\in X}
} \ \bar{c}_0(x) \epc
\mbox{\bf subject to } \ \mathbb{P}({\cal Z}_k (x,\tilde{z} ) \geq 0) \, \leq \,
\zeta_k, \epc k \in [ K ],
\end{equation}
that is the focus of study in much of the literature on chance-constrained SP
\cite{Henrion05}.  There are two prominent departures
of (\ref{eq:focus sp_dc_constraint}) from the traditional case (\ref{general ccp}):
(a) some coefficients $e_{k\ell}$ may be negative, and (b) each functional
${\cal Z}_{\ell}(\bullet,z)$ is nonconvex and nondifferentiable.
We postpone the detailed discussion of these features until the next section.
Here, we simply note that with these two distinguished features, the formulation
(\ref{eq:focus sp_dc_constraint}) covers much
broader applications and requires non-traditional treatment with novel theoretical
tools and computational methods.
The latter constitutes the main contribution of our work.


\gap

It is well known that the feasible regions of the CCPs are usually nonconvex even for
the linear random functionals and the resulting
optimization problems are NP-hard \cite{Luedtke10,nemirovski12}.
This nonconvexity partially makes the CCPs one of the most challenging
stochastic programs to solve.
With over half a century of research, there is an extensive literature on the
methodologies and applications of the CCPs.
Interested readers are referred to the review papers \cite{AhmedXie18,GengXie19},
book chapters \cite{Prekopa70, Dentcheva06}, the monograph
\cite{ShapiroDentchevaRuszczynski09}, and the lecture notes \cite{Henrion05} for
detailed discussion.

\gap

One direction of developing numerical algorithms for the CCPs focuses on
special probability distributions and random functionals, where the
multi-dimensional probability function and its subdifferential can be evaluated
either directly \cite{Henrion07}
or efficiently via numerical integration \cite{vanAckooij18,Hantoute19}.
However, this direct approach does not work for  random functionals with general
and possibly unknown probability distributions.
Numerous methods dealing with general probability distributions include the
scenario approximation approach \cite{CalafioreCampi05, NemirovskiShapiro},
the $p$-efficient point \cite{Dentcheva00, Dentcheva04}
and the sample average approximation (SAA) \cite{Sen92, LuedtkeAhmed08,PAhmedShapiro09}.
In fact,  for any random variable $Z$, it holds that
\begin{equation} \label{eq:probability and expectation}
\mathbb{P}(Z > 0 ) = \mathbb{E}\left[ \, \mathbf{1}_{( 0, \infty)}(Z) \, \right]
\epc \mbox{and similarly} \epc
\mathbb{P}(Z \geq 0 ) = \mathbb{E}\left[ \, \mathbf{1}_{[ \, 0, \infty)}(Z) \, \right],
\end{equation}
where $\mathbf{1}_{( 0, \infty )}( \bullet )$ is the indicator function of the interval
$( 0,\infty )$; i.e.,
$\mathbf{1}_{(0, \infty)}(t) \triangleq \left\{ \begin{array}{ll}
1 & \mbox{if $t > 0$} \\
0 & \mbox{otherwise;}
\end{array} \right.$  Similarly for $\mathbf{1}_{[  0, \infty )}( \bullet )$.
We call these indicator functions the open and closed Heaviside functions, respectively.
The above equalities indicate that the CCPs under discrete or discretized distributions
are in principle mixed integer programs (MIPs) that can be solved either by
mixed integer algorithms or continuous approximation methods.  The former approach
leverages auxiliary binary variables and adds
the big-M constraints into the lifted feasible set \cite{Luedtke10}.   Strong
formulations can also be derived based on specific
forms of the random functionals in the constraints.  One may consult
\cite{SimgeJiang21} for a recent survey of the MIP approach
for solving the linear CCPs.  Although the MIP approach has the advantage of
yielding globally optimal solutions, it may not work
efficiently in practice when the sample size is large or when nonlinear random
functionals are present;   {its appeal for general distributions
may diminish when considering inherent discretization and its effect};
the choice of the scalar $M$ is potentially a serious bottleneck.
Conservative convex approximations of the CCPs \cite{nemirovski2007convex} are proposed
to resolve this scalability issue. To tighten
these convex approximations, recent research \cite{CaoZavala20,GHKLi17} has proposed
using nonconvex smooth functions  as surrogates of the Heaviside functions
in the expectation \eqref{eq:probability and expectation}; the references
\cite{hong2014conditional,hong2011sequential, POrdieresLuedtkeWachter20}
further proposed a sample-based scheme using difference-of-convex or other nonconvex
smooth functions to deal with general probability distributions.

\gap

In this paper, we consider the continuous nonconvex approximation methods to solve
the generalized class of CCPs (\ref{eq:focus sp_dc_constraint}).  It is worth
mentioning that the primary goal of the present
paper is neither about proposing new approximation schemes of the CCPs,
nor about the comparison of which approximation scheme for the chance constraints is
more effective; but rather, we aim to provide a systematic
and rigorous mathematical treatment of the nonseparable co-existence of nonconvexity
and nondifferentiability in a class of CCPs
that extend broadly beyond the settings commonly studied in the literature.
Below, we give an overview of the distinguished features
of our model and highlight the prevalence of nonconvexity and nondifferentiability:

\gap

\noindent (a) we treat affine combinations of probability functions in the constraints,
which are central to the first-order stochastic dominance of random variables, but
cannot be written in the form
(\ref{general ccp}) due to the mixed signs of the coefficients in the combinations;

\gap

\noindent (b) we approximate the discontinuous Heaviside functions within the expectation
formulation of the probability by nonsmooth and nonconvex functions and treat them
faithfully; this double ``non''-approach enriches the traditional family of convex
and/or smooth approximations;

\gap

\noindent (c) pointwise maximum and/or minimum operators are present within the
probabilities; these operators provide algebraic descriptions of conjunctive and/or
disjunctive random functional inequalities, and thus logical relations among these
inequalities  whose probabilities are constrained; and

\gap

\noindent (d) the resulting random functionals ${\cal Z}_{\ell}(\bullet,\tilde{z})$ within
the probabilities are specially structured dc (for difference-of-convex)
functions; more precisely, each can be expressed as a difference of two pointwise
maxima of finitely many differentiable convex functions.

\gap

Mathematical details of this framework and realistic examples of the random functionals
are presented in the next section.  Throughout this paper, the class of functions in
point (d) above plays a central role, \textcolor{black}{although the probability function
$\mathbb{P}({\cal Z}_{\ell}(\bullet,\tilde{z}) \geq 0)$ may be discontinuous in general, and
not of the dc kind even if it is continuous.}
By a result of \cite{Scholtes02}, piecewise affine
functions, which constitute the most basic class of difference-of-convex
functions, can be expressed as the difference of two pointwise maxima of finitely many
affine functions; see \cite[Subsection~4.4.1]{CuiPang2020} for details.  Extending this
basic result, a related development is the paper \cite{Royset20} which
shows that every upper semicontinuous function is the limit of a hypo-convergent sequence
of piecewise affine functions.  Compared with the linear or convex random functionals
considered in the existing literature of CCPs, our overall modeling framework
with nonconvex and nondifferentiable random functionals together with probabilities of
conjunctive and/or disjunctive random functional
inequalities accommodates broader applications in operations
research and statistics such as piecewise statistical estimation models
\cite{CuiPang2020,LiuPang20} and in optimal control such as optimal
path planning models \cite{BlackmoreOnoWilliams11,LopezLudivigetal20}.

\gap

When nonconvexity is present, one needs to be mindful of the fact that globally
(or even locally) optimal solutions can rarely be provably computed; thus for
practical reasons, it is of paramount
importance to design computationally solvable optimization subproblems and to study
the computable solutions (instead of minimizers that cannot be computed).
In the case of the CCP with nonconvex and nondifferentiabilty features in the
constraints, this computational issue becomes more pronounced
and challenging.   With this in mind, convex programming based sampling methods
are desirable for the former task and stationary solutions for the latter.

\gap

  {The locally Lipschitz continuity of the probability distribution function
is an important requirement for the applicability of Clarke's nonsmooth analysis \cite{Clarke83}.
There are a few results about this property; for instance, in \cite[Section~2.6]{vanAckooij13},
the random function ${\cal Z}_{\ell}(x,z)$ is separable in
its arguments and additional conditions on the vector random variable $\tilde{z}$ are
in place; the paper \cite{Hantoute19} analyzed in detail the subdifferentiability (including
the locally Lipschitzian property) of the probability function in Banach space under Gaussian distribution.
Even if the Clarke subdifferential of the probability function is well defined, its calculation is usually a nontrivial
task except in special cases; thus hindering its practical use.  In the event
when the probability function fails to be locally Lipschitz continuous, the smoothing-based stochastic
approximation methods as in \cite{KannanLuedtke18} are not applicable.}
Therefore, instead of a stochastic (sub)gradient-type method,
we consider two sampling schemes.
One is the external sampling, or SAA \cite{Shapiro03}, where samples of a fixed
(presumably large) size are generated to define
an empirical optimization problem. The major focus of the external sampling is the
statistical analysis of the solution(s) to the empirical optimization
problem; such an analysis aims to establish asymptotic properties of the SAA solution(s)
in relation to the given expectation problem
when the sample size tends to infinity.  While computability remains a main concern,
the actual computation of the solution is not
for the external sampling scheme.  In contrast, in an internal, or sequential sampling
method \cite{ZiegelGhoshSen91,HigleSen91, BayraksanMorton11,LiuCuiPang20}, samples are
gradually accumulated as the iteration proceeds
in order to potentially improve the approximation of the expectation operator.
By taking advantage of the early stage of the algorithm,
the computational cost of subsequent iterations can be reduced.  Thus practical
computation is an important concern in an internal sampling method.

\gap

In order to deal with the expectation constraints and their approximations,
we embed the exact penalty approach into the two sampling schemes.
Different from the majority of the literature of the exact penalty theory on the
asymptotic analysis of the globally optimal solutions
whose computation is practically elusive,
we focus on the asymptotic  behavior of the stationary solutions that are computable
by a convex programming based surrogation method.
Thus, besides the modeling extensions and the synthesis of various computational
schemes, our main contributions pertaining to the sampling methods are twofold:

\gap

\noindent $\bullet $  The SAA scheme: for the \textcolor{black}{stochastic program} (SP)
with expectation constraints, we establish the almost sure convergence of the Clarke
stationary solutions of penalized SAA subproblems to a Clarke stationary solution of
the expectation constrained SP problem when the sample size increases to infinity while
the penalty parameter remains finite.  Furthermore, we establish that the directional
stationary points of the SAA problems are local minima when the random functionals have
a ``convex-like'' property.

\gap

\noindent $\bullet $  The sequential sampling scheme: we propose a one-loop algorithm
that allows for the simultaneous variations of the penalty parameters, either fixed or
diminishing approximation accuracy of the Heaviside functions, with the suitable choice
of an incremental sequence of sample sizes.  This is in contrast to the recent work
\cite{vanAckooijetal20} on solving the nonconvex and nonsmooth CCPs under the fixed
sample size and the fixed approximation accuracy, where the convergence of the bundle
method is derived for the approximation problem of the CCP; this framework is more
restrictive than ours.

\gap

The rest of the paper is organized as follows.  Section~\ref{sec:piecewise probabilistic
constraints} presents the structural assumptions of the bivariate function
${\cal Z}_{\ell}$ and illustrates the sources of the
nonsmoothness and nonconvexity by several examples.  In Section~\ref{sec:ACCSP},
we provide the approximations of the Heaviside functions
composite with the nonconvex random functionals and summarize their properties.
Section~\ref{sec:parameterized exp SP} is devoted to the study of the stationary
solutions of the approximated CCPs and their relationship
with the local minima.  In Section~\ref{sec:exact penalization}, we establish the
uniform exact penalty theory for the external sampling scheme of the CCPs
in terms of the Clarke stationary solutions.  Following that we discuss the internal
sampling scheme under both fixed and diminishing parametric
approximations of the Heaviside functions in Section~\ref{sec:MM}.  The paper ends with
a concluding section.  Two appendices
provide details of some omitted derivations in the main text.

\section{Sources of Nonsmoothness and Nonconvexity of the CCP}
\label{sec:piecewise probabilistic constraints}

In this section, we present the structural assumptions of the CCP and provide the
sources of nonsmoothness and nonconvexity.
Let $X$ be a closed convex set in $\mathbb{R}^n$ and
$c_0 : \mathbb{R}^{n+d} \to \mathbb{R}$
 be a given bivariate Carath\'eodory function;
i.e., $c_0(x,\bullet)$ is a measurable function for all $x \in X$ and $c_0(\bullet,z)$
is continuous on $X$ for all $z \in \Xi$ with more
properties on the latter function to be assumed subsequently.
We consider the stochastic program (\ref{eq:focus sp_dc_constraint}) with affine
chance constraints (ACCs) at levels $\{\zeta_k\}_{k \in [K]} $ with
$\zeta_k \in \mathbb{R}$ for $k \in [ K ]$.
The following blanket assumption is made throughout the paper:

\gap

%
%

{($\boldsymbol{\cal Z}$)} for $\ell = 1, \cdots, L$, the bivariate function
${\cal Z}_{\ell} : {\cal O} \times \Xi \to \mathbb{R}$ is a specially structured,
nondifferentiable, difference-of-convex (dc), function given by:
for some positive integers $I_{\ell}$ and $J_{\ell}$,
\begin{equation} \label{eq:g and h}
{\cal Z}_\ell(x,z) \, \triangleq \, \underbrace{\displaystyle{
\max_{1 \leq i \leq I_{\ell}}
} \, g_{i\ell}(x,z)}_{\mbox{denoted $g_{\ell}(x,z)$}} - \underbrace{\displaystyle{
\max_{1 \leq j \leq J_{\ell}}
} \, h_{j\ell}(x,z)}_{\mbox{denoted $h_{\ell}(x,z)$}},
\end{equation}
where each $g_{i\ell} : {\cal O} \times \Xi  \to \mathbb{R}$ and
$h_{j\ell} : {\cal O} \times \Xi  \to \mathbb{R}$ are such that

\gap

--  the functions $g_{i\ell}(\bullet,z)$ and $h_{j\ell}(\bullet,z)$ are convex,
differentiable, and Lipschitz continuous with
constant $\mbox{Lip}_{\rm c}(z) > 0$ satisfying $\displaystyle{
\sup_{z \in \Xi}
} \, \mbox{ Lip}_{\rm c}(z) \, < \, \infty$, and
$g_{i\ell}(x,\bullet)$ and $h_{j\ell}(x,\bullet)$ are measurable with
\[
\displaystyle{
\max_{1 \leq i \leq I_{\ell}}
} \, \mathbb{E}\left[ \, \left| \, g_{i\ell}(x,\tilde{z}) \, \right| \, \right]
\, < \, \infty \ \mbox{ and } \
\displaystyle{
\max_{1 \leq j \leq J_{\ell}}
} \, \mathbb{E}\left[ \, \left| \, h_{j\ell}(x,\tilde{z}) \, \right| \, \right]
\, < \, \infty, \epc \forall \, x \, \in \, X,
\]
where $\mathbb{E}$ is the expectation operator;
in particular, $g_{i\ell}$ and $h_{j\ell}$ are Carath\'eodory functions.

\gap

Thus, the gradients $\nabla g_{i\ell}(\bullet,z)$ and
$\nabla h_{j\ell}(\bullet,z)$ are   {globally} bounded on $X$
uniformly in $z \in \Xi$; that is,
\[ \displaystyle{
\sup_{(x,z) \, \in \, X \times \Xi}
} \quad \max\left\{ \, \displaystyle{
\max_{1 \leq i \leq I_{\ell}}
} \, \| \, \nabla_x g_{i\ell}(x,z) \, \|, \, \displaystyle{
\max_{1 \leq j \leq J_{\ell}}
} \, \| \, \nabla_x h_{j\ell}(x,z) \, \| \, \right\} \, < \, \infty.
\]
\textcolor{black}{Throughout the paper, the two pointwise maxima in (\ref{eq:g and h})
are treated as stated without smoothing.}
In the following, we explain the role of each component in the constraints of
\eqref{eq:focus sp_dc_constraint} with some examples.

\gap

\underline{\bf $e_{k\ell}$: affine combinations of probabilities}.
Mixed-signed affine combinations of probabilities are useful for the modeling of linear
relations among probabilities.
For example, given a random variable $Z$, a simple inequality like
$\mathbb{P}( f_1(x,Z) \geq 0 ) \leq \mathbb{P}( f_2(x,Z) \geq 0 )$ stipulates that
the probability of the event $f_1(x,Z) \geq 0$ does not exceed that of
the event $f_2(x,Z) \geq 0$.  Another example of an affine combination of probability
functions is the discrete relaxation of the first-order stochastic
dominance constraint \cite[Chapter 4]{ShapiroDentchevaRuszczynski09} in the form of
$\mathbb{P}(f_1(x,Z)\leq \eta_1) \leq \mathbb{P}(f_2(x,Z)\leq \eta_2)$
at given levels $\eta_1$ and $\eta_2$.  A third example of a negative coefficient
$e_{k\ell}$ is derived from the formula
$\mathbb{P}( A \setminus B ) = \mathbb{P}(A) - \mathbb{P}(A \cap B)$ to model
the probability of event $A$ and the negation of event $B$.
To illustrate: suppose that $A$ is the event $f_1(x,Z) \geq 0$ and $B$ is the event
$f_2(x,Z) \geq 0$.  Then $A \setminus B$ is the event
that $f_1(x,Z) \geq 0$ and $f_2(x,Z) < 0$.  Using the formula for the probability of
the latter joint event, we obtain
\[
\mathbb{P}( f_1(x,Z) \geq 0 \mbox{ and } f_2(x,Z) < 0 ) \, = \,
\mathbb{P}( f_1(x,Z) \geq 0 ) - \mathbb{P}( g(x,Z) \geq 0 )
\]
where $g(x,Z) \, \triangleq \, \min( f_1(x,Z),f_2(x,Z) )$.  Lastly, a conditional probability constraint
also leads to an affine combination of probabilities.  For example,
\[ \begin{array}{l}
\mathbb{P}\left( \, f_1(x,Z) \geq 0 \ | \ f_2(x,Z) \geq 0 \, \right) \, \leq \, b \\ [0.1in]
\epc \Longleftrightarrow \ \mathbb{P}\left[ \, \min\left( \, f_1(x,Z), \ f_2(x,Z) \, \right) \geq 0 \, \right] -
b \, \mathbb{P}( f_2(x,Z) \geq 0 ) \, \leq \, 0
\end{array}
\]
\underline{\bf ${\cal Z}_\ell$: \textcolor{black}{conjunctive} and disjunctive
combinations of random inequalities}.  It is clear that the probability of
joint random inequalities $\mathbb{P}(\,f_i(x,Z) \geq 0, \; i\in [I]\,)$ is equal to
$\mathbb{P}\left(\displaystyle{
\min_{1\leq i\leq I}
} \, f_i(x,Z) \geq 0\right)$.  Similarly, one can reformulate the
probability of disjunctive functional inequalities using the pointwise max operator.
\textcolor{black}{Most generally, combinations
$\displaystyle{
\bigwedge_{p=1}^{P_{\ell}}
} \, \left[ \, \wh{\cal Z}_{\ell p}(x,\tilde{z}) \leq 0 \, \right]$ \mbox{ and/or }
$\displaystyle{
\bigvee_{q=1}^{Q_{\ell}}
} \, \left[ \, \wt{\cal Z}_{\ell q}(x,\tilde{z}) \leq 0 \, \right]$ for arbitrary nonnegative integers
$P_{\ell}$ and $Q_{\ell}$ can be modelled by pointwise min/max functions to define ${\cal Z}_{\ell}(x,\tilde{z})$.}
As a simple example, let
$f_{1} : \mathbb{R}^{n+d} \to \mathbb{R}$ and $f_{2} : \mathbb{R}^{n+d} \to \mathbb{R}$ and
scalars $\{a_i\}_{i=1, 2}$ and $\{b_i\}_{i=1, 2}$  satisfying
$a_1 < b_1$ and $a_2 < b_2$ be given.  Then,
 \begin{equation*}
 \begin{array}{ll}
& \mathbb{P}\left( a_1 \leq f_{1}(x,Z) \leq b_1 \ \mbox{ or } \ a_2 \leq f_{2}(x,Z) \leq b_2 \, \right)\, \leq \, \zeta, \\[5pt]
\Longleftrightarrow
& \mathbb{P}\left(   \max\left\{  \min\{ \, b_1 - f_{1}(x,Z) , \, f_{1}(x,Z)  - a_1 \}, \, \min\{ \, b_2 - f_{2}(x,Z) , \, f_{2}(x,Z)  - a_2 \} \, \right\}
\, \geq \, 0 \, \right)
\, \leq \, \zeta.
\end{array}
\end{equation*}  
Notice that the composition of maximum and minimum of the above kind is a piecewise
linear function and can be written as the difference of two pointwise maxima as follows:
\[ \begin{array}{l}
\max\left\{ \, \min\{ \, b_1-t_1,t_1-a_1 \}, \, \min\{ \, b_2 -t_2,t_2-a_2 \} \, \right\}
 \\ [5pt]
\epc = \,  \max\left\{   \, -b_1 +t_1, \, a_1-t_1, \, -b_2+t_2, \, a_2-t_2 \, \right\}
\\[0.1in]
\epc - \, \max\left\{ \, -b_1+t_1-b_2 +t_2, \, -b_1+t_1-t_2+a_2, \,
-t_1+a_1-b_2+t_2, \,-t_1+a_1-t_2+a_2\, \right\}.
\end{array}
\]
When each $f_i(\bullet,z)$ is of the kind (\ref{eq:g and h}), then so is the above
difference of two pointwise maxima.  More generally, the following result provides the
basis to obtain the difference-of-convex representation of a piecewise affine function
composite with a function that is
the difference of two convex functions each being the pointwise maximum of finitely
many convex differentiable functions.
\textcolor{black}{While the difference-of-convexity property of such composite functions
is addressed by the so-called mixture property in the dc literature
(see e.g., \cite{BBorwein11,Hartman59}), the result shows how the explicit dc
representation is defined in terms of the element functions.}

\begin{lemma} \label{lm:PA of PA} \rm
Let each $\psi_{\ell}(x) = g_{\ell}(x) - h_{\ell}(x)$ with
$g_{\ell}, h_{\ell} : \mathbb{R}^n \to \mathbb{R}$ being convex differentiable
functions for $\ell = 1, \cdots, L$.  Let
$\varphi : \mathbb{R}^L \to \mathbb{R}$ be a piecewise affine function written as:
\[
\varphi(y) \, = \, \displaystyle{
\max_{1 \leq i \leq I}
} \, \left( \, y^{\top}a^i + \alpha_i \, \right) - \displaystyle{
\max_{1 \leq j \leq J}
} \, \left( \, y^{\top}b^{\, j} + \beta_j \, \right), \epc y \, \in \, \mathbb{R}^L,
\]
for some $L$-vectors $\{ a^i \}_{i=1}^I$ and $\{ b^{\, j} \}_{j=1}^J$ and scalars
$\{ \alpha_i\}_{i=1}^I$ and $\{ \beta_j\}_{j=1}^J$.  With
$\Psi(x) \triangleq \left( \psi_{\ell}(x) \right)_{\ell=1}^L$, the composite function
$\varphi \circ \Psi$ can be written as
\[
\varphi \circ \Psi(x) \, = \, \displaystyle{
\max_{1 \leq i \leq \wh{I}}
} \, \wh{g}_i(x) - \displaystyle{
\max_{1 \leq j \leq \wh{J}}
} \, \wh{h}_j(x)
\]
for some positive integers $\wh{I}$ and $\wh{J}$ and convex differentiable functions
$\wh{g}_i$ and $\wh{h}_j$.  A similar expression can
be derived when $g_{\ell}$ and $h_{\ell}$ are each the pointwise maximum of finitely
many convex differentiable functions.
\end{lemma}

\begin{proof} Write $a^i_{\ell} = a^i_{\ell+} - a^i_{\ell-}$ where $a^i_{\ell\pm}$
are the nonnegative $(+)$ and nonpositive $(-)$ parts of $a^i_{\ell}$, we have
\[ \begin{array}{rll}
\displaystyle{
\max_{1 \leq i \leq I}
} \, \left\{ \, \displaystyle{
\sum_{\ell=1}^L
} \, a^i_{\ell} \, \left[ g_{\ell}(x) - h_{\ell}(x) \, \right] + \alpha_i \, \right\}
& = & \displaystyle{
\max_{1 \leq i \leq I}
} \, \left\{ \, \varphi_{1+}^{\, i}(x) - \varphi_{1-}^{\, i}(x) + \alpha_i \, \right\}
\\ [0.1in]
\mbox{where} \epc \varphi_{1\pm}^{\, i}(x) & = & \displaystyle{
\sum_{\ell=1}^L
} \, \left[ \, a^i_{\ell\pm} \, g_{\ell}(x) + a^i_{\ell\mp} \, h_{\ell}(x) \, \right],
\end{array}
\]
are both convex and differentiable.  Thus,
\[
\varphi \circ \Psi(x) \, = \, \displaystyle{
\max_{1 \leq i \leq I}
} \, \left( \, \varphi_{1+}^{\, i}(x) - \varphi_{1-}^{\, i}(x) + \alpha_i \, \right) -  \displaystyle{
\max_{1 \leq j \leq J}
} \, \left( \, \varphi_{2+}^{\, j}(x) - \varphi_{2-}^{\, j}(x) + \beta_j \, \right)
\]
for some similarly defined convex and differentiable functions $\varphi_{2\pm}^j$.
Finally, one more manipulation yields
{\small
\[ 
\varphi \circ \Psi(x) = \displaystyle{
\max_{1 \leq i \leq I}
} \, \underbrace{\left( \, \wh{\varphi}_{1+}^{\, i}(x) + \displaystyle{
\sum_{i^{\, \prime} \neq i}
} \, \varphi_{1-}^{\, i^{\, \prime}}(x) + \displaystyle{
\sum_{j=1}^J
} \, \varphi_{2-}^{\, j}(x) \, \right)}_{\mbox{convex and differentiable in $x$}} -
\displaystyle{
\max_{1 \leq j \leq J}
} \, \underbrace{\left( \, \wh{\varphi}_{2+}^{\, j}(x) + \displaystyle{
\sum_{j^{\, \prime} \neq j}
} \, \varphi_{2-}^{\, j^{\, \prime}}(x) + \displaystyle{
\sum_{i=1}^I
} \, \varphi_{1-}^{\, i}(x) \, \right)}_{\mbox{convex and differentiable in $x$}},
\]
}

\noindent where $\wh{\varphi}_{1+}^{\, i}(x) \triangleq \varphi_{1+}^{\, i}(x)
+ \alpha_i$ and
$\wh{\varphi}_{2+}^{\, j}(x) \triangleq \varphi_{2+}^{\, j}(x) + \beta_j$.
Thus the claimed representation of $\varphi \circ \Psi$ follows.  We omit the proof
of the last statement of the lemma; see Appendix~1.
\end{proof}

Consequently, the structure (\ref{eq:g and h}) of ${\cal Z}_{\ell}(\bullet,z)$ allows
us to model the probability of disjunctive and conjunctive
inequalities of random functionals.  Probabilities of conjunctive functional
inequalities are fairly common in the literature on chance constraints and their
treatment using the min function (in our setting) is standard;
see e.g.\ \cite{hong2011sequential,PAhmedShapiro09,POrdieresLuedtkeWachter20}.
Nevertheless, it appears that the corresponding literature about
probabilities of disjunctive inequalities is scarce; applications of the latter
probabilities can be found in optimal path planning to avoid obstacles
in robot control \cite{BlackmoreOnoWilliams11,LopezLudivigetal20}.  The  latter
references treat the resulting probability constraint by utilizing the bound
$\mathbb{P}( \mbox{A or B} ) \leq \mathbb{P}(A) + \mathbb{P}(B)$ which provides a
very loose approximation of the resulting chance constraint.  Thus one contribution
of our work is to give a tighter treatment of chance constraints in the presence of
conjunctive and disjunctive functional events by modeling them faithfully within
the probability operator.

\gap

In addition, the piecewise dc structure (\ref{eq:g and h})
is central to piecewise statistical models, e.g., in
piecewise affine regression \cite{CuiPangSen18} and
deep neural networks with piecewise affine activation functions \cite{CuiHePang20}.
Nonsmooth structures such as these, not only cover more general applications,
but also provide computational tractability in terms of directional stationary points.
See \cite{CuiPangSen18,PangRazaAlvarado16,LuZhouSun19} for algorithms to
solve deterministic (composite) optimization problems involving such functions.
Furthermore, adding to the large body of literature, the paper
\cite{nouiehed2017pervasiveness}
has highlighted the fundamental role of the class of difference-of-convex functions
in statistics and optimization.

\gap

In summary, the class of nonconvex and nondifferentiable random functions
${\cal Z}_{\ell}(\bullet,z)$ given by (\ref{eq:g and h})
arises in different ways in statistical modeling and optimization under
uncertainty.  Their composition with the discontinuous Heaviside functions within the
expectation operator makes the exact evaluation of multi-dimensional integrations
impossible. Hence, the variational analysis and numerical computation of the overall
CCP in \eqref{eq:focus sp_dc_constraint} are much more involved than a linear or convex
random functional that is usually considered in the existing literature, thus
necessitating an in-depth treatment that goes beyond a smooth convex programming
approach.

\section{Computable Approximations of the CCP}\label{sec:ACCSP}

In order to design implementable and scalable computational methods to solve the CCP
in \eqref{eq:focus sp_dc_constraint}, we consider a family of
computationally tractable continuous approximations of the indicator functions of the
difference-of-convex type.  We denote the feasible set
of \eqref{eq:focus sp_dc_constraint} as
\begin{equation}\label{cc ori}
{X}_{\rm cc} \triangleq  \Big\{ \, x \, \in \, X \ \big| \ \displaystyle{
\sum_{\ell=1}^L
} \, e_{k\ell} \, \mathbb{P}\left( \, {\cal Z}_{\ell}(x,\tilde{z} ) \, \geq \, 0 \, \right)  - \zeta_k \, \leq \, 0, \epc k=1, \cdots, K\,\Big\}.
\end{equation}
One major difficulty of the above constraint is that for each $k\in [K]$, the
constraint function
\[
x \mapsto \displaystyle{
\sum_{\ell=1}^L
} \, e_{k\ell} \, \mathbb{P}\left( \, {\cal Z}_{\ell}(x,\tilde{z} )
\, \geq \, 0 \, \right) \, = \, \displaystyle{
\sum_{\ell=1}^L
}  e_{k\ell} \, \mathbb{E} \left[\,{\bf 1}_{[\,0,\infty)}
\left( \, {\cal Z}_{\ell}(x,\tilde{z} ) \, \right)\,\right]
\]
is not necessarily continuous.
A classical treatment of the continuity of probability functions can be found in
\cite{Raik71}; see \cite[Section 2.3]{vanAckooij13} for a more recent summary
of this continuity issue.
\textcolor{black}{Even if such constraint functions are Lipschitz continuous
(see \cite[Section~2.6]{vanAckooij13} for some conditions),
their generalized subdifferentials are impossible to evaluate but their
elements can be useful as conceptual targets for computation.}  Our treatment of
the feasible set $X_{\rm cc}$ begins with approximations of the Heaviside functions.

\subsection{Approximations of the discontinuous indicator functions}
\label{subsec:approx Heaviside}

Notice that the  function ${\bf 1}_{(\,0, \,\infty)}(\bullet)$ within the expectation
function in \eqref{eq:probability and expectation}
is lower semicontinuous while the function ${\bf 1}_{[\,0, \,\infty)}(\bullet)$ is
upper semicontinuous.  In general, there are three steps in obtaining an approximation
of these Heaviside functions: i) approximate the indicator
functions; ii) parameterize the approximation; and iii) control the parameterization.
One way to control the parameterization is to rely on the perspective function and
minimize over the parameter, resulting in the conditional
value-at-risk approximation of the chance constraint \cite{nemirovski2007convex}.
For complex random functionals, one needs to be careful about the minimization of
this parameter over the positive reals and to
ensure that a zero value will not be encountered during the solution process.
An alternative way to control the parameter is either to take a diminishing sequence
of positive parameters and study the limiting process,
or to fix a sufficiently small parameter and study the problem with the fixed parameter.
We will study both cases in the subsequent sections.
As one can expect, the analysis of the former is more challenging.

\gap

In what follows, we introduce the unified nonconvex relaxation and restriction of
the general affine chance constraints in  (\ref{eq:focus sp_dc_constraint}) where
the coefficients $\{e_{k\ell}\}$ have mixed signs.  Specifically, we employ

\gap

$(\boldsymbol{\Theta})$  a convex (thus continuous) function
$\wh{\theta}_{\rm cvx} : \mathbb{R} \to
\mathbb{R}$ and a concave (thus continuous) function
$\wh{\theta}_{\rm cve} : \mathbb{R} \to \mathbb{R}$ satisfying
\[
\wh{\theta}_{\rm cvx}(0) \, = \, 0 \, = \, \wh{\theta}_{\rm cve}(0) \epc \mbox{and} \epc
\wh{\theta}_{\rm cvx}(1) \, = \, 1 \, = \, \wh{\theta}_{\rm cve}(1),
\]
and with both functions being increasing in the interval $[ 0,1 ]$ and nondecreasing outside.

\gap

Truncating these two functions to the range $[\, 0, \,1 \,]$, we obtain the upper and
lower bounds of the two indicator functions
$\onebld_{[ \, 0,\, \infty \, )}(t)$ and $\onebld_{( \, 0,\, \infty \, )}(t)$
as follows: for any $( t,\gamma ) \, \in \, \mathbb{R} \, \times \, \mathbb{R}_{++}$,
\begin{equation}
\label{eq:lower_upper_approx}
 \begin{array}{lll}
\phi_{\rm ub}(t,\gamma) & \triangleq & \min\left\{ \, \max\left( \,
\wh{\theta}_{\rm cvx}\left( 1 + \displaystyle{
\frac{t}{\gamma}
} \right), \, 0 \, \right), \, 1 \, \right\} \\ [0.3in]
& \geq & \onebld_{[ \, 0,\, \infty \, )}(t) \, \geq \, \onebld_{( \, 0,\, \infty \, )}(t)
\\ [0.1in]
& \geq & \max\left\{ \, \min\left( \, \wh{\theta}_{\rm cve}\left( \displaystyle{
\frac{t}{\gamma}
} \right), \, 1 \, \right), \ 0 \, \right\} \, \triangleq \, \phi_{\rm lb}(t,\gamma).
\end{array}
\end{equation}
One can easily verify that the functions $\phi_{\rm ub}(\bullet,\gamma)$ and
$\phi_{\rm lb}(\bullet,\gamma)$ are difference-of-convex functions.
When $\wh{\theta}_{\rm cvx}$ reduces to the identity function, we obtain
$\phi_{\rm ub}(t,\gamma) = \min\Big\{\max\Big(1+\displaystyle\frac{t}{\gamma}, \, 0
\Big), \, 1\Big\}$.
This function is used  as an approximation of the indicator function in
\cite{hong2014conditional,hong2011sequential}, in which
the authors made several restrictive assumptions in deriving their analytical results
and fixed the scalar $\gamma$ at a prescribed (small) value in their computations.
Compared with the conservative convex approximations in \cite{NemirovskiShapiro},
the difference-of-convex approximation can provide tighter bounds of the indicator
functions. 

\gap
Illustrated by Figure \ref{fig:out_inner} with $\gamma = 1$, the two bivariate
functions $\phi_{\rm ub}$ and $\phi_{\rm lb}$ have important properties that we
summarize in the result below; these include connections with the Heaviside functions.

\begin{figure}[h]
\centering
\fbox{
		\begin{minipage}{.3\textwidth}
			\centering
			\includegraphics[width = \textwidth]{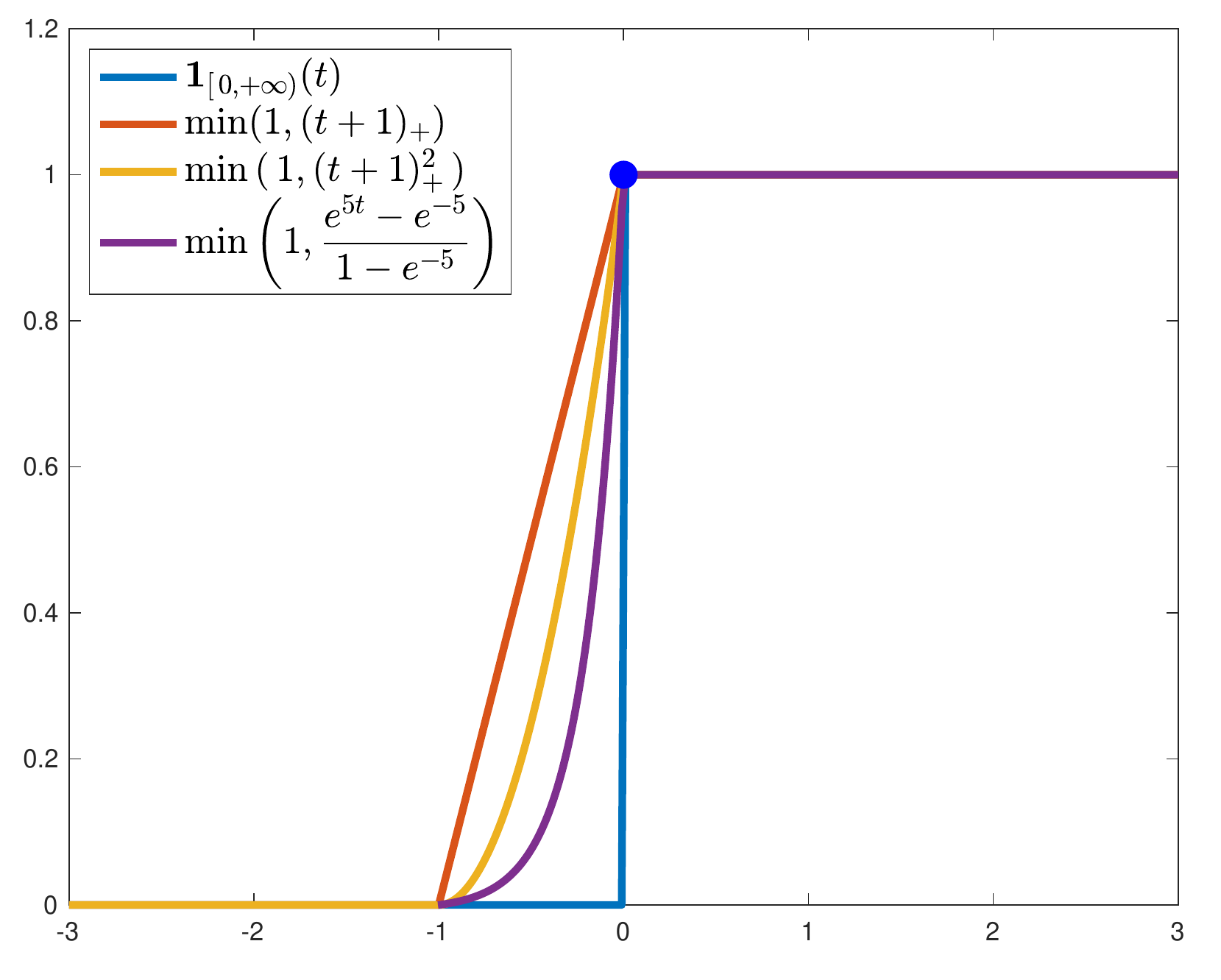}
			$\phi_{\rm ub}(t,1)$
		\end{minipage}
		\qquad
		\begin{minipage}{.3\textwidth}
			\centering
			\includegraphics[width = \textwidth]{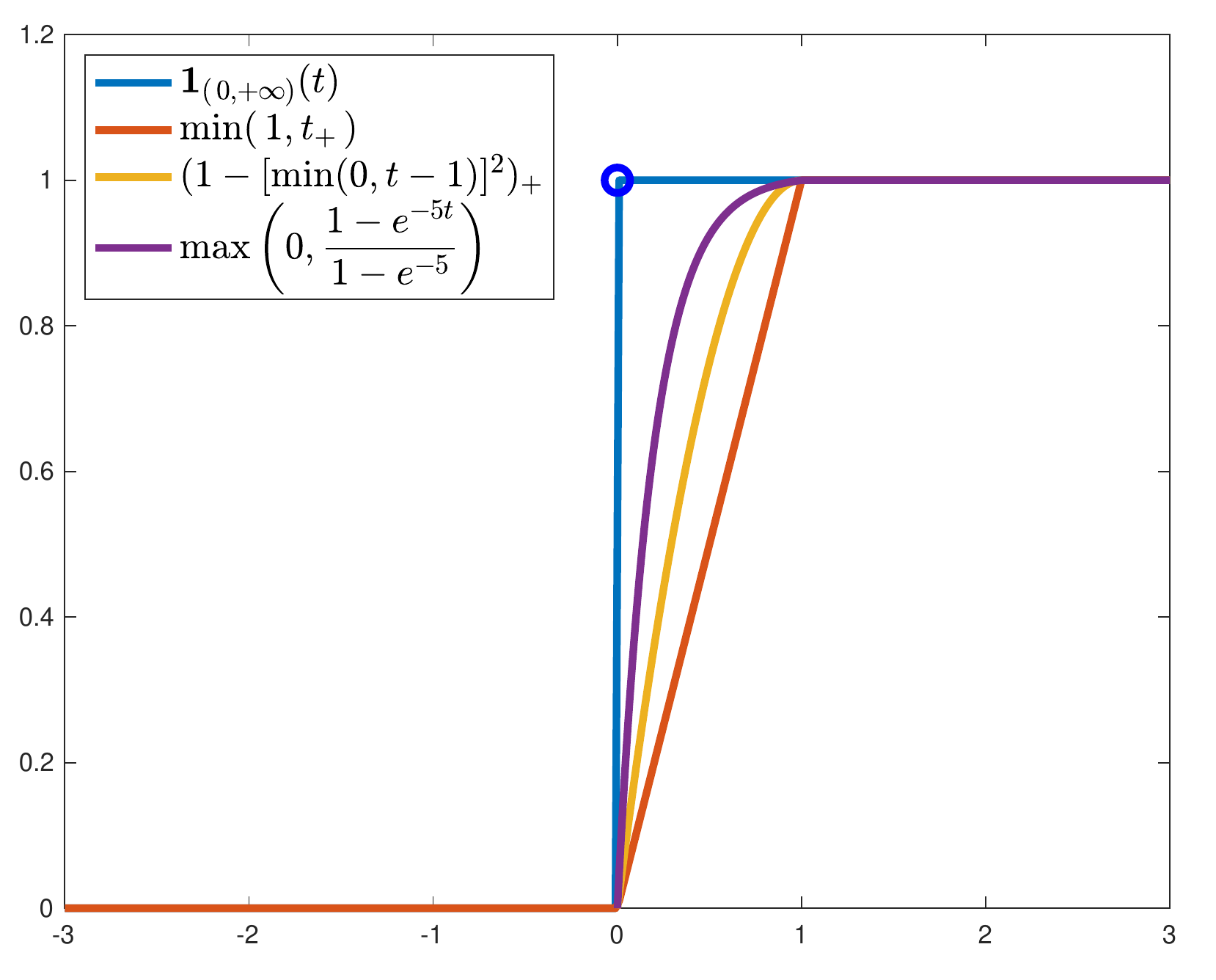}
			$\phi_{\rm lb}(t,1)$
		\end{minipage}
	}
\caption{Upper bound $\phi_{\rm ub}(t,\gamma)$ and lower bound
$\phi_{\rm lb}(t,\gamma)$ of ${\bf 1}_{[0, \infty)}$ and  ${\bf 1}_{(0, \infty)}$ with
$\gamma = 1$.}
\label{fig:out_inner}
\end{figure}	

\begin{proposition} \rm \label{pr:inner_outer_approximations}
The bivariate functions $\phi_{\rm ub}$ and $\phi_{\rm lb}$ defined above have the
following properties:

\gap

(a) For any $t\in \mathbb{R}$, $\phi_{\rm ub}(t,\gamma)$ is a nondecreasing function
in $\gamma$ on $\mathbb{R}_{++}$ and
$\phi_{\rm lb}(t,\gamma)$ is a nonincreasing function in $\gamma$ on $\mathbb{R}_{++}$.
Both functions $\phi_{\rm ub}$ and $\phi_{\rm lb}$
are Lipschitz continuous on every compact set
$T \times \Gamma \subseteq \mathbb{R} \times \mathbb{R}_{++}$.

\gap

(b) The following equalities hold:
\begin{equation} \label{eq:deterministic inf}
\begin{array}{rll}
\mathbf{1}_{[ \, 0, \infty \, )}(t) & = & \displaystyle{
\operatornamewithlimits{\mbox{\bf infimum}}_{\gamma > 0}
} \ \phi_{\rm ub}(t,\gamma) \, =   \, \displaystyle{
\operatornamewithlimits{\mbox{\bf limit}}_{\gamma \downarrow 0}
} \ \phi_{\rm ub}(t,\gamma), \epc \forall \, t \, \in \, \mathbb{R} \\ [0.15in]
\mbox{and} \ \mathbf{1}_{( \, 0, \infty \, )}(t) & = & \displaystyle{
	\operatornamewithlimits{\mbox{\bf supremum}}_{\gamma > 0}
} \ \phi_{\rm lb}(t,\gamma) \, =
\displaystyle{
\operatornamewithlimits{\mbox{\bf limit}}_{\gamma \downarrow 0}
} \ \phi_{\rm lb}(t,\gamma), \epc \forall \, t \, \in \, \mathbb{R}.
\end{array}
\end{equation}
\end{proposition}

\begin{proof}
(a) When $t \geq 0$, $\phi_{\rm ub}(t ,\gamma)=1$ for any $\gamma >0$.  When
$t \leq 0$,  $\phi_{\rm ub}(t,\bullet)$ is a nondecreasing function on $\mathbb{R}_{++}$.
Thus $\phi_{\rm ub}(t,\bullet)$ is a nondecreasing function on $\mathbb{R}_{++}$ for any
$t\in \mathbb{R}$.  Similarly, $\phi_{\rm lb}(t,\bullet)$ can be proved to
be a nonincreasing  function on $\mathbb{R}_{++}$ for any $t\in \mathbb{R}$.
To see the Lipschitz continuity of $\phi_{\rm ub}$ and $\phi_{\rm lb}$,  it suffices to
note that the bivariate function:
\[
( \, t,\gamma \, ) \, \mapsto \, \displaystyle{
\frac{t}{\gamma}
}, \epc \gamma \, > \, 0
\]
is Lipschitz continuous on any such Cartesian set $T \times \Gamma$.

\gap

(b) The two equalities in the upper-bound expression in (\ref{eq:deterministic inf})
clearly hold when $t \geq 0$ because all three quantities are equal to 1.
For $t < 0$, we have $1 + \displaystyle{
\frac{t}{\gamma}
} \, < \, 0$ for all $\gamma \in ( 0, -t )$; thus $\displaystyle{
\operatornamewithlimits{\mbox{\bf infimum}}_{\gamma > 0}
} \, \phi_{\rm ub}(t,\gamma) \, = \, \displaystyle{
\operatornamewithlimits{\mbox{\bf limit}}_{\gamma \downarrow 0}
} \, \phi_{\rm ub}(t,\gamma) \, = \, 0$.   Similarly, the two equalities in the
lower-bound expression clearly hold when $t \leq 0$ because
all three quantities are equal to 0.  For $t > 0$, since
$\phi_{\rm lb}(t,\gamma) = 1$ for all $\gamma \in ( 0,t ]$, the proof of
(\ref{eq:deterministic inf}) is complete.
\end{proof}

By defining $\phi_{\rm ub}(t,0) \triangleq \onebld_{[ 0,\infty )}(t)$ and
$\phi_{\rm lb}(t,0) \triangleq \onebld_{( 0,\infty )}(t)$,
Proposition~\ref{pr:inner_outer_approximations} allows us to extend the functions
$\phi_{\rm ub}(t,\gamma)$ and $\phi_{\rm lb}(t,\gamma)$
to $\gamma = 0$, making the former upper semicontinuous and the latter lower
semicontinuous on the closed domain {$\mathbb{R} \times \mathbb{R}_+$}.
This is formally stated and proved in the following result.

\begin{proposition} \label{pr:usc and lsc of phi} \rm
The following limiting inequalities hold for all pairs
$( t_*,\gamma_* ) \in \mathbb{R} \times \mathbb{R}_+$:
\[
\phi_{\rm ub}(t_*,\gamma_*) \, \geq \, \displaystyle{
\operatornamewithlimits{\mbox{\bf limsup}}_{( t,\gamma ) \to ( t_*,\gamma_* )}
} \ \phi_{\rm ub}(t,\gamma) \, \geq \, \displaystyle{
\operatornamewithlimits{\mbox{\bf liminf}}_{( t,\gamma ) \to ( t_*,\gamma_* )}
} \ \phi_{\rm lb}(t,\gamma) \, \geq \, \phi_{\rm lb}(t_*,\gamma_*).
\]
\end{proposition}

\begin{proof}  \textcolor{black}{With the definition of $\phi_{\rm ub}$ and
$\phi_{\rm lb}$ extended to the entire
domain $\mathbb{R} \times \mathbb{R}_+$ as described above}, the first inequality
clearly holds for $t_* \geq 0$ and all $\gamma_* \geq 0$ because
$\phi_{\rm ub}(t_*,\gamma_*) = 1$ for all such pairs $( t_*,\gamma_* )$;
\textcolor{black}{see the left curve in Figure~\ref{fig:out_inner}}.
Similarly the last inequality holds for $t_* \leq 0$ and all $\gamma_* \geq 0$  because
$\phi_{\rm lb}(t_*,\gamma_*) = 0$ for all such pairs $( t_*,\gamma_* )$;
\textcolor{black}{see the right curve in Figure~\ref{fig:out_inner}}.
Moreover, these two inequalities clearly hold for $\gamma_* > 0$ and all $t_*$
\textcolor{black}{because $\phi_{\rm ub}$ and
$\phi_{\rm lb}$ are both continuous
on $\mathbb{R} \times \mathbb{R}_{++}$}.  To complete the proof,
it \textcolor{black}{remains to consider $\gamma_* = 0$ and} show
\[
\left[ \, \displaystyle{
\operatornamewithlimits{\mbox{\bf lim}}_{( t,\gamma ) \to ( t_*,0 )}
} \ \phi_{\rm ub}(t,\gamma) \, = \, 0, \ \forall \, t_* \, < \, 0 \, \right]
\epc \mbox{and} \epc
\left[ \, \displaystyle{
\operatornamewithlimits{\mbox{\bf lim}}_{( t,\gamma ) \to ( t_*,0 )}
} \ \phi_{\rm lb}(t,\gamma) \, = \, 1, \ \forall \, t_* \, > \, 0 \, \right].
\]
The latter two limits are fairly obvious and no further proof is needed;
\textcolor{black}{indeed, it suffices to note that all $t$ near a nonzero $t_*$
must have the same sign as $t_*$.}
\end{proof}

The equalities in (\ref{eq:deterministic inf}) are deterministic results.
With $Z$ being a random variable, we have similar results in probability.
In particular, the proposition below shows that the gap between the limits of the outer
and inner approximations as $\gamma \downarrow 0$ is $\mathbb{P}(Z=0)$.

\begin{proposition} \label{pr:stochastic inf} \rm
For any real-valued random variable $Z$, it holds that
\[
\begin{array}{ll}
\mathbb{P}(Z \geq 0 ) \, = \, \displaystyle{
\operatornamewithlimits{\mbox{\bf infimum}}_{\gamma > 0}
} \, \mathbb{E}\left[ \, \phi_{\rm ub}(Z,\gamma) \, \right] \, = \,
\displaystyle{
\operatornamewithlimits{\mbox{\bf limit}}_{\gamma \downarrow 0}
} \, \mathbb{E}\left[ \, \phi_{\rm ub}(Z,\gamma) \, \right],\\[0.2in]
\mathbb{P}(Z > 0 ) \, = \, \displaystyle{
\operatornamewithlimits{\mbox{\bf supremum}}_{\gamma > 0}
} \, \mathbb{E}\left[ \, \phi_{\rm lb}(Z,\gamma) \, \right] \, = \,
\displaystyle{
\operatornamewithlimits{\mbox{\bf limit}}_{\gamma \downarrow 0}
} \, \mathbb{E}\left[ \, \phi_{\rm lb}(Z,\gamma) \, \right].
\end{array}
\]
\end{proposition}

\begin{proof}
From \eqref{eq:deterministic inf},
\[
\mathbb{P}(Z \geq 0) = \mathbb{E}\Big[ \, \mathbf{1}_{[0, \infty)}(Z) \, \Big] \, = \,
\mathbb{E}\Big[ \, \displaystyle{
\operatornamewithlimits{\mbox{\bf infimum}}_{\gamma > 0}
} \ \phi_{\rm ub}(Z,\gamma)  \, \Big] \, = \, \mathbb{E}\Big[ \, \displaystyle{
\operatornamewithlimits{\mbox{\bf limit}}_{\gamma \downarrow 0}
} \ \phi_{\rm ub}(Z,\gamma) \, \Big].
\]
Since $\phi_{\rm ub}(z,\bullet)$ is a monotonic function on $\mathbb{R}_{++}$,
by the Monotone Convergence Theorem, we have
\[
\mathbb{P}(Z \geq 0 ) \, = \, \displaystyle{
\operatornamewithlimits{\mbox{\bf infimum}}_{\gamma > 0}
} \, \mathbb{E}\left[ \, \phi_{\rm ub}(Z,\gamma) \, \right] \, = \,
\displaystyle{
\operatornamewithlimits{\mbox{\bf limit}}_{\gamma \downarrow 0}
} \, \mathbb{E}\left[ \, \phi_{\rm ub}(Z,\gamma) \, \right].
\]
The proof for the two equalities of $\mathbb{P}(Z > 0 )$ is similar and omitted.
\end{proof}

Note that for all $t$ in a compact interval of $\mathbb{R}$, the differences
$| \, \phi_{\rm ub/lb}(t,\gamma_1) - \phi_{\rm ub/lb}(t,\gamma_2) \, |$ are bounded by
a positive multiple of $\left| \, \displaystyle{
\frac{1}{\gamma_1}
} - \displaystyle{
\frac{1}{\gamma_2}
} \, \right|$ for all $\gamma_1 > \gamma_2 > 0$.
In the next result, we derive a similar bound on the expectation of the differences
$\mathbb{E}[ \, | \, \phi_{\rm ub/lb}(Z,\gamma_1) -
\phi_{\rm ub/lb}(Z,\gamma_2) \, | \, ]$ for a given random variable $Z$; the obtained bounds are
the basis for understanding the choice of the scaling
parameter in the convergence analysis of the algorithm for solving the ACC-SP (\ref{eq:focus sp_dc_constraint})
when $\gamma \downarrow 0$.   To derive these bounds, let $F_Z$ be the cumulative distribution
function (cdf) of $Z$, and for $\gamma > 0$,
\begin{equation} \label{eq:h of Z}
h_Z^{\rm lb}(\gamma) \, \triangleq \, 
\displaystyle{
\frac{1}{\gamma}
} \, \displaystyle{
\int_{\, 0}^{\, \gamma}
} \, F_Z(t) \, dt \epc \mbox{and} \epc
h_Z^{\rm ub}(\gamma) \, \triangleq \, 
\displaystyle{
\frac{1}{\gamma}
} \, \displaystyle{
\int_{\, -\gamma}^{\, 0}
} \, F_Z(t) \, dt. 
\end{equation}
These are nonnegative functions with $\displaystyle{
\lim_{\gamma \downarrow 0}
} \, h_Z^{\, \rm ub/lb}(\gamma) = F_Z(0)$; moreover,
$h_Z^{\rm ub/lb}$ are
nonincreasing/nondecreasing on $\mathbb{R}_{++}$, respectively.  Indeed, we have,
\[ \begin{array}{lll}
( \, h_Z^{\rm ub} )^{\, \prime}(\gamma) & = & -\displaystyle{
\frac{1}{\gamma^2}
} \, \displaystyle{
\int_{\, -\gamma}^{\, 0}
} \, F_Z(t) \, dt + \displaystyle{
\frac{1}{\gamma}
} \, F_Z(-\gamma) \\ [0.2in]
& \leq & -\displaystyle{
\frac{1}{\gamma^2}
} \, \displaystyle{
\int_{\, -\gamma}^{\, 0}
} \, F_Z(-\gamma) \, dt + \displaystyle{
\frac{1}{\gamma}
} \, F_Z(-\gamma) \epc \mbox{because $F_Z$ is nondecreasing} \\ [0.2in]
& = & -\displaystyle{
\frac{1}{\gamma}
} \, F_Z(-\gamma) + \displaystyle{
\frac{1}{\gamma}
} \, F_Z(-\gamma) \, = \, 0.
\end{array}
\]
In terms of the functions $h_Z^{\, \rm ub/lb}$, we have the following result.

\begin{proposition} \label{pr:error Heaviside and indicator} \rm
Let $\mbox{Lip}_{\theta}$ denote
the Lipschitz modulus of $\wh{\theta}_{\rm cvx/cve}$ on $[ 0,1 ]$.
For any random variable $Z$, it holds that for any two scalars
$\gamma_1 > \gamma_2 > 0$,
\begin{equation} \label{eq:error Heaviside and indicator}
\begin{array}{lll}	
0 \, \leq \,
\mathbb{E}\left[ \, \phi_{\rm ub}(Z,\gamma_1) - \phi_{\rm ub}(Z,\gamma_2) \, \right]
& \leq & \mbox{Lip}_{\theta} \, \left[ \, h_Z^{\rm ub}(\gamma_2) - h_Z^{\rm ub}(\gamma_1) \, \right] \\ [0.1in]
0 \, \leq \, \mathbb{E}\left[ \, \phi_{\rm lb}(Z,\gamma_2) - \phi_{\rm lb}(Z,\gamma_1)
\, \right] & \leq & \mbox{Lip}_{\theta} \, \left[ \, h_Z^{\rm lb}(\gamma_1) - h_Z^{\rm lb}(\gamma_2) \, \right].
\end{array}
\end{equation}
\end{proposition}

\begin{proof} We prove only the right-hand inequality in
(\ref{eq:error Heaviside and indicator}) for $\phi_{\rm ub}$.  We have
\[ \begin{array}{l}
\mathbb{E}\left[ \, \phi_{\rm ub}(Z,\gamma_1) - \phi_{\rm ub}(Z,\gamma_2) \, \right]
\\ [0.1in]
= \, \displaystyle{
\int_{\, -\gamma_1}^{\, -\gamma_2}
} \, \wh{\theta}_{\rm cvx}\left( 1 + \displaystyle{
\frac{t}{\gamma_1}
} \, \right) \, dF_Z(t) + \displaystyle{
\int_{\, -\gamma_2}^{\, 0}
} \, \left[ \, \wh{\theta}_{\rm cvx}\left( 1 + \displaystyle{
\frac{t}{\gamma_1}
} \, \right) - \wh{\theta}_{\rm cvx}\left( 1 + \displaystyle{
\frac{t}{\gamma_2}
} \, \right) \, \right] dF_Z(t) \\ [0.2in]
\leq \, \mbox{Lip}_{\theta} \, \left[ \, \displaystyle{
\int_{\, -\gamma_1}^{\, -\gamma_2}
} \, \left| \, 1 + \displaystyle{
\frac{t}{\gamma_1}
} \, \right| \, dF_Z(t) + \displaystyle{
\int_{\, -\gamma_2}^{\, 0}
} \, \left( \, \displaystyle{
\frac{1}{\gamma_2}
} - \displaystyle{
\frac{1}{\gamma_1}
} \, \right) \, | \, t \, | \, dF_Z(t) \, \right].
\end{array}
\]
Integration by parts yields
\[ \begin{array}{lll}
\displaystyle{
\int_{\, -\gamma_1}^{\, -\gamma_2}
} \, \left| \, 1 + \displaystyle{
\frac{t}{\gamma_1}
} \, \right| \, dF_Z(t) & = & \displaystyle{
\int_{\, -\gamma_1}^{\, -\gamma_2}
} \, \left( \, 1 + \displaystyle{
\frac{t}{\gamma_1}
} \, \right) \, dF_Z(t) \\ [0.2in]
& = & \left( \, 1 - \displaystyle{
\frac{\gamma_2}{\gamma_1}
} \, \right) \, F_Z(-\gamma_2) - \displaystyle{
\frac{1}{\gamma_1}
} \, \displaystyle{
\int_{\, -\gamma_1}^{\, -\gamma_2}
} \, F_Z(t) \, dt
\end{array} \]
and
\[
\displaystyle{
\int_{\, -\gamma_2}^{\, 0}
} \, \left( \, \displaystyle{
\frac{1}{\gamma_2}
} - \displaystyle{
\frac{1}{\gamma_1}
} \, \right) \, | \, t \, | \, dF_Z(t) \, = \, \left( \, \displaystyle{
\frac{1}{\gamma_2}
} - \displaystyle{
\frac{1}{\gamma_1}
} \, \right) \, \left[ \, -\gamma_2 \, F_z(-\gamma_2) + \displaystyle{
\int_{\, -\gamma_2}^{\, 0}
} \, F_Z(t) \, dt \, \right].
\]
Adding the two terms yields
\[ \begin{array}{lll}
\mathbb{E}\left[ \, \phi_{\rm ub}(Z,\gamma_1) - \phi_{\rm ub}(Z,\gamma_2) \, \right]
& \leq & \mbox{Lip}_{\theta} \, \left[ \,  - \displaystyle{
\frac{1}{\gamma_1}
} \, \displaystyle{
\int_{\, -\gamma_1}^{\, -\gamma_2}
} \, F_Z(t) \, dt  + \left( \, \displaystyle{
\frac{1}{\gamma_2}
} - \displaystyle{
\frac{1}{\gamma_1}
} \, \right) \, \displaystyle{
\int_{\, -\gamma_2}^{\, 0}
} \, F_Z(t) \, dt \, \right] \\ [0.2in]
& = & \mbox{Lip}_{\theta} \, \left[ \, \displaystyle{
\frac{1}{\gamma_2}
} \, \displaystyle{
\int_{\, -\gamma_2}^{\, 0}
} \, F_Z(t) \, dt - \displaystyle{
\frac{1}{\gamma_1}
} \, \displaystyle{
\int_{\, -\gamma_1}^{\, 0}
} \, F_Z(t) \, dt \, \right],
\end{array} \]
which is the desired bound.
\end{proof}

\subsection{Approximation of the chance-constrained set $X_{\rm cc}$}

In the following, we discuss the continuous approximation of the chance constraints
in \eqref{cc ori} via the
upper and lower approximations of the
Heaviside functions provided in the last subsection.
Recalling the signed decomposition
$e_{k\ell} = e_{k\ell}^{\,\,+} - e_{k\ell}^{\,\,-}$, we have
\[
\displaystyle{
\sum_{\ell=1}^L
} \, e_{k\ell} \, \mathbb{P}\left( \, {\cal Z}_{\ell}(x,\tilde{z} ) \, \geq \, 0 \,
\right) \, = \, \displaystyle{
\sum_{\ell=1}^L
} \, \left( \, e_{k\ell}^{\,\,+} - e_{k\ell}^{\,\,-} \, \right) \,
\mathbb{P}\left( \, {\cal Z}_{\ell}(x,\tilde{z} ) \, \geq \, 0 \, \right).
\]
To proceed, we denote, for any $x\in X$ and any $\gamma>0$,
\begin{equation}\label{defn: c_gamma}
\left\{\begin{array}{l}
\bar{c}_k^{\, \rm rlx}(x;\gamma) \, \triangleq \, \mathbb{E}\left[ \,
c_k^{\, \rm rlx}(x,\tilde z;\gamma) \,\triangleq \,\displaystyle{
\sum_{\ell=1}^L
} \, \left( \, \underbrace{e_{k\ell}^+ \,
\phi_{\rm lb}({\cal Z}_{\ell}(x,\tilde{z} ),\gamma) - e_{k\ell}^- \,
\phi_{\rm ub}({\cal Z}_{\ell}(x,\tilde{z} ),\gamma)}_{\mbox{denoted by
$c_{k\ell}^{\, \rm rlx}(x,\tilde z;\gamma)$}} \, \right) \, \right]
\\ [0.45in]
\bar{c}_k^{\, \rm rst}(x;\gamma) \, \triangleq \, \mathbb{E}\left[ \,
c_{k}^{\, \rm rst}(x,\tilde{z};\gamma) \,\triangleq \, \displaystyle{
\sum_{\ell=1}^L
} \, \left( \, \underbrace{e_{k\ell}^+ \,
\phi_{\rm ub}({\cal Z}_{\ell}(x,\tilde{z} ),\gamma)  -
e_{k\ell}^- \, \phi_{\rm lb}({\cal Z}_{\ell}(x,\tilde{z} ),\gamma)}_{\mbox{denoted by
$c_{k\ell}^{\, \rm rst}(x,\tilde{z};\gamma)$}} \, \right) \, \right]
\end{array}\right.
\end{equation}
and
\begin{equation}\label{defn:gamma constraints}
\left\{\begin{array}{ll}
\overline{X}_{\rm rlx}(\gamma) \, \triangleq \, \left\{ \, x \in X \ \big| \
\bar{c}_k^{\, \rm rlx}(x;\gamma) - \zeta_k \, \leq \, 0, \epc k \, \in \, [ \, K \, ]
\,\right\} \\ [0.15in]
\overline{X}_{\rm rst}(\gamma) \,  \triangleq \, \left\{ \, x \in X \ \big| \
\bar{c}_k^{\, \rm rst}(x;\gamma) - \zeta_k \, \leq \, 0, \epc k \, \in \, [ \, K \, ]
\,\right\}.
\end{array}\right.
\end{equation}
It then follows by Proposition \ref{pr:inner_outer_approximations} that for any
$\gamma > 0$,
\[
\bar{c}_k^{\rm \, rlx}(x;\gamma) \, \leq \, \displaystyle{
\sum_{\ell=1}^L
} \, e_{k\ell} \,\mathbb{P}\left( \, {\cal Z}_{\ell}(x,\tilde{z} ) \, \geq \, 0 \,
\right) \,\leq \,
\bar{c}_k^{\rm \, rst}(x;\gamma) \epc \mbox{and}\epc \overline{X}_{\rm rst}(\gamma)
\subseteq X_{\rm cc} \subseteq \overline{X}_{\rm rlx}(\gamma).
\]
The set inclusions show that for any $\gamma>0$, the set
$\overline{X}_{\rm rst}(\gamma)$ yields a more restrictive feasible region compared
with the set $X_{\rm cc}$ of the original chance constraints while
$\overline{X}_{\rm rlx}(\gamma)$ is a relaxation of the latter set.  This explains the
scripts ``rst'' and ``rlx'' in the above notations,
\textcolor{black}{which stand for ``restricted'' and ``relaxed'', respectively.}
With each ${\cal Z}_{\ell}$ given by assumption ($\boldsymbol{\cal Z}$), the sets
$\overline{X}_{\rm rst}(\gamma)$ and $\overline{X}_{\rm rlx}(\gamma)$
are closed.  However, with the definition of  the limit of set-valued mappings
in \cite[Chapters~4 and 5]{RockafellarWets98},
the limits of these two sets when $\gamma \downarrow 0$ may not   {be} equal to $X_{\rm cc}$
in general.  In order to derive their respective limits,  we further define
\begin{equation} \label{eq:prob of c_gamma}
\left\{ \begin{array}{l}
\bar{c}_k^{\, \rm rlx}(x) \,\triangleq \,
\displaystyle{
\sum_{\ell=1}^L
} \, \Big( \, e_{k \ell}^+ \, \mathbb{P}({\cal Z}_{\ell}(x,\tilde{z} )> 0) -
e_{k \ell}^- \, \mathbb{P}({\cal Z}_{\ell}(x,\tilde{z} )\geq 0) \, \Big) \\ [0.2in]
\bar{c}_k^{\, \rm rst}(x) \, \triangleq \,
\displaystyle{
\sum_{\ell=1}^L
} \, \Big( \, e_{k \ell}^+ \, \mathbb{P}({\cal Z}_{\ell}(x,\tilde{z} )\geq  0) -
e_{k \ell}^- \, \mathbb{P}({\cal Z}_{\ell}(x,\tilde{z}) > 0) \, \Big)
\end{array} \right.
\end{equation}
and
\begin{equation}\label{eq:prob of gamma constraints}
\left\{\begin{array}{ll}
\overline{X}_{\rm rlx} \, \triangleq \, \left\{ \, x \in X \ \big| \
\bar{c}_k^{\, \rm rlx}(x) - \zeta_k \, \leq \, 0, \epc k \, \in \, [ \, K \, ]
\, \right\} \\ [0.15in]
\overline{X}_{\rm rst} \,  \triangleq \, \left\{ \, x \in X \ \big| \
\bar{c}_k^{\, \rm rst}(x) - \zeta_k \, \leq \, 0, \epc k \, \in \, [ \, K \, ]
\, \right\}.
\end{array}\right.
\end{equation}
Based on Proposition~\ref{pr:error Heaviside and indicator}, we can given the following error of
the restricted/relaxed approximations of the affine constraint functions.

\begin{proposition} \label{pr:error of rst/rlx and acc} \rm
For any two scalars $\gamma_1 > \gamma_2 > 0$, it holds that for all $x \in X$,
\[ \begin{array}{lll}
\left| \, \bar{c}_k^{\, \rm rst/rlx}(x;\gamma_1) - \bar{c}_k^{\, \rm rst/rlx}(x;\gamma_2) \, \right|
\\ [0.1in]
\leq \,   {\mbox{Lip}_{\theta}} \, \displaystyle{
\sum_{\ell=1}^L
} \ | \, e_{k\ell} \, | \ \max\left( \, h_{{\cal Z}_{\ell}(x,\bullet)}^{\rm ub}(\gamma_2) - h_{{\cal Z}_{\ell}(x,\bullet)}^{\rm ub}(\gamma_1), \
\, h_{{\cal Z}_{\ell}(x,\bullet)}^{\rm lb}(\gamma_1) - h_{{\cal Z}_{\ell}(x,\bullet)}^{\rm lb}(\gamma_2) \, \right).
\end{array}
\]	
\end{proposition}

\begin{proof} We prove only the inequality for the restricted function.  But this is fairly easy because
\[
\bar{c}_k^{\, \rm rst}(x;\gamma_1) - \bar{c}_k^{\, \rm rst}(x;\gamma_2)
\, = \, \displaystyle{
\sum_{\ell=1}^L
} \, \mathbb{E}\left[ \, \begin{array}{l}
e_{k\ell}^+ \, \left[ \,
\phi_{\rm ub}({\cal Z}_{\ell}(x,\tilde{z} ),\gamma_1) -
\phi_{\rm ub}({\cal Z}_{\ell}(x,\tilde{z} ),\gamma_2) \, \right] \ - \\ [0.15in]
e_{k\ell}^- \,
\left[ \, \phi_{\rm lb}({\cal Z}_{\ell}(x,\tilde{z} ),\gamma_1)  -
\phi_{\rm lb}({\cal Z}_{\ell}(x,\tilde{z} ),\gamma_2) \, \right]
\end{array} \right];
\]
the desired inequality then follows readily from (\ref{eq:error Heaviside and indicator}).
\end{proof}

The proposition below summarizes several set-theoretic properties of
the two families of close\textcolor{black}{d} sets
$\{ \overline{X}_{\rm rlx}(\gamma) \}_{\gamma > 0}$ and
$\{ \overline{X}_{\rm rst}(\gamma) \}_{\gamma > 0}$.
The obtained result also provides a sufficient condition under which the limits
of these approximating sets coincide with the feasible set $X_{\rm cc}$ of the
ACC-SP.

\begin{proposition} \label{pr:relations of limiting sets} \rm
The following statements hold:

\gap

(i) The family $\left\{ \, \overline{X}_{\rm rlx}(\gamma) \, \right\}$
is nondecreasing in $\gamma > 0$; the family
$\left\{ \, \overline{X}_{\rm rst}(\gamma) \, \right\}$ is nonincreasing in $\gamma > 0$.

\gap

(ii) $\displaystyle{
\lim_{\gamma \downarrow 0}
} \ \overline{X}_{\rm rst}(\gamma) \, = \, \mbox{cl}\left( \,
\displaystyle{
\underset{\gamma > 0}{\bigcup}} \, \overline{X}_{\rm rst}(\gamma) \, \right)
\, \subseteq \, \mbox{cl}( \, \overline{X}_{\rm rst} \, )
\, \subseteq \, \mbox{cl}( \, \overline{X}_{\rm rlx} \, ) \, = \, \overline{X}_{\rm rlx}
\, = \, \displaystyle{
\underset{\gamma > 0}{\bigcap}} \, \overline{X}_{\rm rlx}(\gamma) \, = \, \displaystyle{
\lim_{\gamma \downarrow 0}
} \ \overline{X}_{\rm rlx}(\gamma)$.

\gap

(iii) If $\mbox{cl}( \, \overline{X}_{\rm rst} \, )
\, \subseteq \, \mbox{cl}\left\{ \, x \in X \ \left| \right. \
\bar{c}_k^{\rm rst}(x) < \zeta_k, \ \forall \, k \in [K] \, \right\}$, then
$\displaystyle{
\lim_{\gamma \downarrow 0}
} \ \overline
{X}_{\rm rst}(\gamma) = \mbox{cl}( \, \overline{X}_{\rm rst} \, )$.

\gap

(iv) If $\mathbb{P}({\cal Z}_{\ell}(x,\tilde{z} ) = 0) = 0$ for all
$\ell = 1, \cdots, L$ and all $x \in X$, then
$\overline{X}_{\rm rst} = X_{\rm cc} = \overline{X}_{\rm rlx}$.
If in addition the assumption in part (iii) holds, then all sets in part (ii) are equal.
\end{proposition}

\begin{proof}
Since $\phi_{\rm ub}(t, \gamma)$ is a nondecreasing function and
$\phi_{\rm lb}(t, \gamma)$ is a nonincreasing function in $\gamma$
for any $t\in \mathbb{R}$, statement (i) is obvious.  For statement (ii), the first and
last equalities follow from statement~(i) and \cite[Exercise~4.3]{RockafellarWets98};
\textcolor{black}{in particular, the set $\overline{X}_{\rm rlx}$ is closed because
$\bar{c}_k^{\, \rm rlx}(\bullet)$ is lower semicontinuous.}
For the other relations, it suffices to prove the inclusion
$\displaystyle{
\underset{\gamma > 0}{\bigcup}} \, \overline{X}_{\rm rst}(\gamma) \, \subseteq \,
\overline{X}_{\rm rst}$ and the second-to-last equality.
Let $x \in \displaystyle{
\underset{\gamma > 0}{\bigcup}} \, \overline{X}_{\rm rst}(\gamma)$   {be given}.  Then
$x \in \overline{X}_{\rm rst}(\gamma)$ for all $\gamma > 0$ sufficiently small
because the family $\{ \overline{X}_{\rm rst}(\gamma) \}$ is nonincreasing in $\gamma$.
Thus, for such $\gamma$, we have
\[
\displaystyle{
\sum_{\ell=1}^L
} \, \Big\{ \, e_{k\ell}^+ \, \mathbb{E}\Big[ \,
\phi_{\rm ub}({\cal Z}_{\ell}(x,\tilde{z} ), \gamma) \, \Big] -
e_{k\ell}^- \, \mathbb{E}\Big[ \, \phi_{\rm lb}({\cal Z}_{\ell}(x,\tilde{z} ), \gamma)
\, \Big] \, \Big\} \, \leq \, \zeta_k.
\]
By letting $\gamma \downarrow 0$ on both sides, with
Proposition~\ref{pr:stochastic inf}, we deduce
$\bar{c}_k^{\rm rst}(x)
\, \leq \, \zeta_k$.  Hence $\displaystyle{
\underset{\gamma > 0}{\bigcup}
} \, \overline{X}_{\rm rst}(\gamma) \subseteq \overline{X}_{\rm rst}$.
In a similar manner, we can prove $\displaystyle{
\underset{\gamma > 0}{\bigcap}
} \, \overline{X}_{\rm rlx}(\gamma) \subseteq \overline{X}_{\rm rlx}$.
\textcolor{black}{Indeed, let $x$ be an element in the left-hand intersection.
We then have, for all $\gamma > 0$.
\[
\bar{c}^{\, \rm rlx}_k(x;\gamma) \, = \, \mathbb{E}\left[ \, \displaystyle{
\sum_{\ell=1}^L
} \, \left( \, e_{k\ell}^+ \, \phi_{\rm lb}({\cal Z}_{\ell}(x,\tilde{z} ),\gamma) -
e_{k\ell}^- \, \phi_{\rm ub}({\cal Z}_{\ell}(x,\tilde{z} ),\gamma) \, \right) \, \right]
\, \leq \, \zeta_k.
\]
By letting $\gamma \downarrow 0$ on both sides, with Proposition~\ref{pr:stochastic inf}
we deduce $\bar{c}_k^{\, \rm rlx}(x)
\, \leq \, \zeta_k$.  Thus $x \in  \overline{X}_{\rm rlx}$, showing that $\displaystyle{
\underset{\gamma > 0}{\bigcap}
} \, \overline{X}_{\rm rlx}(\gamma) \subseteq \overline{X}_{\rm rlx}$.  Conversely, let
$x \in \overline{X}_{\rm rlx}$.}  Since
$\mathbb{P}( \mathcal{Z}_\ell(x,\tilde{z}) > 0 ) \geq
\mathbb{E}\left[ \, \phi_{\rm lb}({\cal Z}_{\ell}(x,\tilde{z} ), \gamma) \, \right]$ and
$\mathbb{P}( \mathcal{Z}_\ell(x,\tilde{z}) \geq 0) \leq \mathbb{E}\left[ \,
\phi_{\rm ub}({\cal Z}_{\ell}(x,\tilde{z} ), \gamma) \, \right] $ for any $\gamma >0$
by Proposition \ref{pr:stochastic inf}, it follows that $\overline{X}_{\rm rlx} \subseteq
\overline{X}_{\rm rlx}(\gamma)$ for any $\gamma >0$.  Hence,
$\overline{X}_{\rm rlx} = \displaystyle{
\underset{\gamma > 0}{\bigcap}} \, \overline{X}_{\rm rlx}(\gamma)$.   To prove (iii),
it suffices to note that
\[
\left\{ \, x \in X \ \left| \right. \ \bar{c}_k^{\, \rm rst}(x) < \zeta_k, \
\forall \, k \, \in \, [ K ] \, \right\}
\subseteq \, \displaystyle{
\bigcup_{\gamma > 0}
} \, \overline{X}_{\rm rst}(\gamma),
\]
taking closures on both sides and using the assumption easily establishes the equality
of the two sets $\displaystyle{
\lim_{\gamma \downarrow 0}
} \ \overline{X}_{\rm rst}(\gamma)$ and $\mbox{cl}( \, \overline{X}_{\rm rst} \, )$.
Finally, to prove (iv), note that
\[ \begin{array}{lll}
\overline{c}_k^{\, \rm rlx}(x) & = & \displaystyle{
\sum_{\ell=1}^L
} \, e_{k\ell} \, \mathbb{P}({\cal Z}_{\ell}(x,\tilde{z} )> 0) - \displaystyle{
\sum_{\ell=1}^L
} \, e_{k\ell}^- \, \mathbb{P}({\cal Z}_{\ell}(x,\tilde{z} ) = 0), \\ [0.2in]
\bar{c}_k^{\, \rm rst}(x) & = & \displaystyle{
\sum_{\ell=1}^L
} \, e_{k\ell} \, \mathbb{P}({\cal Z}_{\ell}(x,\tilde{z} )> 0) + \displaystyle{
\sum_{\ell=1}^L
} \, e_{k\ell}^+ \, \mathbb{P}({\cal Z}_{\ell}(x,\tilde{z} ) = 0).
\end{array} \]
Hence the equalities $\overline{X}_{\rm rst} = X_{\rm cc} = \overline{X}_{\rm rlx}$
follow readily
under the zero-probability assumption; and so does the last assertion in this part.
\end{proof}

Proposition~\ref{pr:relations of limiting sets} shares much resemblance with
\cite[Theorem 3.6]{GHKLi17}.  The only difference is that the cited theorem has a
blanket assumption (A0), which implies in particular the closedness of the feasible set
$X_{\rm cc}$.  We drop this assumption until the last part where we equate all the sets.
In the following, we provide an example showing that for a closed set $X_{\rm cc}$
(empty set included), strict inclusions between the three sets
$\overline{X}_{\rm rst}$, $X_{\rm cc}$ and $\overline{X}_{\rm rlx}$ are possible
if there exists $\ell \in [L]$ such that
$\mathbb{P}(\mathcal{Z}_\ell(x,\tilde{z} ) = 0 ) \neq 0$ for some $x$.

\begin{example} \label{ex:Bernoulli} \rm
Consider the set ${X}_{\rm cc} = \left\{ \, x \, \in \, \mathbb{R} \, : \, e \,
\mathbb{P}( x \, Z \, \geq \, 0 \, ) \, \leq \, \zeta \, \right\}$,
where $Z$ is a Bernoulli random variable such that
$\mathbb{P}(Z=1) = \mathbb{P}(Z=-1) = 1/2$.
Then with $e = 1$ and $\zeta = 0.1$, we have
$X_{\rm cc} =\overline{X}_{\rm rst} = \emptyset$ while
$\overline{X}_{\rm rlx} = \{ 0 \}$.  With $e = -1$ and $\zeta = -0.6$,
we have $X_{\rm cc} =\overline{X}_{\rm rlx}=\{ 0\}$
while $\overline{X}_{\rm rst}=\emptyset$. \hfill $\Box$
\end{example}

To end the section, it would be useful to summarize the
notations for the constraint functions used throughout the paper.  Absence of the
scalar $\gamma$, the notations for the objective function are similar.
\gap

\fbox{
\parbox{6.5in}{
\begin{center}
\textcolor{black}{\bf Notations for constraint functions} \\
(similar notations for the objective function)
\end{center}

{\bf I. Plain:} (3 arguments) for the defining functions of the problems

\gap

$\bullet $ $c_{k\ell}^{\, \rm rlx/rst}(x,z;\gamma)$ defined in (\ref{defn: c_gamma});

--- superscripts rlx/rst are omitted in general discussion;
e.g.\ $c_k(x,z;\gamma) \triangleq
\displaystyle{
\sum_{\ell=1}^L
} \, c_{k\ell}(x,z;\gamma)$ in Sections~\ref{sec:parameterized exp SP} and \ref{sec:MM};

--- the scalar $\gamma$ is fixed and thus omitted in Section~\ref{sec:exact penalization}

\gap

{\bf II. bar:} for expectation (2 arguments) and probability (1 argument)

\gap

$\bullet $ $\bar{c}_k^{\, \rm rlx/rst}(x;\gamma) \, \triangleq \,
\mathbb{E}\left[ \, \displaystyle{
\sum_{\ell=1}^L
} \, c_{k\ell}^{\, \rm rlx/rst}(x,\tilde{z};\gamma) \, \right]$ along
with the associated sets $\overline{X}_{\rm rlx/rst}(\gamma)$;

\gap

$\bullet $ $\bar{c}_{k\ell}^{\, \rm rlx/rst}(x)$ defined in (\ref{eq:prob of c_gamma})
along with the associated sets $\overline{X}_{\rm rlx/rst}$;

\gap

{\bf III. hat:} (4 arguments) for surrogation used in Section~\ref{sec:MM}

\gap

$\bullet $ $\wh{c}_k(\bullet,z;\gamma;\bar{x})$, derived from the surrogation
$\wh{c}_{k\ell}(\bullet,z;\gamma;\bar{x})$ of the summands
$c_{k\ell}(\bullet,z;\gamma)$ at $\bar{x}$;

--- superscripts rlx/rst used when referred to the relaxed/restricted problems;

\gap

{\bf IV. tilde:} (3 arguments) for limiting function in convergence analysis of
diminishing $\gamma_{\nu}$

\gap

$\bullet $ $\wt{c}_k(x,z;\bar{x})$ used in Subsection~\ref{subsec:diminishing approx}.
}}

\section{The Expectation Constrained SP} \label{sec:parameterized exp SP}

In this section, we start by considering the following abstract stochastic program
without referring to the detailed structure of the constraint functions:
for given positive integers $K$ and $L$ and a parameter $\gamma > 0$,
\begin{equation} \label{eq:new SP}
\begin{array}{ll}
\displaystyle{
\operatornamewithlimits{\mbox{\bf minimize}}_{x \in X}
} & \bar c_0(x) \, \triangleq \, \mathbb{E}\left[ \, c_0(x,\tilde{z} ) \, \right]
\\ [0.1in]
\mbox{\bf subject to} &\bar  c_k(x;\gamma) \, \triangleq \,
\mathbb{E}\left[ \, \underbrace{\displaystyle{
\sum_{\ell=1}^L
} \,  c_{k\ell}(x,\tilde{z} ;\gamma)}_{\mbox{denoted $c_k(x,\tilde{z};\gamma)$}} \, \right] \, \leq \, \zeta_k,
\epc k \, = \, 1, \cdots, K.
\end{array} \end{equation}
Subsequently, we will specialize the constraint functions to those in the sets
$\overline{X}_{\rm rlx}(\gamma)$ and $\overline{X}_{\rm rst}(\gamma)$
that are defined in \eqref{defn:gamma constraints} and apply the results to the
following two problems:

\gap

$\bullet $ {\bf Relaxed Problem:}
\begin{equation} \label{eq:relaxed sp_dc_constraint}
\begin{array}{ll}
\displaystyle{
\operatornamewithlimits{\mbox{\bf minimize}}_{x \, \in \, \overline{X}_{\rm rlx}(\gamma)}
} &  \mathbb{E}[ \, c_0(x,\tilde{z} )\, ].
\end{array}
\end{equation}
$\bullet $ {\bf Restricted Problem:}
\begin{equation} \label{eq:restricted sp_dc_constraint}
\begin{array}{ll}
\displaystyle{
\operatornamewithlimits{\mbox{\bf minimize}}_{x \, \in \, \overline{X}_{\rm rst}(\gamma)}
} & \mathbb{E}[ \, c_0(x,\tilde{z} )\, ].
\end{array}
\end{equation}
Abstracting assumption ($\boldsymbol{\cal Z}$) in
Section~\ref{sec:piecewise probabilistic constraints}
for the functionals $\{ {\cal Z}_{\ell} \}_{\ell \in [ L ]}$ and
assumption ($\boldsymbol{\Theta}$) for the functions $\wh{\theta}_{\rm rst/rlx}$,
we make the following blanket assumptions on the functions in (\ref{eq:new SP}).
Thus the assumptions below on
$c_{k\ell}(\bullet,\bullet;\gamma)$ are satisfied for
$c_{k\ell}^{\rm rlx}(\bullet,\bullet;\gamma)$ and
$c_{k\ell}^{\rm rst}(\bullet,\bullet;\gamma)$ that define the feasible sets in
problems \eqref{eq:relaxed sp_dc_constraint} and
\eqref{eq:restricted sp_dc_constraint}.

\gap

\fbox{
\parbox{6.5in}{
\begin{center}
{\bf Blanket Assumptions on (\ref{eq:new SP})}
\end{center}

$\bullet $ $X \subseteq \mathbb{R}^n$ is a closed convex set (and is a polytope starting
from Propostion~\ref{pr:approx stationary penalization})
and the objective function $c_0(\bullet,z)$ is nonnegative on $X$ for all $z \in \Xi$;
this holds for instance when $c_0(\bullet,\bullet)$ has a known lower bound on
$X \times \Xi$;

\gap

$\bullet $ {\bf Objective (A$_{\rm o}$):} the  function $c_0(\bullet,z)$ is
directionally differentiable and globally Lipschitz continuous
with a Lipschitz constant $\mbox{Lip}_0(z) > 0$ satisfying
$\mathbb{E}\left[ \, \mbox{Lip}_0(\tilde{z} ) \, \right] < \infty$.  This implies
that the expectation function $\bar c_0(x)$
is directionally differentiable and globally Lipschitz
continuous; moreover its directional derivative
$\bar{c}_0^{\, \prime}(\bar{x};v) = \mathbb{E}\left[
c_0(\bullet,\tilde{z} )^{\prime}(\bar{x};v)  \right]$
for all $( \bar{x},v ) \in X \times \mathbb{R}^n$; see
\cite[Theorem~7.44]{ShapiroDentchevaRuszczynski09} for the latter directional
derivative formula.

\gap

$\bullet $ {\bf Constraint (A$_{\rm c}$):}
there exist integrable functions $\mbox{Lip}_{\rm c}(\bullet)$
\textcolor{black}{and $\wh{\mbox{Lip}}_{\rm c}( \bullet )$ both mapping $\Xi$ into
$\mathbb{R}_{++}$} and a probability-one set $\Xi_{\rm c}$ such that
$\displaystyle{
\sup_{z \in \Xi_{\rm c}}
} \, \mbox{Lip}_{\rm c}(z) < \infty$ and

--- {\bf Uniform Lipschitz continuity in $x$:} for all tuples
$( \, x^1,x^2,z,\gamma \, ) \, \in \, X \, \times \, X \, \times \, \Xi_{\rm c} \,
\times \, \mathbb{R}_{++}$,
\begin{equation} \label{eq:Lip in x}
\left| \, c_{k\ell}(x^1,z;\gamma) - c_{k\ell}(x^2,z;\gamma) \, \right| \, \leq \,
\displaystyle{
\frac{\mbox{Lip}_{\rm c}(z)}{\gamma}
} \, \| \, x^1 - x^2 \, \|, \epc \forall \, (\, k,\ell \, ) \, \in \, [ \, K \, ] \,
\times \, [ \, L \, ];
\end{equation}
--- {\bf Uniform Lipschitz continuity in $1/\gamma$:} for all tuples
$( \, x,z,\gamma_1,\gamma_2 \, ) \, \in \, X \, \times\, \Xi_{\rm c}\, \, \times \,
\mathbb{R}_{++}^2$,
\[
\left| \, c_{k\ell}(x,z;\gamma_1) - c_{k\ell}(x,z;\gamma_2) \, \right|
\, \leq \, \wh{\mbox{Lip}}_{\rm c}(z) \, \left[ \, 1 + \| \, x \, \| \, \right] \,
\left| \, \displaystyle{
\frac{1}{\gamma_1}
} - \displaystyle{
\frac{1}{\gamma_2}
} \, \right|,
\epc \forall \,  (k,\ell\,)  \, \in \, [ \, K \, ] \times [ \, L \, ].
\]
{\bf Remark:} As it turns out, the latter Lipschitz continuity in $1/\gamma$ is not useful
for the analysis; nevertheless we include it for completeness and also in contrast to
the former Lipschitz continuity in $x$.  The noteworthy point of (\ref{eq:Lip in x}) is that
$\gamma$ appears in the denominator; this feature carries over to a later assumption about the
growth of the ``Rademacher average'' of the random variables $c_k(x,\bullet;\gamma) \triangleq
\displaystyle{
\sum_{\ell=1}^L
} \, c_{k\ell}(x,\bullet;\gamma)$.

\gap

$\bullet $ {\bf Interchangeability of directional derivatives ({\bf I}$_{\rm dd}$):}
each expectation function $\bar c_{k}(\bullet\,;\gamma)$ is
directionally differentiable with directional derivative given by
\[
\bar c_k(\bullet;\gamma)^{\, \prime}(\bar{x};v) \, = \, \displaystyle{
\sum_{\ell=1}^L
} \, \mathbb{E}\left[ \,c_{k\ell}(\bullet,\tilde{z};\gamma)^{\, \prime}(\bar{x};v)
\, \right],
\;\, \forall \, ( \bar{x},\gamma;v ) \, \in \, X \times \mathbb{R}_{++} \times
\mathbb{R}^n \;\mbox{and all $k \in [ K ]$}.
\]
}}

\gap

Associated with the expectation problem (\ref{eq:new SP}) is its
discretized/empirical (or sample average approximated) version corresponding
to a given family of samples
$Z^N \, \triangleq \, \{  z^s \}_{s=1}^N \, \subseteq \, \mathbb{R}^d$ for some positive integer $N$ that are realizations of the
nominal random variable $\tilde{z}$:
\begin{equation} \label{eq:sampled new SP}
\begin{array}{ll}
\displaystyle{
\operatornamewithlimits{\mbox{\bf minimize}}_{x \in X}
} & c_0^N(x) \, \triangleq \, \displaystyle{
\frac{1}{N}
} \, \displaystyle{
\sum_{s=1}^N
} \ c_0(x,z^s) \\ [0.1in]
\mbox{\bf subject to} & c_k^N(x;\gamma) \, \triangleq \,  \displaystyle{
\frac{1}{N}
} \, \displaystyle{
\sum_{s=1}^N
} \ \displaystyle{
\sum_{\ell=1}^L
} \, c_{k\ell}(x,z^s;\gamma) \, \leq \, \zeta_k,\epc k=1, \cdots, K,
\end{array} \end{equation}
whose feasible set we denote $\overline{X}(Z^N;\gamma)$.
This empirical problem is the key computational workhorse for solving the
expectation problem \eqref{eq:new SP}.

\subsection{Preliminaries on stationarity} \label{subsec:deterministic stationarity}

In order to define the stationary solutions of problem \eqref{eq:new SP} and its
empirical counterpart \eqref{eq:sampled new SP},
we first review some concepts in nonsmooth analysis \cite{RockafellarWets09,CuiPang2020}.
\textcolor{black}{By definition, a function $\phi : {\cal O} \subseteq  \mathbb{R}^n
\to \mathbb{R}$ defined on the open set ${\cal O}$
is {\sl B(ouligand)-differentiable} at $\bar{x} \in {\cal O}$ if $\phi$ is locally
Lipschitz continuous and directionally differentiable
at $\bar{x}$; the latter means that the (elementary) one-sided directional derivative:
\[
\phi^{\prime}(\bar{x};v) \, \triangleq \, \displaystyle{
\lim_{\tau \downarrow 0}
} \, \displaystyle{
\frac{\phi(\bar{x} + \tau v) - \phi(\bar{x})}{\tau}
}
\]
exists for all directions $v \in \mathbb{R}^n$.  By the locally Lipschitz continuity of
$\phi$ at $\bar{x}$, we have \cite[Proposition~4.4.1]{CuiPang2020}
\begin{equation} \label{eq:Lip+dd}
\displaystyle{
\lim_{\bar{x} \neq x \to \bar{x}}
} \, \displaystyle{
\frac{\phi(x) - \phi(\bar{x}) - \phi^{\prime}(\bar{x};x - \bar{x})}{
\| \, x - \bar{x} \, \|} \, = \, 0.
}
\end{equation}
The directional derivative $\phi^{\prime}(\bar{x};v)$ is in contrast to the
Clarke directional derivative
\[
\phi^{\circ}(\bar{x};v) \, \triangleq \, \displaystyle{
\limsup_{\substack{x \to \bar{x} \\ \tau \downarrow 0}}
} \, \displaystyle{\frac{\phi(x + \tau v) - \phi(x)}{\tau}
}\, , \epc ( \bar{x},v ) \, \in \, {\cal O} \, \times \, \mathbb{R}^n,
\]
which is always well defined and satisfies
$\phi^{\circ}(\bar{x};v) \geq \phi^{\, \prime}(\bar{x};v)$ for any pair $( \bar{x},v )$.
If equality holds for all $v \in \mathbb{R}^n$ at some $\bar{x} \in {\cal O}$, then we
say that $\phi$ is {\sl Clarke regular} at $\bar{x}$.  One key property of the Clarke
directional derivative is that} it is jointly upper
semicontinuous in the base point $\bar{x} \in {\cal O}$
and the direction $v \in \mathbb{R}^n$; that is, for
every sequence $\{ ( x^\nu,v^\nu ) \}$ converging to $( \bar{x},\bar{v})$, it holds that
\begin{equation} \label{eq:Clarke usc}
\displaystyle{
\limsup_{\nu \to \infty}
} \, \phi^{\circ}(x^\nu;v^\nu) \, \leq \, \phi^{\circ}(\bar{x};\bar{v}).
\end{equation}
The Clarke subdifferential of $\phi$ at $\bar{x}$ is defined as the set
\[
\partial_C \phi(\bar{x}) \, \triangleq \, \left\{ \, a \, \in \, \mathbb{R}^n \, : \, \phi^{\circ}(\bar{x};v) \, \geq \, a^{\top}v, \
\forall \, v \, \in \, \mathbb{R}^n \, \right\}.
\]
In general, we say that a vector $\bar{x}$ is a {\sl B-stationary} point of a
B-differentiable function $f_0$ on
a closed set $\wh{X} \subseteq {\cal O}$ if $\bar{x} \in \wh{X}$ and
\begin{equation} \label{eq:d-stationarity}
f_0^{\, \prime}(\bar{x};v) \, \geq \, 0, \epc \forall \, v \, \in \,
{\cal T}(\bar{x};\wh{X}),
\end{equation}
where ${\cal T}(\bar{x};\wh{X})$ is the (Bouligand) tangent cone of the set $\wh{X}$
at $\bar{x}$; by definition, a tangent vector $v$ in this cone
is the limit of a sequence $\left\{ \, \displaystyle{
\frac{x^{\nu} - \bar{x}}{\tau_{\nu}}
} \, \right\}$ where $\{ x^{\nu} \} \subset \wh{X}$ is a sequence of vectors converging
to $\bar{x}$ and $\{ \tau_{\nu} \}$ is a sequence of positive
scalars converging to zero.   When $\wh{X}$ is convex, we use the terminology
``d(irectional) stationarity'' for B-stationarity; in this case, the condition
(\ref{eq:d-stationarity}) is equivalent to
\[
f_0^{\, \prime}(\bar{x};x - \bar{x}) \, \geq \, 0, \epc \forall \, x \, \in \, \wh{X}.
\]
We say that $\bar{x}$ is a {\sl C(larke)-stationary} point of $f_0$ on $\wh{X}$
if the directional derivative $f_0^{\, \prime}(\bar{x};v)$
in (\ref{eq:d-stationarity}) is replaced
by the Clarke directional derivative.  An important special case is when the
set $\wh{X}$ is defined by B-differentiable constraints intersecting
a polyhedron $X$:
\[
\wh{X} \, = \, \displaystyle{
\bigcap_{k \, \in \, [ \, K \, ]}
} \, \left\{ \, x \, \in \, X \, \mid \, f_k(x) \, \leq \, 0 \, \right\},
\]
where each $f_k$ is B-differentiable.
We may then define the directional derivative based  ``linearization cone'' of
$\wh{X}$ at a given vector $\bar{x} \in \wh{X}$ as
\begin{equation} \label{eq:def of linearization}
{\cal L}(\bar{x};\wh{X}) \, \triangleq \, \displaystyle{
\bigcap_{k \in {\cal A}(\bar{x})}
} \, \left\{ \, v \, \in \, {\cal T}(\bar{x};X) \, : \, f_k^{\, \prime}(\bar{x};v)
\, \leq \, 0 \, \right\},
\end{equation}
where ${\cal A}(\bar{x}) \triangleq \left\{ \, k \, : \, f_k(\bar{x}) \, = \, 0 \,
\right\}$ is the index set of active constraints at $\bar{x}$.
Clearly we have
\begin{equation} \label{eq:relations of cones}
\mbox{cl}\left\{ \, v \, \in \, {\cal T}(\bar{x};X) \, \mid \,
f_k^{\, \prime}(\bar{x};v) < 0,  \,\,  \forall \
k \in {\cal A}(\bar{x}) \, \right\} \, \subseteq \, {\cal T}(\bar{x};\wh{X} ) \,
\subseteq \, {\cal L}(\bar{x};\wh{X}),
\end{equation}
where the first inclusion holds because by the closedness of the tangent
cone ${\cal T}(\bar{x};\wh{X})$, one
may take closures on both sides of the inclusion:
\begin{equation} \label{eq:weak relations of cones}
\left\{ \, v \, \in \, {\cal T}(\bar{x};X) \, \mid \, f_k^{\, \prime}(\bar{x};v) < 0 , \,\,  \forall \
k \in {\cal A}(\bar{x}) \, \right\} \, \subseteq \, {\cal T}(\bar{x};\wh{X} ).
\end{equation}
The second inequality in (\ref{eq:relations of cones}) holds
because 
for any sequence $\{ x^{\nu} \} \subset \wh{X}$ converging  to $\bar{x}$ and any
sequence $\{ \tau_{\nu} \} \downarrow 0$ with
$\displaystyle{
\lim_{\nu \to \infty}
} \, \displaystyle{
\frac{x^{\nu} - \bar{x}}{\tau_{\nu}}
} \, = \, v$, we have, by the B-differentiability of $f_k$ at $\bar{x}$,
\[ \begin{array}{llll}
f_k^{\, \prime}(\bar{x};v) & = & \displaystyle{
\lim_{\nu \to \infty}
} \, \displaystyle{
\frac{f_k(x^{\, \nu}) - f_k(\bar{x})}{\tau_{\nu}}
} &   {\mbox{by (\ref{eq:Lip+dd})}} \\ [0.2in]
& \leq & 0 &   {\mbox{for $k \in {\cal A}(\bar{x})$}}.
\end{array}
\]
  {Clearly, if
${\cal A}(\bar{x})$ is empty, the cone ${\cal L}(\bar{x};\wh{X})$ coincides with
${\cal T}(\bar{x};\wh{X}) = {\cal T}(\bar{x};X)$; i.e., the intersection operation in (\ref{eq:def of linearization}) is vacuous
in this case.}
This remark applies throughout the paper.  In general,
we say that the {\sl Abadie constraint qualification} (ACQ) holds for $\wh{X}$ at
$\bar{x} \in \wh{X}$ if
the last two sets in (\ref{eq:relations of cones}) are equal.  A sufficient condition
for the ACQ to hold is that
the {\sl directional Slater constraint qualification} holds for $\wh{X}$ at
$\bar{x} \in \wh{X}$; i.e., if the
first and the third sets in (\ref{eq:relations of cones}) are equal.
In turn, the latter directional Slater CQ holds if the left-hand set in
(\ref{eq:weak relations of cones}) is nonempty and $f_k^{\, \prime}(\bar{x};\bullet)$
is a convex function.   A function $f_k$ with the latter directional-derivative
convexity property has been coined a dd-convex function
in \cite[Definition~4.3.3]{CuiPang2020}.

\subsection{Convex-like property: B-stationarity implies locally minimizing}

A function $f : \mathbb{R}^n \to \mathbb{R}$ is said to be
{\sl convex-like} near a vector $\bar{x}$ if there exists a neighborhood
${\cal N}_{\bar x}$
of $\bar{x}$ such that
\[
f(x) \, \geq \, f(\bar{x}) + f^{\, \prime}(\bar{x};x - \bar{x}), \epc
\forall \, x \, \in \, {\cal N}_{\bar x}.
\]
It is clear that the class of convex-like functions near a fixed vector is closed
under nonnegative addition.
The fundamental role of this property for nonconvex functions was first discussed
in \cite[Proposition~4.1]{CuiChangHongPang20}, which we restate
in part (ii) of the following result.

\begin{proposition} \label{pr:stationary locmin} \rm
Let $X$ be a polyhedron.  Suppose that each $f_k$ for $k = 0, 1, \cdots, K$ is
B-differentiable on $\mathbb{R}^n$.  Let $\bar{x} \in \wh{X}$ be
arbitrary.  The following two statements hold.

\gap

(i) If $\bar{x}$ is a local minimizer of $f_0$ on $\wh{X}$, then $\bar{x}$ is a
B-stationary point of $f_0$ on $\wh{X}$.
%
%
\gap

(ii) If $f_k$ for $k = 0, 1, \cdots, K$ are all convex-like near $\bar{x}$,
the ACQ holds for $\wh{X}$ at $\bar{x}$,
and $\bar{x}$ is a B-stationary point of $f_0$ on $\wh{X}$, then $\bar{x}$ is a
local minimizer of $f_0$ on $\wh{X}$.
%
%
\end{proposition}

\begin{proof}  The first statement is a standard result.
To prove (ii), let $x \in \wh{X}$ be sufficiently near $\bar{x}$ such that the
convex-like inequality holds for all functions $f_k$.
For any $k \in \mathcal{A}(\bar x)$, it follows that
$f_k^{\,\prime}(\bar x; x-\bar x) \leq f_k(x) \leq 0$, and thus,
$x - \bar{x} \in {\cal L}(\bar{x};\wh{X}) = {\cal T}(\bar{x};\wh{X})$.
By the convex-like inequality for the function $f_0$ and the B-stationarity of
$\bar x$, we have
\[
f_0(x) \, \geq \, f_0(\bar{x}) + f_0^{\, \prime}(\bar{x};x-\bar{x}) \, \geq \,
f_0(\bar{x}),
\]
and thus the claim in (ii) follows.
\end{proof}

In what follows, we present a broad class of composite functions that have this
property.  Let
\begin{equation} \label{eq:generic fk}
f(x) \, \triangleq \, \varphi \circ \theta \, \circ \, \psi(x),
\end{equation}
where $\varphi : \mathbb{R} \to \mathbb{R}$ is piecewise affine and nondecreasing;
$\theta : \mathbb{R} \to \mathbb{R}$ is convex, and
$\psi : \mathbb{R}^n \to \mathbb{R}$ is piecewise affine.

\begin{lemma} \label{lm:piecewise convex-like} \rm
The function $f$ given by (\ref{eq:generic fk}) with properties as described is
convex-like near any $\bar{x} \in \mathbb{R}^n$.
\end{lemma}

{\it Proof.}  The key of the proof is the fact
(cf.~\cite[Proposition~4.1]{CuiChangHongPang20}) that for any piecewise
affine (PA) function $H : \mathbb{R}^M \to \mathbb{R}$
and any $\bar{y} \in \mathbb{R}^M$, there exists a neighborhood ${\cal N}_{\bar{y}}$
of $\bar{y}$ such that
\[
H(y) \, = \, H(\bar{y}) + H^{\, \prime}(\bar{y};y - \bar{y}), \epc \forall \,
y \, \in \, {\cal N}_{\bar{y}}.
\]
Applying this result to $\psi$ at $\bar{x}$ and also to $\varphi$ at
$\bar{t} \triangleq \theta(\psi(\bar{x}))$, we deduce the
existence of a neighborhood ${\cal N}_{\bar{x}}$ of $\bar{x}$ such that
\[ \begin{array}{lll}
\varphi \circ \theta \, \circ \, \psi(x) & = & \varphi(\bar{t}) + \varphi^{\, \prime}(\bar{t};\theta(\psi(x)) - \theta(\psi(\bar{x})))
\hspace{0.4in} \mbox{by PA property of $\varphi$ at $\bar{t}$} \\ [0.1in]
& \geq & \varphi(\bar{t}) + \varphi^{\, \prime}(\bar{t};\theta^{\, \prime}(\psi(\bar{x});\psi(x) - \psi(\bar{x})))
\ \mbox{by convexity of $\theta$ and $\uparrow$ property of $\varphi^{\, \prime}(\bar{t};\bullet)$} \\ [0.1in]
& = & \varphi(\bar{t}) + \varphi^{\, \prime}(\bar{t};\theta^{\, \prime}(\psi(\bar{x});\psi^{\, \prime}(\bar{x};x - \bar{x})))
\hspace{0.35in} \mbox{by PA property of $\psi$ at $\bar{x}$} \\ [0.1in]
& = & \varphi(\theta(\psi(\bar{x}))) + \left( \, \varphi \circ \theta \, \circ \, \psi \, \right)^{\, \prime}(\bar{x};x - \bar{x})
\ \mbox{by the chain rule of the dir.\ derivative.} \epc \Box
\end{array}
\]

With the above lemma, we can easily obtain the following corollary of Proposition~\ref{pr:stationary locmin} applied to the
empirical problem (\ref{eq:sampled new SP})  for a fixed sample batch $Z^N = \{ z^s \}_{s=1}^N$ when the problem is
derived from the expectation
problems \eqref{eq:relaxed sp_dc_constraint} and \eqref{eq:restricted sp_dc_constraint} with fixed $\gamma > 0$.  This requires
the functions
$c^{\, \rm rst}_{k\ell}(\bullet,z;\gamma)$  and $c^{\, \rm rlx}_{k\ell}(\bullet,z;\gamma)$ to have the composite structure
in  \eqref{eq:generic fk}.

\begin{corollary} \label{co:stationarity implies locmin} \rm
Let $X$ be a polyhedron and $\gamma > 0$ be a fixed but arbitrary scalar. Using the notation in
\eqref{defn: c_gamma}, we let each constraint function

\gap

$\bullet $ for the restricted problem:
$c_{k\ell}(\bullet,z;\gamma) = c_{k\ell}^{\, \rm rst}(\bullet,z;\gamma)$
for all $(k,\ell) \in [ K ] \times [ L ]$;

\gap

$\bullet $ for the relaxed problem:
$c_{k\ell}(\bullet,z;\gamma) = c_{k\ell}^{\, \rm rlx}(\bullet,z;\gamma)$
for all $(k,\ell) \in [ K ] \times [ L ]$.

\gap
Suppose that $\left\{ \, g_{i\ell}(\bullet,z) \, \right\}_{i=1}^{I_{\ell}}$ and
$\left\{ \, h_{j\ell}(\bullet,z) \, \right\}_{j=1}^{J_{\ell}}$
are all affine functions.  Then $c_{k\ell}(\bullet,z;\gamma)$ is convex-like near any
$\bar{x} \in X$ for all $(k,\ell) \in [ K ] \times [ L ]$,
provided that

\gap

$\bullet $ for the restricted problem: $\wh{\theta}_{\rm cvx}$ and
$\wh{\theta}_{\rm cve}$ are convex and concave functions, respectively;

\gap

$\bullet $ for the relaxed problem: $\wh{\theta}_{\rm cvx}$ and
$\wh{\theta}_{\rm cve}$ are piecewise affine (not necessarily convex/concave).

\gap

If additionally,
the objective function $c_0(\bullet,z)$ is convex-like near a B-stationary point
$\bar{x}$ of (\ref{eq:sampled new SP})
satisfying the ACQ for the feasible set $\overline{X}(Z^N;\gamma)$, then $\bar{x}$ is a
local minimizer of (\ref{eq:sampled new SP}).
\end{corollary}

\begin{proof}
Writing $t_{\ell} \triangleq {\cal Z}_{\ell}(x,z)$, we have, for the restricted problem,
\[
c^{{\, \rm rst}}_{k\ell}(x,z;\gamma) \, = \, e_{k\ell}^+
\min\left\{ \, \max\left( \, \wh{\theta}_{\rm cvx}\left( 1 + \displaystyle{
\frac{t_{\ell}}{\gamma}
} \right), \, 0 \right), \, 1 \, \right\} + e_{k\ell}^- \,
\min\left\{ \, \max\left( \, - \wh{\theta}_{\rm cve}\left( \displaystyle{
\frac{t_{\ell}}{\gamma}
} \right), \, -1 \, \right), \, 0 \, \right\}.
\]
By (\ref{eq:g and h}), it follows $t_{\ell}$ is a piecewise affine function of $x$
for fixed $z$.  Thus $c^{{\, \rm rst}}_{k\ell}(\bullet,z;\gamma)$
is of the kind (\ref{eq:generic fk}) and the claims hold in this case.
For the relaxed problem, we have
\[
c^{{\, \rm rlx}}_{k\ell}(x,z;\gamma) \, = \, -e_{k\ell}^-
\min\left\{ \, \max\left( \, \wh{\theta}_{\rm cvx}\left( 1 + \displaystyle{
\frac{t_{\ell}}{\gamma}
} \right), \, 0 \right), \, 1 \, \right\} - e_{k\ell}^+ \,
\min\left\{ \, \max\left( \, - \wh{\theta}_{\rm cve}\left( \displaystyle{
\frac{t_{\ell}}{\gamma}
} \right), \, -1 \, \right), \, 0 \, \right\},
\]
which shows that $c^{{\, \rm rlx}}_{k\ell}(\bullet,z;\gamma)$ is the composite
of piecewise affine functions, thus is piecewise affine itself.
Hence the  claims also hold in this case.
\end{proof}

\subsection{Asymptotic results for $\gamma \downarrow 0$}

Based on Proposition~\ref{pr:relations of limiting sets} that asserts
the limits of the approximating sets  $\overline{X}_{\rm rlx}(\gamma)$ and
$\overline{X}_{\rm rst}(\gamma)$ as $\gamma \downarrow 0$,
it is easy to show that under the zero-probability assumption in
Proposition~\ref{pr:relations of limiting sets}(iv), any accumulation
point of the globally optimal solutions of the relaxed problem
\eqref{eq:relaxed sp_dc_constraint} as $\gamma \downarrow 0$ must be a globally
optimal solution of the original chance-constrained
problem \eqref{eq:focus sp_dc_constraint}.  Additionally under the condition
in Proposition~\ref{pr:relations of limiting sets}(iii), any
accumulation point of the globally optimal solutions of the restricted problem
\eqref{eq:restricted sp_dc_constraint} as $\gamma \downarrow 0$
must be a globally optimal solution of the original chance-constrained problem
\eqref{eq:focus sp_dc_constraint}.  However, an accumulation point of
(strictly) locally optimal solutions $\{ \bar{x}_{\, \rm rst}(\gamma) \}$ of
\eqref{eq:restricted sp_dc_constraint} may not  be a locally optimal
solution of \eqref{eq:focus sp_dc_constraint} even with the conditions in
Proposition~\ref{pr:relations of limiting sets}(iii) and (iv).
In the following, we provide an example to illustrate the latter fact.  A slightly
modified example illustrates an unexpected limit with the relaxed problem.


\begin{example} \label{ex:a counterexample} \rm
Consider the problem
\begin{equation} \label{eq:example cc}
\displaystyle{
\operatornamewithlimits{\mbox{\bf minimize}}_{-1 \, \leq \, x \, \leq \, 1}
} \ x \epc
\mbox{\bf subject to } \ \mathbb{P}\left(\,
\textcolor{black}{Z - \max(2x,1-2x) \geq 0} \, \right) \, \leq \, \displaystyle{
\frac{1}{4}
},
\end{equation}
where the random variable $Z$ is uniformly distributed on $[\, -1,1\,]$.  We can show

\gap
	
$\bullet $ $\mathbb{P}\left(\, Z - \max(2x,1-2x) \, = \, 0 \, \right) = 0$
for $x \in [-1,1]$;
	
\gap
	
$\bullet $ $	\mathbb{P}\left( \, Z - \max(2x,1-2x) \geq 0 \, \right) \, = \,
\left\{\begin{array}{ll}
\left( 1-2x \right)/2 & \mbox{if $x\in \left[\,1/4, \, 1/2\,\right)$} \\[0.1in]
x & \mbox{if $x\in \left[ \, 0, 1/4\, \right)$} \\[0.1in]
0 & \mbox{if $x \in [-1,0) \, \cup \, \left[ \, 1/2 , \, 1 \,\right].$}
\end{array} \right.$

\gap

It follows that the conditions in
Proposition~\ref{pr:relations of limiting sets} (iii) and (iv) both hold.
Therefore, ${X}_{\rm cc} = [-1,1]$ and the unique
B-stationary point/local minimizer/global minimizer
of \eqref{eq:example cc} is $\bar{x} = -1$.

\gap
	
$\bullet $ Let $\wh{\theta}_{\rm cvx}(t) = t$ in $\phi_{\rm ub}(t,\gamma)$.
We have for any $\gamma\in (0,1/2)$,
\[
\begin{array}{ll}
\bar c^{\, \rm rst}(x; \gamma) & = \mathbb{E}\left[ \, \min\left( \, 1, \max\left( \, 1
+ \displaystyle{
\frac{1}{\gamma}
} \, \left( \, Z - \max(2x,1-2x) \, \right), \, 0 \, \right) \, \right) \, \right]
\\ [0.25in]
& = \mathbb{E}\left[ \, 1 \,\left| \;  1+ \displaystyle{
\frac{1}{\gamma}
} \, \left( Z-\max(2x,1-2x) \right) \, \geq 1 \right. \right] \times
\mathbb{P} \left(  {Z} - {\max(2x,1-2x)} \geq 0 \right) \ +  \\ [0.3in]
& \epc \left( \begin{array}{l}
\, \mathbb{E}\left[ 1+ \displaystyle{
\frac{1}{\gamma}
} \, \left( Z - \max(2x,1-2x) \right) \, \left| \; 1 + \displaystyle{
\frac{1}{\gamma}
} \, \left( \, Z - \max(2x,1-2x)  \, \right) \, \in \, (0,1) \, \right. \right] \\[0.3in]
\epc \ \times \, \mathbb{P} \left(  -\gamma < Z - \max(2x,1-2x)  < 0 \right)
\end{array} \, \right) \\ [0.5in]
& = \, \left\{ \begin{array}{ll}
\displaystyle{
\frac{(\gamma+2 x)^2}{4\gamma}
} & \mbox{if $x \, \in \, \left[ \, -\displaystyle{
\frac{\gamma}{2}
}, \, 0 \, \right)$} \\[0.25in]
x + \displaystyle{
\frac{\gamma}{4}
} & \mbox{if $x \, \in \, \left[ \, 0, \,
\displaystyle{
\frac{1}{4}
} \, \right)$} \\[0.25in]
\displaystyle{
\frac{1}{2}
} \, (1-2x) +\displaystyle{
\frac{\gamma}{4}
} & \mbox{if $x\, \in \, \left[ \, \displaystyle{
\frac{1}{4}
}, \, \displaystyle{
\frac{1}{2}
} \, \right)$} \\[0.25in]
\displaystyle{
\frac{(\gamma+1 - 2x)^2}{4\gamma}
} & \mbox{if $x \, \in \, \left[ \, \displaystyle{
\frac{1}{2}
}, \, \displaystyle{
\frac{1}{2}(1+\gamma)
} \, \right)$} \\[0.25in]
0 & \mbox{if $x \, \in \, \left[ \, -1, \, -\displaystyle{
\frac{\gamma}{2}
} \, \right) \, \cup \, \left[ \, \displaystyle{
\frac{1}{2}
} \, (1+\gamma), \, 1 \, \right]$}.
\end{array} \right.
\end{array}
\]
Therefore, $\overline{X}_{\rm rst}(\gamma) = \left[ \, -1, \, \displaystyle{
\frac{1-\gamma}{4}
} \, \right] \, \cup \, \left[ \, \displaystyle{
\frac{1 + \gamma}{4}
}, \, 1 \, \right]$ for any $\gamma \in (0,1/2)$.  Hence,
$\bar{x}_{\rm rst}(\gamma) = \displaystyle{
\frac{1 + \gamma}{4}
}$ is a B-stationary point and a strict local minimizer
of \eqref{eq:restricted sp_dc_constraint} for any $\gamma \in (0,1/2)$.
	
\gap
	
$\bullet $ However, the limit of $\{ \bar{x}_{\rm rst}(\gamma) \}$ as
$\gamma \downarrow 0$ is $\displaystyle\frac{1}{4}$,
which is not a local minimizer of \eqref{eq:example cc}.

\gap

\textcolor{black}{Alternatively, consider the following slight modification of the problem
\eqref{eq:example cc}:
\begin{equation}\label{ex for relaxed}
\displaystyle{
\operatornamewithlimits{\mbox{\bf minimize}}_{-1 \, \leq \, x \, \leq \, 1}
} \ -\left|\,x-\frac{3}{8}\,\right| \epc 
\mbox{\bf subject to } \ \mathbb{P}\left(\, Z - \max(2x,1-2x)\geq 0\, \right) \, \leq \,
\displaystyle{
\frac{1}{8}
}
\end{equation}
for  the same random variable $Z$.  Then
$X_{\rm cc} = \left[\,-1, \displaystyle\frac{1}{8}\,\right] \cup \left[\,
\displaystyle\frac{3}{8},1\,\right]$ and the local minimizer of the above problem is
$\{ -1, 1 \}$.  Letting $\wh{\theta}_{\rm cve}(t) = t$ and omitting the details,
we get
\[
\begin{array}{ll}
\bar{c}^{\, \rm rlx}(x; \gamma)
& = \, \left\{ \begin{array}{cl}
\displaystyle{
\frac{(1-2 x)^2}{4\gamma}
} & \mbox{if $x \, \in \, \left[ \, \displaystyle{
\frac{1-\gamma}{2}
}, \, \frac{1}{2} \, \right)$} \\[0.2in]
\displaystyle{
\frac{x^2}{\gamma}
} & \mbox{if $x \, \in \, \left[ \, 0, \, \displaystyle{
\frac{\gamma}{2}
} \, \right)$} \\[0.2in]
\displaystyle{
\frac{2 - \gamma}{4}
} - x & \mbox{if $x\, \in \, \left[ \, \displaystyle{
\frac{1}{4}
}, \, \displaystyle{
\frac{1-\gamma}{2}
} \, \right)$} \\[0.2in]
x - \displaystyle{
\frac{\gamma}{4}
} & \mbox{if $x \, \in \, \left[ \, \displaystyle{
\frac{\gamma}{2}
}, \, \displaystyle{
\frac{1}{4}
} \, \right)$} \\[0.2in]
0 & \mbox{if $x \, \in \, \left[ \, -1, \, 0\, \right) \, \cup \, \left[ \,
\displaystyle{
\frac{1}{2}
}, \, 1 \, \right]$}.
\end{array} \right.
\end{array}
\]
Therefore,
$\overline{X}_{\rm rlx}(\gamma) = \left[\, -1, \,\displaystyle{
\frac{1+2\gamma}{8}
} \,\right] \, \cup \, \left[ \,\displaystyle{
\frac{3-2\gamma}{8}
}, \, 1 \,\right]$ for any $\gamma \in (0,1/2)$.  Hence
$\bar{x}_{\rm rlx}(\gamma) = \displaystyle{
\frac{3-2\gamma}{8}
}$ is a strict local minimizer of the relaxed
problem \eqref{eq:relaxed sp_dc_constraint}.  However, the limit of
$\{\bar{x}_{\rm rlx}(\gamma)\}$ as $\gamma \downarrow 0$ is $\displaystyle{
\frac{3}{8}
}$, which is a global maximizer instead of a local minimizer of the original
problem \eqref{ex for relaxed}.  In this case, the relaxed problem has a bad local
minimizer that converges to a most undesirable point.}

\gap

Figure~\ref{fig:illustration cc} below shows the plot of the probability function
$\mathbb{P}\left(\, {Z}-{\max(2x,1-2x)}\geq 0\, \right)$, its restricted
approximation using $\wh{\theta}_{\rm cvx}(t)=t$ (left) and its relaxed
approximation using
$\wh{\theta}_{\rm cve}(t)=t$ (right).  From the figure, one can easily observe
the respective feasible regions of the original and approximate problems. \hfill $\Box$
\begin{figure*}[h]
\begin{center}
\includegraphics[width=0.45\textwidth]{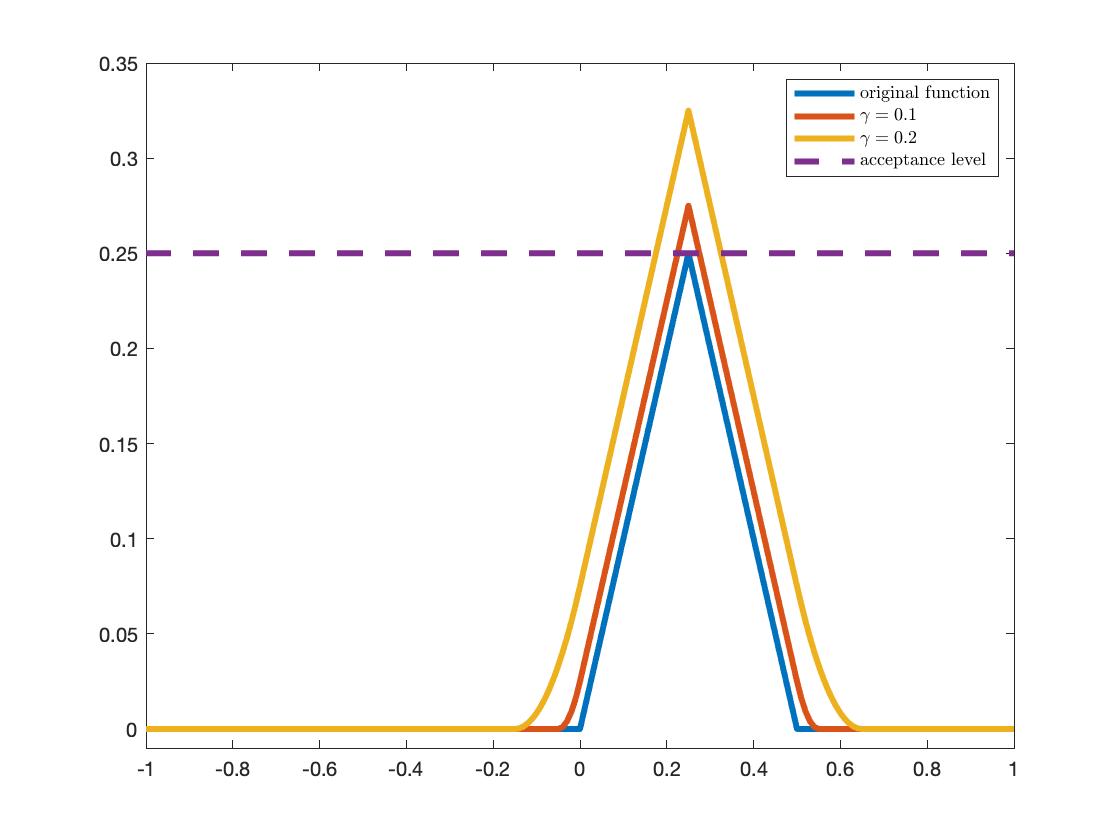}
\includegraphics[width=0.45\textwidth]{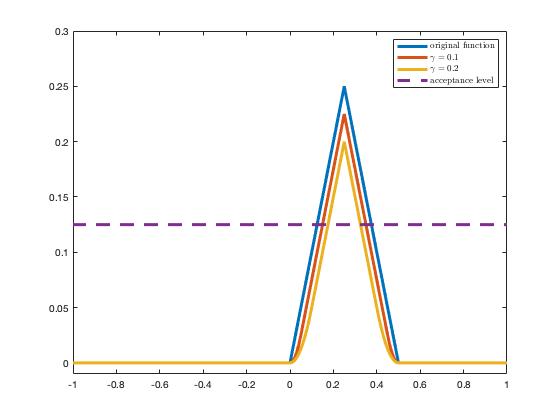}
\end{center}
\caption{A chance constraint and its approximations (left: restriction; right:
relaxation).} \label{fig:illustration cc}
\end{figure*}
\end{example}


While the above examples 
illustrates that limit points of the sequence of strict local minima of
the restricted/relaxed problem may not be a local minimum of the original
chance-constrained problem, it is possible to derive a simple result asserting
a weak kind of stationarity property of such a limit under minimal assumptions.
Phrasing this in a more general context, we consider a parameterized
family of closed sets $\{ C(w) \}_{w \in {\cal W}}$ and the associated optimization
problem:
\begin{equation} \label{eq:gamma-problem}
\displaystyle{
\operatornamewithlimits{\mbox{\bf minimize}}_{x \in C(w)}
} \ c_0(x),
\end{equation}
where the objective function $c_0$ is locally Lipschitz continuous.  Being fairly straightforward,
the next result
has two parts: the first part pertains to C-stationary points without
assuming convexity; this part is applicable to the families of restricted sets
$\{ \overline{X}_{\rm rst}(\gamma) \}$ and relaxed sets
$\{ \overline{X}_{\rm rlx}(\gamma) \}$.  The second part pertains to global minimizers
when these are computationally meaningful (e.g., when (\ref{eq:gamma-problem}) is a convex program);
this part is applicable to a family
$\{ C(\gamma_{\nu};x^{\nu} )\}_{\nu=1}^{\infty}$ of surrogate convex feasible sets
where $\{ x^{\nu} \}$ is a sequence of iterates with each
$x^{\nu}$ being associated with the scalar $\gamma_{\nu} > 0$.

\gap

For each $w \in {\cal W}$, let $\bar{x}^{\rm C}(w)$ be a C-stationary point of
(\ref{eq:gamma-problem}) and $\bar{x}^{\rm O}(w)$ be a globally optimal solution.
\textcolor{black}{Let the sequence $\{ w^{\nu} \}$ converge to $w^{\infty}$, and let
$\bar{C}(w^{\infty}) \triangleq \displaystyle{
\limsup_{\nu \to \infty}
} \, C(w^{\nu}) \triangleq \displaystyle{
\bigcap_{\nu \geq 1}
} \, \displaystyle{
\bigcup_{j \geq \nu}
} \, C(w^j)$.  Consider two arbitrary sequences $\{ \bar{x}^{\rm C}(w^{\nu}) \}$ and
$\{ \bar{x}^{\rm O}(w^{\nu}) \}$ of C-stationary points and global minima, respectively,
of the problem (\ref{eq:gamma-problem}) corresponding to the sequence $\{ w^{\nu} \}$.
We are interested in the respective C-stationary and globally minimizing properties
of the limit points of these sequences.  If the union $\displaystyle{
\bigcup_{\nu}
} \ C(w^{\nu})$ is bounded, then the two sequences must have convergent subsequences
whose limits we take as
$\bar{x}^{\rm C}(w^{\infty})$ and $\bar{x}^{\rm O}(w^{\infty})$.
It is clear that both limits belong to $\bar{C}(w^{\infty})$.}

\begin{proposition} \label{pr:limiting stationarity} \rm
In the above setting, the following two statements hold:
	
\gap
	
(a) $c_0^{\, \circ}\left( \bar{x}^{\rm C}(w^{\infty});v \right) \geq 0$ for all
$v \in \displaystyle{
\limsup_{\nu \to \infty}
} \, {\cal T}\left( \bar{x}^{\rm C}(w^{\nu});C(w^{\nu}) \right)$;
in particular, if
\[
{\cal T}\left( \bar{x}^{\rm C}(w^{\infty});\bar{C}(w^{\infty}) \right) \, \subseteq
\, \displaystyle{
\limsup_{\nu \to \infty}
} \, {\cal T}\left( \bar{x}^{\rm C}(w^{\nu});C(w^{\nu}) \right),
\]
then $\bar{x}^{\rm C}(w^{\infty})$ is a C-stationary solution of the limiting problem:
$\displaystyle{
\operatornamewithlimits{\mbox{\bf minimize}}_{x \in \bar{C}(w^{\infty})}
} \ c_0(x)$;
	
\gap
	
(b)
$\bar{x}^{\rm O}(w^{\infty}) \in \displaystyle{
\operatornamewithlimits{\mbox{\bf argmin}}_{x \in \bar{C}(w^{\infty})}
} \ c_0(x)$.
\end{proposition}

\begin{proof}  To prove statement (a), let $v \in \displaystyle{
\limsup_{\nu \to \infty}
} \, {\cal T}\left( \bar{x}^{\rm C}(w^{\nu});C(w^{\nu}) \right)$.  Then there exist
an infinite index set $\kappa$ and a sequence of vectors
$\{ v^{\nu} \}_{\nu \in \kappa}$ such that $v = \displaystyle{
\lim_{\nu (\in \kappa) \to \infty}
} \, v^{\nu}$ and
$v^{\nu} \in {\cal T}\left( \bar{x}^{\rm C}(w^{\nu});C(w^{\nu}) \right)$
for all $\nu \in \kappa$.  Therefore, we have
\[
c_0^{\, \circ}\left( \bar{x}^{\rm C}(w^{\nu});v^{\nu} \right) \, \geq \, 0 \, \epc
\forall \, \nu \, \in \, \kappa.
\]
By (\ref{eq:Clarke usc}), we pass to the limit $\nu (\in \kappa) \to \infty$ and
obtain the desired C-stationarity property
of $\bar{x}^{\rm C}(w^{\infty})$.  The second assertion in statement (a) is clear.
The proof of statement (b) is similar to that of (a) and omitted.	
\end{proof}

{\bf Example \ref{ex:a counterexample} continued.}  We have
${\cal T}\left( \bar{x}_{\rm rst}(\gamma);\overline{X}_{\rm rst}(\gamma) \right)
= \mathbb{R}_+$ for all $\gamma \in (0,1/2)$.
Since the objective function is the identity function, therefore
$c_0^{\, \prime}(x;v) = v$ for all pairs $(x,v) \in \mathbb{R}^2$;
hence the first assertion of Proposition~\ref{pr:limiting stationarity}(a)
is valid, even though the limit of
$\bar{x}_{\, \rm rst}(\gamma)$ as $\gamma \downarrow 0$
regrettably has no minimizing property with regards to the original
chance-constrained problem (\ref{eq:example cc}).  Of course, it is possible
in this example to obtain the unique global minimizer of the problem if one
identifies the global minimizers of the objective function over
the various approximating sets $\overline{X}_{\rm rst}(\gamma)$ for $\gamma > 0$.
From a practical computational perspective, it is in general not possible
to identify such a global minimizer when the problem is highly nonconvex
and coupled with nondifferentiability.  So one has to settle for the computable
solutions and understand their properties to the extent possible.  \hfill $\Box$

\gap

\textcolor{black}{{\bf A general comment:} In the above examples, the feasible regions
of the restricted
and relaxed problems are each the union of two intervals; due to the simplicity of the
objective functions, global minima of the restricted and relaxed problems can therefore
be identified and they will converge to the global minima of the respective problems
(\ref{eq:example cc}) and (\ref{ex for relaxed}).
However, in practical applications, we do not have the luxury of computing the global
minima exactly and the best we can settle for are stationary solutions, which under the
convexity-like property, are local minima.   These examples illustrate that if the
restricted/relaxed problems have ``bad'' local minima, their limits can be very
undesirable for the original CCP.   In the absence of favorable structures that can be
exploited, computing the ``sharpest'' kind of stationary solutions of the
restricted/relaxed/approximated problems, which themselves are most likely nonconvex
and nondifferentiable problems too, provides the first step toward obtaining a desirable
solution of the given CCP.  This important step is the guiding principle for the
developments in the rest of the paper.}

\section{External Sampling: Uniform Exact Penalization}
\label{sec:exact penalization}

This section develops a uniform exact penalization theory for the
following (un-parameterized) expectation constrained stochastic program,
without assuming any special structures on the constraint functions except for the
well-definedness of the expectation functions
and the Lipschitzian properties in Assumption (A$_{\rm Lip}$) below. In particular,
it covers the relaxed problem \eqref{eq:relaxed sp_dc_constraint}
and the restricted  problem \eqref{eq:restricted sp_dc_constraint} for the CCP with
a fixed $\gamma>0$ which we omit in this section.  Specifically, we consider
\begin{equation} \label{eq:exp SP}
\begin{array}{ll}
\displaystyle{
\operatornamewithlimits{\mbox{\bf minimize}}_{x \in X}
} & \bar c_0(x) \, \triangleq \, \mathbb{E}\left[ \, c_0(x, \tilde z) \, \right]
\\ [0.1in]
\mbox{\bf subject to} & \bar c_k(x) \, \triangleq \,
\underbrace{\mathbb{E}\left[ \, c_k(x, \tilde z) \, \right] \, \leq \, \zeta_k,
\epc k=1, \cdots, K}_{
\mbox{constraint set denoted $\wh{S}$}}.
\end{array} \end{equation}
To be self-contained for this section, we restate the blanket assumptions
(A$_{\rm o}$) and (A$_{\rm c}$) in the context of (\ref{eq:exp SP}):

\gap

{\bf Assumption (A$_{\rm Lip}$) for (\ref{eq:exp SP}):} the functions $\bar c_0$
and $c_k(\bullet,z)$ for all $z \in \Xi$ are
directionally differentiable; moreover,
the objective function $\bar c_0$ is Lipschitz continuous on $X$ with constant
$\mbox{Lip}_{\rm 0}$ and there exists an integrable
function $\mbox{Lip}_{\rm c} : \Xi \to \mathbb{R}_{++}$ such that for all
$k = 1, \cdots, K$,
\[
\left| \, c_k(x,z) - c_k(x^{\prime},z) \, \right| \, \leq \, \mbox{Lip}_{\rm c}(z)
\, \left\| \, x - x^{\prime} \, \right\|,
\epc \forall \; x,x^{\prime} \, \in \, X \mbox{ and all } z \, \in \, \Xi.
\]
Besides the
well-known benefit of transferring the (hard) constraints to the objective, exact
penalization is particularly useful in a stochastic setting where
the expectation constraints are discretized by sampling.  In practice, random sampling
of the constraint functions can generate a discretized
problem that is not feasible, thus leading to computational difficulties in a solution
algorithm.  With penalization, this becomes a non-issue. However, penalization raises
the question of exactness; that is, can feasibility be recovered with a uniformly
finite penalty parameter for all SAA problems with sufficiently large sample sizes?
Consistent with our perspective of solving nonconvex problems \cite{CuiPang2020},
our analysis below addresses stationary solutions under penalization.  We
denote the feasible set of (\ref{eq:exp SP}) by
$\wh{X} \, \triangleq \, X \, \cap \, \wh{S}$.

\gap

Given a penalty parameter $\lambda > 0$  applied to the residual function
$r_{\rm c}(x)$, we obtain the penalized version of
(\ref{eq:exp SP}):
\begin{equation} \label{eq:penalized exp SP}
\displaystyle{
\operatornamewithlimits{\mbox{\bf minimize}}_{x \in X}
} \ \bar c_0(x) + \lambda \, r_{\rm c}(x), \epc \mbox{where} \epc
r_{\rm c}(x) \, \triangleq \, \displaystyle{
\sum_{k=1}^K
} \, \max\left( \,\bar  c_k(x) - \zeta_k, \, 0 \, \right).
\end{equation}
Considering the above two problems with the family $\{ \bar c_k (x) \}_{k=0}^K$
treated as deterministic functions, we have the following exact penalization result
which is drawn from \cite[Proposition~9.2.2]{CuiPang2020}.

\begin{proposition}  \label{pr:stationary penalization} \rm
Let $X$ be a closed convex set and let $\{ \bar c_k \}_{k=0}^K$ be B-differentiable
functions defined on an open set containing $X$.  Suppose in addition that
$\bar c_0$ is Lipschitz continuous on $X$ with Lipschitz modulus $\mbox{Lip}_0 > 0$.  If
\begin{equation} \label{eq:dd assumption}
\displaystyle{
\operatornamewithlimits{\mbox{\bf supremum}}_{x \, \in \, X \, \setminus \, \wh{S}}
} \ \left[ \, \displaystyle{
\operatornamewithlimits{\mbox{\bf minimum}}_{v \in {\cal T}(x;X); \, \| v \| = 1}
} \ r_{\rm c}^{\, \prime}(x;v) \, \right] \, \leq \, -1,
\end{equation}
then for every $\lambda > \mbox{Lip}_0$, every directional stationary point of
(\ref{eq:penalized exp SP}) is a B-stationary point of (\ref{eq:exp SP}).
\end{proposition}

\gap
We make several remarks about the condition (\ref{eq:dd assumption}):

\gap

(a) It holds that
\begin{equation} \label{eq:dd expression}
r_{\rm c}^{\, \prime}(x;v) \, = \, \displaystyle{
\sum_{k \, : \, \bar c_k(x) \, > \, \zeta_k}
} \, \bar c_k^{\, \prime}(x;v) +  \displaystyle{
\sum_{k \, : \, \bar c_k(x) \, = \, \zeta_k}
} \, \max\left( \, \bar c_k^{\, \prime}(x;v), \, 0 \, \right), \epc \forall \, ( x,v ) \, \in \, X \times \mathbb{R}^n;
\end{equation}
(b) $r_{\rm c}^{\, \prime}(x;v) \geq 0$ for all $x \in \wh{S}$ and all $v \in \mathbb{R}^n$
  {(because the first summation on the right-hand side is vacuous for $x \in \wh{S}$)},
this sign property makes it clear that the restriction of $x$ outside the set $\wh{S}$ is essential in the condition (\ref{eq:dd assumption}).

\gap

(c) The condition (\ref{eq:dd assumption}) stipulates that for every  $x \in X$ which
is  infeasible to (\ref{eq:exp SP}), it is possible to drive $x$ closer to
feasibility by reducing the constraint residual function $r_c$ starting at $x$ and
moving along a descent direction that is tangent to the base
set $X$ at $x$.  Of course, this condition is intuitively needed for a penalized
vector to reach feasibility eventually for finite $\lambda$.

\gap

(d) The lower bound $\mbox{Lip}_0$ of the penalty parameter $\lambda$ matches the
right-hand bound of $-1$ in (\ref{eq:dd assumption}).  The essential
requirement in obtaining the exactness of the penalization (i.e., a finite lower bound
of $\lambda$) is that the left-hand supremum is negative.

\subsection{Stochastic penalization for Clarke stationarity}

Extending the deterministic treatment, we consider the approximation of the expectation
constraint functions by their sample averages,
leaving the expected objective function $\bar c_0$ as is (so that we can focus
on the treatment of the constraints).  Specifically, given the family of
samples $Z^N \triangleq \{ z^s \}_{s=1}^N$ of the random variable $\tilde{z}$,
we consider
\begin{equation} \label{eq:sampled for penalization}
\begin{array}{ll}
\displaystyle{
\operatornamewithlimits{\mbox{\bf minimize}}_{x \in X}
} & \bar c_0(x) \\ [5pt]
\mbox{\bf subject to} & c_k^N(x) \, \triangleq \, \displaystyle{
\frac{1}{N}
} \, \displaystyle{
\sum_{s=1}^N
} \, c_k(x,z^s) \, \leq \, \zeta_k, \epc k \, = \, 1, \cdots, K.
\end{array}
\end{equation}
The penalization of the latter problem with the penalty parameter $\lambda>0$ is:
\begin{equation} \label{eq:sampled penalized}
\displaystyle{
\operatornamewithlimits{\mbox{\bf minimize}}_{x \in X}
} \ \bar{c}_0(x) + \lambda \, r_{\rm c}^N(x), \epc \mbox{where} \epc
r_{\rm c}^N(x) \, \triangleq \, \displaystyle{
\sum_{k=1}^K
} \, \max\left( \, c_k^N(x) - \zeta_k, \, 0 \, \right).
\end{equation}
For an arbitrary $x \in X$ and a sample family $Z^N$, we define the index sets corresponding to the expectation problem \eqref{eq:exp SP} and the SAA problem \eqref{eq:sampled for penalization}:
\[ \begin{array}{lll}
{\cal A}_>(x) \, \triangleq \, \{ \, k \in [K] \, \mid \, \bar c_k(x) \, > \, \zeta_k \, \} & \mbox{ versus } &
{\cal A}_>^N(x) \, \triangleq \, \{ \, k \in [K] \, \mid \, c_k^N(x) \, > \, \zeta_k \, \} \\ [0.1in]
{\cal A}_<(x) \, \triangleq \, \{ \, k\in [K]  \, \mid \, \bar c_k(x) \, < \, \zeta_k \, \} & \mbox{ versus } &
{\cal A}_<^N(x) \, \triangleq \, \{ \, k \in [K] \, \mid \, c_k^N(x) \, < \, \zeta_k \, \} \\ [0.1in]
{\cal A}_=(x) \, \triangleq \, \{ \, k\in [K]  \, \mid \,\bar  c_k(x) \, = \, \zeta_k \, \} & \mbox{ versus } &
{\cal A}_=^N(x) \, \triangleq \, \{ \, k\in [K]  \, \mid \, c_k^N(x) \, = \, \zeta_k \, \}.
\end{array} \]
Our goal in what follows is to show that, under appropriate assumptions, for finite
values of the penalty parameter $\lambda > 0$ (that is independent of $N$),
if $\left\{ \, \bar{x}^{N,\lambda} \, \right\}_{N=1}^{\infty}$ is a sequence of
C-stationary points of (\ref{eq:sampled penalized}), then
every accumulation point of that sequence  is a {\sl weak C-stationary point} of the
expectation constrained problem (\ref{eq:exp SP}).
The latter point is defined as a feasible vector $\bar{x}$
to (\ref{eq:exp SP}) such that
\[ \begin{array}{l}
\bar{c}_0^\circ(\bar{x};v) \, \geq \, 0,
\epc \forall \;
v \, \in \, {\cal T}_{wC}(\bar{x};\wh{X} ) \, \triangleq \,
\left\{ \, v \in \, {\cal T}(\bar{x};X) \, \mid \,
\mathbb{E}\left[ \, c_k(\bullet,\tilde{z} )^\circ(\bar{x};v) \, \right] \,
\leq \, 0, \; \forall\; k \in {\cal A}_=(\bar{x}) \, \right\}.
\end{array} \]
We term this as a ``weak'' C-stationary point because
\[
\mathbb{E}\left[ \, c_k(\bullet,\tilde{z} )^\circ(x;v) \, \right] \, \geq \,
\mathbb{E}\left[ \, c_k(\bullet,\tilde{z} )^{\prime}(x;v) \, \right]
\, = \, \left( \, \mathbb{E}\left[ \, c_k(\bullet,\tilde{z} ) \, \right]
\, \right)^{\prime}(x;v) \, = \, \bar c_k^{\, \prime}(x;v).
\]
Hence, ${\cal T}_{wC}(\bar{x};\wh{X} ) \subseteq {\cal L}(\bar{x};\wh{X})$
with the right-hand (directional derivative based) linearization cone equal to
${\cal T}(\bar{x};\wh{X})$   {under the ACQ for the set $\wh{X}$ at $\bar{x}$}.
As we have noted, a sufficient condition
for the ACQ to hold is that the directional Slater CQ holds; i.e., if
\[
\mbox{cl}\left\{ \, v \, \in \, {\cal T}(\bar{x};X) \, \mid \,
\bar{c}_k^{\, \prime}(\bar{x};v) < 0,  \,\,  \forall \
k \in {\cal A}_=(\bar{x}) \, \right\} \, = \, {\cal L}(\bar{x};\wh{X}).
\]
It therefore follows that if these CQs hold for the set $\wh{X}$ at $\bar{x}$,
and if $\bar{x}$ is a C-stationary point as defined in
Subsection~\ref{subsec:deterministic stationarity}, then $\bar{x}$ must be a
weak C-stationary
point; the converse holds if 
$c_k(\bullet, z)$ is Clarke regular for almost every $z\in \Xi$ so that the two cones
${\cal T}_{wC}(\bar{x};\wh{X} )$ and ${\cal L}(\bar{x};\wh{X} )$
are equal.  This connection with C-stationarity explains the adjective ``weak''.

\gap

In terms of the above defined index sets, we have
\[
\begin{array}{ll}
( \, r_{\rm c}^N \, )^{\, \prime}(x;v) \, &= \, \displaystyle{
	\sum_{k \, \in \, {\cal A}_>^N(x)}
} \,  ( c_k^N )^{\prime}(x;v) +  \displaystyle{
	\sum_{k \, \in \, {\cal A}_=^N(x)}
} \, \max\left( \,  ( c_k^N )^{\prime}(x;v), \, 0 \, \right) \\[0.2in]
& = \displaystyle{
\sum_{k \, \in \, {\cal A}_>^N(x)}
} \, {\displaystyle{
\frac{1}{N}
} \, \displaystyle{
\sum_{s=1}^N
} \, c_k(\bullet,z^s)^{\prime}(x;v)}  +  \displaystyle{
\sum_{k \, \in \, {\cal A}_=^N(x)}
} \, \max\left( \, \displaystyle{
\frac{1}{N}
} \, \displaystyle{
\sum_{s=1}^N
} \, c_k(\bullet,z^s)^{\prime}(x;v), \, 0 \, \right).
\end{array}
\]
Note that unlike ${\cal A}_>(x)$ and ${\cal A}_=(x)$ which are deterministic index
sets in  (\ref{eq:dd expression}) for the directional derivative of the expectation
function, ${\cal A}_>^N(x)$ and ${\cal A}_=^N(x)$ are sample-dependent index sets.
Using the Clarke directional derivative, we define, for any $x \in X$ and $v$ in
$\mathbb{R}^n$,
\begin{equation}
\label{def:weak_clarke_dd}
 \begin{array}{rll}
( \, \wh{r}_{\rm c}^{\, N} \, )^{\circ}(x;v) & \triangleq  \displaystyle{
\sum_{k \, \in \, {\cal A}_>^N(x)}
} \, \displaystyle{
\frac{1}{N}
} \, \displaystyle{
\sum_{s=1}^N
} \, c_k(\bullet,z^s)^{\circ}(x;v) +  \displaystyle{
\sum_{k \, \in \, {\cal A}_=^N(x)}
} \, \max\left( \, \displaystyle{
\frac{1}{N}
} \, \displaystyle{
\sum_{s=1}^N
} \, c_k(\bullet,z^s)^{\circ}(x;v), \, 0 \, \right) \\ [0.25in]
\wh{r}_{\rm c}^{\, \circ}(x;v) & \triangleq  \displaystyle{
\sum_{k \, \in \, {\cal A}_>(x)}
} \, \mathbb{E}\left[ \, c_k(\bullet,\tilde{z} )^{\circ}(x;v) \, \right] +  \displaystyle{
\sum_{k \, \in \, {\cal A}_=(x)}
} \, \max\left( \, \mathbb{E}\left[ \, c_k(\bullet,\tilde{z} )^{\circ}(x;v) \, \right], \, 0 \, \right).
\end{array}
\end{equation}
Notice that in general $( \, \wh{r}_{\rm c}^N \, )^{\circ}(x;v)$ is not equal to $( \, r_{\rm c}^N \, )^{\circ}(x;v)$ due to the failure of the additivity
of the Clarke directional derivative in terms of the directions; nevertheless, we have
\begin{equation} \label{eq:pointwise of Clarke}
( \, \wh{r}_{\rm c}^N \, )^{\circ}(x;v) \, \geq \, ( \, r_{\rm c}^N \, )^{\circ}(x;v) \, \geq \,
( r_{\rm c}^N )^{\prime}(x;v) , \epc \forall \, ( x,v) \, \in \ X \, \times \,
\mathbb{R}^n;
\end{equation}
similar inequalities hold for the residual of the expectation constraint functions.

\gap

Given two sets $A$ and $B$ in $\mathbb{R}^n$, we denote the (one-side) deviation of
$A$ from $B$ as
\[
\mathbb{D}(A,B) \, \triangleq \, \sup_{x\in A} \, \mbox{dist}(x,B) \, = \,
\sup_{x\in A} \,\inf_{y\in B} \, \| \, x-y \, \|.
\]
The following lemma is a direct consequence of \cite[Theorem 2]{ShapiroXu2002}.
In the lemma, we write $\partial_C c_k(x,z)$ for the Clarke subdifferential
of $c_k(\bullet,z)$ at $x$.

\begin{lemma} \label{lm:limsup of Clarke} \rm
Let $X$ be a compact set and let (A$_{\rm Lip}$) hold.   Let $\{ z^s \}_{s=1}^{\infty}$
be independent realizations of the random vector $\tilde{z}$.
For any $v\in \mathbb{R}^n$, it holds that for all $k = 1, \cdots, K$,
\[
\displaystyle{
\limsup_{N \to \infty}
} \, \displaystyle{
\sup_{x\in X}
} \, \left( \, \displaystyle{
\frac{1}{N}
} \, \displaystyle{
\sum_{s=1}^N
} \, c_k(\bullet,z^s)^{\circ}(x;v) - \mathbb{E}\left[ \,
c_k (\bullet,\tilde{z} )^{\circ}(x;v) \, \right] \, \right) \, \leq \, 0
\epc \mbox{almost surely}.
\]
\end{lemma}

\begin{proof}  It follows from \cite[Theorem 2]{ShapiroXu2002} that for any $\delta > 0$,
\[
\sup_{x\in X} \, \mathbb{D}\left( \, \displaystyle{
\frac{1}{N}
} \, \displaystyle{
\sum_{s=1}^N
} \,  \partial_C \, c_k(x,z^s), \, \displaystyle{
\bigcup_{x^{\prime} \in \mathbb{B}_{\delta}(x)}
} \, \mathbb{E}\left[\,\partial_C \,c_k(x^\prime,\tilde{z} ) \, \right] \, \right) \to 0
\epc \mbox{as $N\to \infty$} \epc \mbox{almost surely}.
\]
Consider any $v\in \mathbb{R}^n$ and any $\delta > 0$.  Since for any $x\in X$ and any
$z \in \Xi$,
\[
\partial_C \,c_k(x,z) \, = \, \left\{ \, a \, \in \, \mathbb{R}^n \, \mid \,
c_k(\bullet,z)^\circ (x;v) \, \geq \, a^\top v,
\;\ \forall\; v \, \in \, \mathbb{R}^n \, \right\},
\]
we derive that for any $x\in X$, $v\in \mathbb{R}^n$, and $\varepsilon>0$, there exist
a positive integer $\overline{N}$ independent of $x$,
vectors $\{ a^s\in \partial_C \,c_k(x,z^s)\}_{s=1}^N$
and $\bar{a}\in \displaystyle{
\bigcup_{x^\prime\in \mathbb{B}_\delta(x)}
} \, \mathbb{E}\left[\,\partial_C \, c_k(x^\prime,\tilde{z} )\,\right]$ such that for
all $N \geq \overline{N}$,
\[
\begin{array}{rl}
\displaystyle{
\frac{1}{N}
} \, \displaystyle{
\sum_{s=1}^N
} \, c_k(\bullet,z^s)^\circ (x;v) \, = & \displaystyle{
\frac{1}{N}
} \, \displaystyle{
\sum_{s=1}^N
} \, ( a^s )^\top v \, \leq \, \bar{a}^\top v + \varepsilon \, \leq \, \displaystyle{
\limsup_{x^{\prime} \in \mathbb{B}_{\delta}(x)}
} \, \mathbb{E}\left[\,c_k(\bullet,\tilde z)^\circ(x^\prime;v)\,\right] + \varepsilon.
\end{array}
\]
We can thus derive the stated result by taking $N\to \infty$, $\varepsilon \downarrow 0$,
and using the upper semicontinuity of the Clarke directional derivative.
\end{proof}

The above lemma yields the following sequential generalization of the pointwise
inequalities (\ref{eq:pointwise of Clarke}).

\begin{lemma} \label{lm:usc of residual} \rm
Let $X$ be a compact set and let (A$_{\rm Lip}$) hold.  Then for any
$v\in \mathbb{R}^n$ and
every sequence $\{x^N \}\subseteq X$ converging to $\bar{x} \in X$, it holds that
\[
\displaystyle{
\limsup_{N \to \infty}
} \, ( \, \wh{r}_{\rm c}^{\, N} \, )^{\circ}(x^N;v) \, \leq \, \wh{r}_c^{\, \circ}(\bar{x};v)
\epc \mbox{almost surely}.
\]
\end{lemma}

\begin{proof}
It follows from the uniform law of large numbers
(cf.\ \cite[Theorem 7.48]{ShapiroDentchevaRuszczynski09}) that
\begin{equation} \label{eq:approx conditions}
\begin{array}{l}
\displaystyle{
\lim_{N \to \infty}
} \, \displaystyle{
\sup_{x \in X}
} \, \left| \, c_k^N(x) - \bar c_k(x) \, \right| \, = \, 0 \epc \mbox{almost surely},
\epc \forall \, k \, = \, 1, \cdots, K.
\end{array} \end{equation}
Since we have
\[
c_k^N(x^N) - \bar c_k(\bar{x}) \, = \, \left[ \, c_k^N(x^N) - \bar c_k(x^N) \, \right]
+ \left[ \, \bar c_k(x^N) - \bar c_k(\bar{x}) \, \right],
\]
we may obtain, \textcolor{black}{by the continuity of $c_k(\bullet,z^s)$,} that for all
$N$ sufficiently large,
\[
{\cal A}_>(\bar{x}) \, \subseteq \,
{\cal A}_>^N(x^N) \epc \mbox{and} \epc
{\cal A}_>^N(x^N) \, \cup \, {\cal A}_=^N(x^N)
\, \subseteq \, {\cal A}_>(\bar{x}) \, \cup \, {\cal A}_=(\bar{x})
\epc \mbox{almost surely}.
\]
  {The first inclusion rules out that an index $k \in {\cal A}_=^N(x^N)$ belongs to ${\cal A}_>(\bar{x})$.}
Hence, for any $\varepsilon>0$, there exists a  sufficiently large $N$ such that
the following string of inequalities hold almost surely:
{\small
\[
\begin{array}{rl}
( \, \wh{r}_{\rm c}^{\, N} \, )^{\circ}(x^N;v) 
\,  = & \displaystyle{
\sum_{k \, \in \, {\cal A}_{\textcolor{black}{>}}^N(x^N)}
} \displaystyle{
\frac{1}{N}
} \, \displaystyle{
\sum_{s=1}^N
} \,  c_k(\bullet,z^s)^{\circ}(x^N;v) +  \displaystyle{
\sum_{k \, \in \, {\cal A}_=^N(x^N)}
} \, \max\left( \, \displaystyle{
\frac{1}{N}
} \, \displaystyle{
\sum_{s=1}^N
} \, c_k(\bullet,z^s)^{\circ}(x^N;v), \, 0 \, \right) \\ [0.3in]
\leq & \displaystyle{
\sum_{k \, \in \, {\cal A}_{\textcolor{black}{>}}(\bar{x})}
} \displaystyle{
\frac{1}{N}
} \, \displaystyle{
\sum_{s=1}^N
} \, c_k(\bullet,z^s)^{\circ}(x^N;v) +  \displaystyle{
\sum_{k \, \in \, {\cal A}_=(\bar{x})}
} \, \max\left( \, \displaystyle{
\frac{1}{N}
} \, \displaystyle{
\sum_{s=1}^N
} \, c_k(\bullet,z^s)^{\circ}(x^N;v), \, 0 \, \right) \\ [0.3in]
\leq & \displaystyle{
\sum_{k \, \in \, {\cal A}_{\textcolor{black}{>}}(\bar{x})}
} \, \mathbb{E}\,\left[\,c_k(\bullet, \tilde z)^{\circ}(x^N;v)\,\right] +
\displaystyle{
\sum_{k \, \in \, {\cal A}_=(\bar{x})}
} \, \max\left( \, \mathbb{E}\,\left[\,c_k(\bullet,\tilde z)^{\circ}(x^N;v)\,\right],
\, 0 \, \right) \ + \\ [0.3in]
&\epc \displaystyle{
\sum_{k=1}^K
} \, \max\left( \, \displaystyle{
\frac{1}{N}
} \, \displaystyle{
\sum_{s=1}^N
} \, c_k(\bullet,z^s)^{\circ}(x^N;v) - \mathbb{E}\left[ \,
c_k(\bullet,\tilde z)^{\, \circ}(x^N;v)\,\right], \, 0 \, \right) \\ [0.3in]
\leq & \displaystyle{
\sum_{k \, \in \, {\cal A}_{\textcolor{black}{>}}(\bar{x})}
} \, \mathbb{E}\,\left[\, c_k(\bullet,\tilde z)^{\circ}(x^N;v)\,\right] + \displaystyle{
\sum_{k \, \in \, {\cal A}_=(\bar{x})}
} \, \max\left( \, \mathbb{E}\,\left[\,c_k(\bullet,\tilde z)^{\circ}(x^N;v)\,\right],
\, 0 \, \right) + \varepsilon,
\end{array}
\]
}

\noindent where the last inequality is due to Lemma \ref{lm:limsup of Clarke}.
By the upper semicontinuity (\ref{eq:Clarke usc}) of the Clarke directional derivative,
the desired conclusion follows.
\end{proof}

For each pair $(Z^N,\lambda)$, let $\bar{x}^{N,\lambda}$ be a C-stationary point of
(\ref{eq:sampled penalized}).
The result below shows that a finite $\bar{\lambda} > 0$ exists such that for all
$\lambda > \bar{\lambda}$, every
accumulation point of the sequence $\{ \bar{x}^{N,\lambda} \}$ is feasible for
(\ref{eq:exp SP}) and is a weak C-stationary point of this expectation-constrained
problem.  The proof of this result is
based on the above two technical lemmas and by strengthening the sufficient condition
(\ref{eq:dd assumption}).  Note that the result
does not address how the iterate $\bar{x}^{N,\lambda}$ is obtained.  Thus, the result
is in the spirit of the convergence analysis
of an SAA scheme, albeit it pertains to a stationary point as opposed to a minimizer.

\begin{proposition} \label{pr:approx stationary penalization} \rm
Let $X$ be a polyhedron.  Assume that (A$_{\rm Lip}$) holds and
\begin{equation} \label{eq:Clarke dd assumption}
\displaystyle{
\operatornamewithlimits{\mbox{\bf supremum}}_{x \, \in \, X \, \setminus \, \wh{S}}
} \ \left[ \, \displaystyle{
\operatornamewithlimits{\mbox{\bf minimum}}_{v \in {\cal T}(x;X); \, \| v \| = 1}
} \ \wh{r}_{\rm c}^{\, \circ}(x;v) \, \right] \, \leq \, -1,
\end{equation}
where $\wh r _c^{\,\circ} (x; v)$ is defined in \eqref{def:weak_clarke_dd}.
Then, for 
every $\lambda > \mbox{Lip}_0$,
the following three statements hold for any accumulation point $\bar{x}^{\lambda}$
of the sequence $\{ \bar{x}^{N,\lambda} \}_{N=1}^{\infty}$:

\gap

(a) $\bar{x}^{\lambda} \in \wh{S}$ almost surely; thus $\bar{x}^{\lambda}$ is feasible
to (\ref{eq:exp SP}) almost surely;

\gap

(b) $\bar{x}^{\lambda}$ is a weak C-stationary point of (\ref{eq:exp SP}) almost surely;

\gap

(c) if for almost every $z \in \Xi$,   {each function in the family
$\{ c_k(\bullet,z) \}_{k=0}^K$} is Clarke regular at $\bar{x}^\lambda$, then
$\bar{x}^\lambda$ is a B-stationary point of
(\ref{eq:exp SP}) almost surely.
\end{proposition}

\begin{proof}  The C-stationarity condition at $\bar{x}^{N,\lambda}$ of
\textcolor{black}{(\ref{eq:sampled penalized})} implies that
\begin{equation} \label{eq:N stationarity}
( \, \bar c_0 \, )^{\circ}(\bar{x}^{N,\lambda};v) +
\lambda \, ( r^N_{\rm c} )^{\circ}(\bar{x}^{N,\lambda};v) \, \geq \, 0, \epc
\forall \, v \, \in \, {\cal T}(\bar{x}^{N,\lambda};X).
\end{equation}
For simplicity, we assume that $\bar{x}^{\lambda}$ is the limit of the sequence
$\{ \bar{x}^{N,\lambda} \}$.  We claim
that $\bar{x}^{\lambda} \in \wh{S}$ almost surely.  Assume by contradiction that there
exists positive probability  such that
$\bar{x}^{\lambda} \not\in \wh{S}$. 
Then restricted to the event where $\bar{x}^{\lambda} \not\in \wh{S}$,  there exists
$\bar{v} \in {\cal T}(\bar{x}^{\lambda};X)$ with $\| \bar{v} \| = 1$ such that
\[
\wh{r}_{\rm c}^{\, \circ}(\bar{x}^{\lambda};\bar{v}) \, \leq \, -1.
\]
Since $X$ is a polyhedron, it follows that with $N$ sufficiently large,
$\bar{v}$ belongs to ${\cal T}(x^{N,\lambda};X)$.
Let $\varepsilon \in \left( \, 0, \, 1 - \displaystyle{
\frac{\mbox{Lip}_0}{\lambda}
} \, \right)$.
Then from Lemma \ref{lm:usc of residual}, there exists $N$ such that
$( r_c^N )^\circ(\bar{x}^{N,\lambda};\bar{v}) \leq \wh{r}_c^{\, \circ}(\bar{x}^{\lambda};
\bar{v}) + \varepsilon$ almost surely.
By substituting $v = \bar{v}$ into (\ref{eq:N stationarity})
\textcolor{black}{and noting
$| \, ( \, \bar c_0 \, )^{\circ}(\bar{x}^{N,\lambda};\bar{v}) \, | \, \leq \,
\mbox{Lip}_0 \, \| \, \bar{v} \, \| \, = \, \mbox{Lip}_0$,} we deduce
\[
0  \leq  \mbox{Lip}_0 + \lambda \, \left( \, \wh{r}_{\rm c}^{\, \circ}(\bar{x}^{\lambda};
\bar{v}) + \varepsilon  \, \right)
\, \leq \, \mbox{Lip}_0 + \lambda \, \left( \, -1 + \varepsilon \, \right) 
\, < \, 0,
\]
which is a contradiction.  Therefore, $\bar{x}^{\lambda} \in \wh{S}$ almost surely.
To show the claimed weak C-stationarity of $\bar{x}^{\lambda}$ for
the problem (\ref{eq:exp SP}), let
$v \in {\cal T}_{wC}(\bar{x}^{\lambda};X \, \cap \, \wh{S})$ be arbitrary with
unit length.  For such a tangent vector $v$, we
have $\mathbb{E}\,\left[\,c_k(\bullet,\tilde z)^{\, \circ}(\bar{x}^{\lambda};v)\,\right]
\leq 0$ for all $k \in {\cal A}_=(\bar{x}^{\lambda})$.  Moreover, since
$\bar{x}^{\lambda} \in \wh{S}$, thus   {${\cal A}_>(\bar{x}^{\lambda}) = \emptyset$},
we have ${\cal A}_>^N(\bar{x}^{N,\lambda}) \, \cup \,
{\cal A}_=^N(\bar{x}^{N,\lambda}) \, \subseteq {\cal A}_=(\bar{x}^{\lambda})$
for all $N$ sufficiently large almost surely.  Hence, for any $\varepsilon^{\,\prime}>0$
and sufficiently large $N$, the following inequalities hold almost surely:
\[
\begin{array}{rl}
0 \, \leq & \Bigg( \, \bar{c}_0(\bullet)  + \lambda \, \displaystyle{
\sum_{k \in {\cal A}_>^N(\bar{x}^{N,\lambda})}
} \, c_k^N(\bullet\,;v) + \lambda \,
\displaystyle{
\sum_{k \in  {\cal A}_=^N(\bar{x}^{N,\lambda})}
} \, \max\left( \, c_k^N(\bullet\,;v),0 \, \right)\, \Bigg)^\circ(\bar{x}^{N,\lambda};v)
\\[0.3in]
\leq & ( \, \bar{c}_0 \, )^{\circ}(\bar{x}^{N,\lambda};v) + \lambda \, \displaystyle{
\sum_{k \, \in \, {\cal A}_=(\bar{x}^{\lambda})}
} \, \max\left( \, ( \, c_k^N \, )^{\circ}(\bar{x}^{N,\lambda};v), \, 0 \, \right)
\\ [0.25in]
\leq & ( \, \bar{c}_0 \, )^{\circ}(\bar{x}^{N,\lambda};v) +  \lambda \, \displaystyle{
\sum_{k \, \in \, {\cal A}_=(\bar{x}^{\lambda})}
} \, \max\Bigg( \, \displaystyle{
\frac{1}{N}
} \, \displaystyle{
\sum_{s=1}^N
} \, c_k(\bullet,z^s)^{\circ}(\bar{x}^{N,\lambda};v) - \mathbb{E}\left[ \,
c_k(\bullet,\tilde z)^{\circ}(\bar{x}^{N,\lambda};v) \, \right], \, 0 \, \Bigg)
\\ [0.3in]
&  +\,\lambda \, \displaystyle{
\sum_{k \, \in \, {\cal A}_=(\bar{x}^{\lambda})}
} \, \max\left( \, \mathbb{E}\left[ \, c_k(\bullet,\tilde z)^{\circ}(
\bar{x}^{N,\lambda};v)\,\right] -
\mathbb{E}\left[\,c_k(\bullet,\tilde z)^{\circ}(\bar{x}^{\lambda};v)\,\right], \, 0
\, \right) \\ [0.25in]
\leq & ( \, \bar{c}_0 \, )^{\circ}(\bar{x}^{N,\lambda};v) + \lambda \, \varepsilon^{\, \prime} +
\lambda \, \displaystyle{
\sum_{k \, \in \, {\cal A}_=(\bar{x}^{\lambda})}
} \, \max\left( \, \mathbb{E}\,\left[\,c_k(\bullet,\tilde z)^{\circ}(
\bar{x}^{N,\lambda};v) - c_k(\bullet,\tilde z)^{\circ}(\bar{x}^{\lambda};v)\,\right],
\, 0 \, \right).
\end{array} \]
Letting $N \to \infty$ and using the upper semicontinuity of the Clarke directional
derivative at $\bar{x}^{\lambda}$, we deduce that almost surely,
\[
( \, \bar{c}_0 \, )^{\circ}(\bar{x}^{\lambda};v) \,\geq \, 0,  \epc \forall \, v \, \in \,
{\cal T}_{wC}(\bar{x}^\lambda;X \, \cap \, \wh{S}),
\]
which is the almost sure weak C-stationarity of $\bar{x}^{\lambda}$ for the problem
(\ref{eq:exp SP}). 

\gap

To prove part (c),   {as we have already noted that,
under the Clarke regularity of $\{ c_k(\bullet,z) \}_{k=1}^K$ at
$\bar{x}^{\lambda}$ for almost all $z \in \Xi$, we have
${\cal T}_{wC}(\bar{x}^\lambda;X \, \cap \, \wh{S}) = {\cal L}(\bar{x}^\lambda;X \, \cap \, \wh{S})$.
Hence, by (b), it follows that $( \, \bar{c}_0 \, )^{\circ}(\bar{x}^{\lambda};v) \geq 0$ for all
$v \in {\cal L}(\bar{x}^\lambda;X \, \cap \, \wh{S}) \supseteq {\cal T}(\bar{x}^\lambda;X \, \cap \, \wh{S})$.}
Since
\[ \begin{array}{lll}
\left( \,\mathbb{E} \,\left[\, c_0(\bullet,\tilde{z} )\,\right] \,
\right)^{\prime}(\bar{x}^\lambda;v) & = &
\mathbb{E} \,\left[\,c_0(\bullet,\tilde{z} )^{\prime}(\bar{x}^\lambda;v)\,\right]
\\ [0.1in]
& = & \mathbb{E} \,\left[\,c_0(\bullet,\tilde{z} )^\circ(\bar{x}^\lambda;v)\,\right] \epc \mbox{by Clarke regularity of $c_0(\bullet,z)$
at $\bar{x}^{\lambda}$} \\ [0.1in]
& = & \left( \,\mathbb{E} \,\left[\, c_0(\bullet,\tilde{z} )\,\right] \,
\right)^\circ(\bar{x}^\lambda;v) \epc \mbox{by \cite[Theorem~7.68]{ShapiroDentchevaRuszczynski09}},
\end{array}
\]
the claimed B-stationary of $\bar{x}^{\lambda}$ for (\ref{eq:exp SP}) almost surely follows readily.
\end{proof}

\section{Sequential Sampling with Majorization}
\label{sec:MM}

In this section, we are interested in the combination of sequential sampling,
penalization (with variable penalty parameter) and upper surrogation
to solve the CCP in  \eqref{eq:focus sp_dc_constraint} via the restricted
(\ref{eq:restricted sp_dc_constraint}) and relaxed (\ref{eq:relaxed sp_dc_constraint})
problems.    {We propose an algorithm based on the unified
formulation (\ref{eq:new SP}) of the latter problems; we also recall the blanket assumptions
for (\ref{eq:new SP}).}
Closely related to majorization minimization that is the basis of
popular ``linearization'' algorithms for solving dc programs,
\cite{LeThiNgaiPham09,LeThiPham05,PhamDinhLeThi97}, the basic idea of surrogation
for solving a nonconvex nondifferentiable optimization problem is to derive upper
bounding functions of the functions involved, followed by the solution of
a sequence of subproblems by convex programming methods.  When this solution strategy
is applied to the problem (\ref{eq:new SP}), there are two most important
points to keep in mind:

\gap

\textcolor{black}{(a) Although in the context of the
relaxed $c_{k\ell}^{\rm rlx}(\bullet,z;\gamma)$ and restricted
$c_{k\ell}^{\rm rst}(\bullet,z;\gamma)$ functions, their unifications
$c_{k\ell}(\bullet,z;\gamma)$ are dc in theory, their practical dc decompositions are
not easily available for the purpose of computations (unless the indicator function is
relaxed/restricted by piecewise affine functions; see Lemma~\ref{lm:PA of PA}.)}

\gap

\textcolor{black}{(b) The resulting expectation
functions $\bar{c}_k(\bullet;\gamma)$
appear in the constraints; the standard dc approach as described in the cited references
would lump all such constraints into the objective via infinity-valued
indicator functions.  Even if the explicit dc representations of the constraint
functions are available, the resulting dc algorithm is at best a conceptual procedure
not readily implementable in practice.}

\gap

To address the former point---lack of explicit dc
representation, the extended idea of “surrogation” is used of which the dc-like
linearization is a special case. A comprehensive treatment of the ``surrogation
approach'' for solving nonconvex nondifferentiable optimization problems is detailed
in \cite[Chapter 7]{CuiPang2020}. To address the second
point---proper treatment of the chance constraints, we employ exact penalization
(i.e., finite value of the penalty parameter) with the aim of recovering solutions of
the original CCP (\ref{eq:new SP}); furthermore,
due to the nonconvexity of the functions involved, recovery is with reference
to stationary solutions instead of minimizers, as exemplified by the results in
Section~\ref{sec:exact penalization}.
When these considerations are combined with the need of sampling to handle the
expectation operator, the end result is the
{\bf S}ampling + {\bf P}enalization + {\bf S}urrogation {\bf A}lgorithm (SPSA) to be
introduced momentarily.
  {We remark that while the cited monograph and the reference \cite{PangRazaAlvarado16}
have discussed a solution approach for a deterministic dc program based on the linearization of the
constraints without penalization, in the context where sampling is needed, this direct treatment of constraints runs the risk of
infeasible sampled subproblems that is avoided by the penalization approach.}

\gap

For any given pair $( \bar{x},z ) \in X  \times \Xi$ and $k \in [ K ]$, we
let $\wh{c}_k(\bullet,z;\gamma;\bar{x})$ be a majorization of the function
$c_k(\bullet,z;\gamma)$ at $\bar{x}$ satisfying

\gap

(a) [B-differentiability] $\wh{c}_k(\bullet,z;\gamma;\bar{x})$
is B-differentiable on $X$;

\gap

(b) [upper surrogation] $\wh{c}_k(x,z;\gamma;\bar{x})\geq c_k(x,z;\gamma)$ for all
$x \in X$;

\gap
(c) [touching condition] $\wh{c}_k(\bar{x},z;\gamma;\bar{x})= c_k(\bar{x},z;\gamma)$;

\gap
(d) [upper semicontinuity] $\wh{c}_k(\bullet,z;\gamma;\bullet)$  is upper semicontinuous
on $X \times X$; and

\gap
(e) [directional derivative consistency]
$\wh{c}_k^{\,\prime}(\bullet,z;\gamma;\bar{x})(\bar{x};d) =
c_k^{\,\prime}(\bullet,z;\gamma)(\bar{x};d)$ for any $d\in \mathbb{R}^n$.

\gap

\textcolor{black}{Starting with the respective summands
$c_{k\ell}^{\rm rlx}(x,z;\gamma)$ and
$c_{k\ell}^{\rm rst}(x,z;\gamma)$ given in \eqref{defn: c_gamma}, there are several
ways to construct majorization functions for the relaxed
$c_k^{\rm rlx}(x,z;\gamma) \triangleq \displaystyle{
\sum_{\ell=1}^L
} \, c_{k\ell}^{\rm rlx}(x,z;\gamma)$ and restricted
$c_k^{\rm rst}(x,z;\gamma) \triangleq \displaystyle{
\sum_{\ell=1}^L
} \, c_{k\ell}^{\rm rst}(x,z;\gamma)$ functions that satisfy the conditions.}
Details can be found in Appendix~1; see also Subsection~\ref{subsec:gamma to zero}.
In what follows, we assume that
the surrogation functions $\wh{c}_k(x,z;\gamma;\bar{x})$ are given.
We also assume a similar surrogation function $\wh{c}_0(\bullet,z;\bar{x})$
of $c_0(\bullet,z)$ in the objective satisfying the same five conditions.  Denote
\begin{equation}  \label{eq:Vfunction}
\begin{array}{lll}
V_{\lambda}(x,Z^N;\gamma) & \triangleq & \displaystyle{
\frac{1}{N}
} \, \displaystyle{
\sum_{s=1}^N
} \, c_0(x,z^s) + \lambda \, \textcolor{black}{\displaystyle{
\sum_{k=1}^K
}} \, \max\left\{ \, \displaystyle{
\frac{1}{N}
} \, \displaystyle{
\sum_{s=1}^N
} \, c_k(x,z^s;\gamma) - \zeta_k, \, 0 \, \right\} \\ [0.3in]
\wh{V}_{\lambda}(x,Z^N;\gamma;\bar{x}) & \triangleq & \displaystyle{
\frac{1}{N}
} \, \displaystyle{
\sum_{s=1}^N
} \,  \wh{c}_0(x,z^s;\bar{x})+ \lambda \, \textcolor{black}{\displaystyle{
\sum_{k=1}^K
}} \, \max\left( \, \displaystyle{
\frac{1}{N}
} \, \displaystyle{
\sum_{s=1}^N
} \,\wh{c}_k(x,z^s;\gamma;\bar{x}) - \zeta_k, \, 0 \, \right) \\ [0.3in]
\wh{V}_{\lambda}^{\rho}(x,Z^N;\gamma;\bar{x}) & \triangleq &
\wh{V}_{\lambda}(x,Z^N;\gamma;\bar{x}) + \displaystyle{
\frac{\rho}{2}
} \, \| \, x - \bar{x} \, \|^2, \epc \mbox{for $\rho > 0$}.
\end{array}
\end{equation}
Notice that in the context of the relaxed/restricted functions
$c_{k\ell}^{\rm rlx/rst}(x,z;\gamma)$, the surrogate functions
$\wh{c}_k(x,z^s;\gamma;\bar{x})$ given in Appendix~1
are the pointwise minimum of finitely many convex functions.
Thus, a global minimizer of the problem
\begin{equation} \label{eq:proximal Vsubproblem}
\displaystyle{
\operatornamewithlimits{\mbox{\bf minimize}}_{x \in X}
} \ \wh{V}_{\lambda}^{\rho}(x,Z^N;\gamma;\bar{x})
\end{equation}
can be obtained by solving finitely many convex programs (see Appendix~2 for an
explanation how this is carried out).
This is an important practical aspect of the SPSA; namely, the iterates
can be constructively obtained by convex programming algorithms.

\gap

In the algorithm below, we present the version where the
subproblems (\ref{eq:proximal Vsubproblem}) are solved to global optimality
\textcolor{black}{without requiring the uniqueness of the minimizer}.  The
algorithm makes use of several sequences:
$\{ N_{\nu} \}_{\nu=1}^{\infty}$ (sample sizes),
$\{ \lambda_{\nu} \}_{\nu=1}^{\infty}$ (penalty parameters), $\{ \rho_{\nu} \}_{\nu=1}^{\infty}$
(proximal parameters), and $\{ \gamma_{\nu} \}_{\nu=1}^{\infty}$ (scaling factors),
as specified below:

\gap

\noindent $\bullet $ $\{N_\nu\}_{\nu=0}^{\infty}$: an increasing sequence of
positive integers with $N_0 = 0$;
each $N_\nu$ denotes the sample batch size at the $\nu$th iteration;

\gap

\noindent   {$\bullet $ $\{ \lambda_{\nu} \}_{\nu=1}^{\infty}$: a nondecreasing sequence
of positive scalars with $\lambda_1 = 1$
and $\displaystyle{
\lim_{\nu \to \infty}
} \, \lambda_{\nu} = \lambda_{\infty} \in [ \, 1,\infty \, ]$ (this includes both
a bounded and unbounded sequence);}

\gap

\noindent $\bullet $ $\{ \rho_{\nu} \}_{\nu=1}^{\infty}$: a sequence of positive
scalars with
$\displaystyle{
\lim_{\nu \to \infty}
} \, \rho_{\nu} = \rho_{\infty} \in [ \, 0, \infty \, )$ and such that for some
constants $\alpha_2 > \alpha_1 > 0$,
\begin{equation} \label{eq:ratio of rho and lambda}
\displaystyle{
\frac{\alpha_1}{\nu}
} \, \leq \, \displaystyle{
\frac{\rho_{\nu}}{\lambda_{\nu}}
} \, \leq \, \displaystyle{
\frac{\alpha_2}{\nu}
}, \epc \forall \, \nu;
\end{equation}
\noindent $\bullet $ $\{ \gamma_\nu \}_{\nu=1}^{\infty}$: a nonincreasing sequence
of positive values with $\displaystyle{
\lim_{\nu \to \infty}
} \, \gamma_\nu = \underline{\gamma} \geq 0$; moreover,
the sum
\begin{equation} \label{eq:gamma diff}
\displaystyle{
\sum_{\nu=1}^{\infty}
} \, \displaystyle{
\sup_{x \in X}
} \,\Bigg| \, \underbrace{\mathbb{E}\left[ \, c_k(x,\tilde z;\gamma_{\nu}) -
c_k(x,\tilde z;\gamma_{\nu-1}) \, \right]}_{\mbox{
$= \bar{c}_k(x,\gamma_{\nu}) - \bar{c}_k(x,\gamma_{\nu-1})$}} \Bigg|
\end{equation}
%
is finite.

\gap

Note that the condition (\ref{eq:ratio of rho and lambda}) implies:
\begin{equation} \label{eq:sum_rho_lambda}
\displaystyle{
\lim_{\nu \to \infty}
} \,  \displaystyle{
\frac{\rho_{\nu}}{\lambda_{\nu}}
} \, = \, 0 \epc \mbox{and} \epc \displaystyle{
\sum_{\nu=1}^{\infty}
} \,  \displaystyle{
\frac{\rho_{\nu}}{\lambda_{\nu}}
} \, \, = \, \infty.
\end{equation}
In particular, such condition permits $\{ \lambda_{\nu} \}$ to stay bounded while
$\{ \rho_{\nu} \} \downarrow 0$,
and also the opposite situation where
$\{ \rho_{\nu} \}$ is bounded away from zero while $\{ \lambda_{\nu} \} \to \infty$.
Condition (\ref{eq:gamma diff}) holds trivially if $\gamma_{\nu} = \gamma_{\nu-1}$ for
all $\nu$ sufficiently large, in particular, when the sequence $\{ \gamma_{\nu} \}$ is
a constant.  In the context of the restricted/relaxed approximations of the probability
constraints, we recall Proposition~\ref{pr:error of rst/rlx and acc} that yields, with
$\gamma_{\nu-1} \geq \gamma_{\nu}$,
\[ \begin{array}{lll}
\left| \, \bar{c}_k^{\, \rm rst/rlx}(x;\gamma_{\nu}) - \bar{c}_k^{\, \rm rst/rlx}(x;\gamma_{\nu-1}) \, \right|
\\ [0.1in]
\leq \,   {\mbox{Lip}_{\theta}} \, \displaystyle{
\sum_{\ell=1}^L
} \ | \, e_{k\ell} \, | \, \max\left( \, h_{{\cal Z}_{\ell}(x,\bullet)}^{\rm ub}(\gamma_{\nu}) - h_{{\cal Z}_{\ell}(x,\bullet)}^{\rm ub}(\gamma_{\nu-1}), \
h_{{\cal Z}_{\ell}(x,\bullet)}^{\rm lb}(\gamma_{\nu-1}) - h_{{\cal Z}_{\ell}(x,\bullet)}^{\rm lb}(\gamma_\nu) \, \right),
\end{array}
\]	
see (\ref{eq:h of Z}) for the definitions of the functions $h_Z^{\rm lb/ub}(\gamma)$ associated
with the random variable $Z \triangleq {\cal Z}_{\ell}(x,\bullet)$.
Hence, condition (\ref{eq:gamma diff}) holds if for all $\ell \in [ L ]$,
\[
\displaystyle{
\sum_{\nu=1}^{\infty}
} \, \displaystyle{
\sup_{x \in X}
} \, \max\left( \, h_{{\cal Z}_{\ell}(x,\bullet)}^{\rm ub}(\gamma_{\nu}) - h_{{\cal Z}_{\ell}(x,\bullet)}^{\rm ub}(\gamma_{\nu-1}), \
h_{{\cal Z}_{\ell}(x,\bullet)}^{\rm lb}(\gamma_{\nu-1}) - h_{{\cal Z}_{\ell}(x,\bullet)}^{\rm lb}(\gamma_{\nu}) \, \right) \, < \, \infty.
\]
Below we give an example to illustrate the above summability condition on the $\gamma$'s focusing on the case of a diminishing sequence.

\begin{example} \label{ex:gamma condition} \rm
Let ${\cal Z}(x,z) = \min( f(x)z,z + 1 )$ and $\tilde{z}$ be a random variable with the uniform distribution in the interval $( -2,2 )$;
let 
$f(X) \subseteq [ 2,a ]$ for some scalar $a > 2$.  Then, for $\gamma \leq 1$,
\[ \begin{array}{l}
h_{{\cal Z}(x,\bullet)}^{\, \rm lb}(\gamma) \, = \, \displaystyle{
\frac{1}{\gamma}
} \, \displaystyle{
\int_{\, 0}^{\, \gamma}
} \, \mathbb{P}\left( \min( f(x) \tilde{z}, \tilde{z} + 1 )  \leq t \, \right) \, dt \\ [0.2in]
= \, \displaystyle{
\frac{1}{\gamma}
} \, \displaystyle{
\int_{\, 0}^{\, \gamma}
} \, \mathbb{P}\left( \{ \, f(x) \tilde{z} \leq \tilde{z} + 1; \, f(x) \tilde{z} \leq t \, \} \, \cup \, \{ \, \tilde{z} + 1 \leq f(x) \tilde{z}; \,
\tilde{z} + 1 \leq t \, \}
\, \right) \, dt \\ [0.1in]
\, = \, \displaystyle{
\frac{1}{\gamma}
} \, \displaystyle{
\int_{\, 0}^{\, \gamma}
} \, \mathbb{P}\left( \left\{ \, \tilde{z} \leq \displaystyle{
\frac{1}{f(x) - 1}
}; \, \tilde{z} \leq \displaystyle{
\frac{t}{f(x)}
} \, \right\} \, \cup \, \underbrace{\left\{ \, \displaystyle{
\frac{1}{f(x) - 1}
} \, \leq \, \tilde{z} \leq t - 1 \, \right\}}_{\mbox{empty set}} \, \right) \, dt \\ [0.4in]
\, = \, \displaystyle{
\frac{1}{\gamma}
} \, \displaystyle{
\int_{\, 0}^{\, \gamma}
} \, \mathbb{P}\left( \left\{ \, \tilde{z} \leq \displaystyle{
\frac{t}{f(x)}
} \, \right\} \, \right) \, dt \epc \mbox{because $\displaystyle{
\frac{t}{f(x)}
} \leq \displaystyle{
\frac{\gamma}{f(x)}
} \leq \displaystyle{
\frac{1}{f(x) - 1}
}$} \\ [0.2in]
= \, \displaystyle{
\frac{1}{4 \gamma}
} \, \displaystyle{
\int_{\, 0}^{\, \gamma}
} \, \left( \, \displaystyle{
\frac{t}{f(x)}
} \, + 2 \, \right) \, dt \\ [0.2in]
= \, \displaystyle{
\frac{\gamma}{8 f(x)}
} + \displaystyle{
\frac{1}{2}
} \, .
\end{array} \]
Hence, for any nonincreasing sequence of positive scalars $\{ \gamma_{\nu} \}$ satisfying $\gamma_0 \leq 1$ and
$\underline{\gamma} \triangleq \displaystyle{
\lim_{\nu \to \infty}
} \, \gamma_{\nu}$, we have
\[
\displaystyle{
\sum_{\nu=1}^{\infty}
} \, \displaystyle{
\sup_{x \in X}
} \, \left[ \, h_{{\cal Z}(x,\bullet)}^{\, \rm lb}(\gamma_{\nu-1}) - h_{{\cal Z}(x,\bullet)}^{\, \rm lb}(\gamma_{\nu}) \, \right]
\, = \, \displaystyle{
\sum_{\nu=1}^{\infty}
} \, \displaystyle{
\frac{\gamma_{\nu-1} - \gamma_{\nu}}{8 \, \displaystyle{
\inf_{x \in X}
} \, f(x)}} \, = \,  \displaystyle{
\frac{\gamma_0 - \underline{\gamma}}{8 \, \displaystyle{
\inf_{x \in X}
} \, f(x)}} \, < \, \infty.
\]
Similarly, we have
\[ \begin{array}{l}
h_{{\cal Z}(x,\bullet)}^{\, \rm ub}(\gamma) \, = \, \displaystyle{
\frac{1}{\gamma}
} \, \displaystyle{
\int_{\, -\gamma}^{\, 0}
} \, \mathbb{P}\left( \min( f(x) \tilde{z}, \tilde{z} + 1 )  \leq t \, \right) \, dt \\ [0.1in]
\, = \, \displaystyle{
\frac{1}{\gamma}
} \, \displaystyle{
\int_{\, -\gamma}^{\, 0}
} \, \mathbb{P}\left( \left\{ \, \tilde{z} \leq \displaystyle{
\frac{1}{f(x) - 1}
}; \, \tilde{z} \leq \displaystyle{
\frac{t}{f(x)}
} \, \right\} \, \cup \, \underbrace{\left\{ \, \displaystyle{
\frac{1}{f(x) - 1}
} \, \leq \, \tilde{z} \leq t - 1 \, \right\}}_{\mbox{empty set}} \, \right) \, dt \\ [0.4in]
\, = \, \displaystyle{
\frac{1}{\gamma}
} \, \displaystyle{
\int_{\, -\gamma}^{\, 0}
} \, \mathbb{P}\left( \, \tilde{z} \leq \displaystyle{
\frac{t}{f(x)}
} \, \right) \, dt \, = \, = \, \displaystyle{
\frac{1}{4 \gamma}
} \, \displaystyle{
\int_{\, -\gamma}^{\, 0}
} \, \left( \, \displaystyle{
\frac{t}{f(x)}
} \, + 2 \, \right) \, dt \, = \, \displaystyle{
\frac{1}{2}
} - \displaystyle{
\frac{\gamma}{8 f(x)}
} \, ,
\end{array} \]
and the series $\displaystyle{
\sum_{\nu=1}^{\infty}
} \, \displaystyle{
\sup_{x \in X}
} \, \left[ \, h_{{\cal Z}(x,\bullet)}^{\, \rm ub}(\gamma_{\nu}) - h_{{\cal Z}(x,\bullet)}^{\, \rm ub}(\gamma_{\nu-1}) \, \right]$ is
also finite.   \hfill $\Box$
\end{example}

\noindent\makebox[\linewidth]{\rule{\textwidth}{1pt}}
\vspace{-0.1in}

\textbf{The SPSA: Global solution of subproblems and incremental sample batches}
\\

\vspace{-0.2in}

\noindent\makebox[\linewidth]{\rule{\textwidth}{1pt}}

\begin{algorithmic}[1]
\STATE \textbf{Initialization:}
Let the parameters $\left\{ \, N_{\nu}; \, \rho_{\nu}; \,
\gamma_{\nu}; \, \lambda_{\nu} \, \right\}_{\nu=1}^{\infty}$
be given.  Start with the empty sample batch $Z^0 = \emptyset$, $N_0=0$, and an
arbitrary $x^1 \in X$.

\gap

\FOR {$\nu= 1, 2, \cdots$}
\STATE generate samples $\{z^s\}_{s=N_{\nu-1}+1}^{N_{\nu}}$ independently from previous
samples, and add them to the present sample set $Z^{N_{\nu-1}}$
to obtain the new sample set $Z^{N_{\nu}}\, \triangleq \, Z^{N_{\nu-1}} \, \cup \,
\{ z^s \}_{s=N_{\nu-1}+1}^{N_{\nu}}$;
\STATE compute $x^{\nu+1} \ \textcolor{black}{\in} \ \displaystyle{
\operatornamewithlimits{\mbox{\bf argmin}}_{x \, \in \, X}
} \  \wh{V}_{\lambda_{\nu}}^{\rho_{\nu}}(x,Z^{N_{\nu}};\gamma_{\nu};x^{\nu})$;
\ENDFOR
\end{algorithmic}

\noindent\makebox[\linewidth]{\rule{\textwidth}{1pt}}
\vspace{-0.05in}

The convergence analysis of the SPSA
consists of two major parts: the first part
relies on some general properties of the functions $c_k(\bullet,z;\gamma)$ and their
majorizations $\wh{c}_k(\bullet,z,\gamma;\bar{x})$
(as described previously and summarized below)
and conditions on the sequences
$\{ (N_{\nu}, \lambda_{\nu}, \rho_{\nu}, \gamma_{\nu} \}_{\nu=1}^{\infty}$
(see Lemma~\ref{lm:condition on sample sizes} and
Proposition~\ref{pr:successive iterates}).   The second part is specific to an
accumulation point of the sequence produced by the
Algorithm and requires the applicability of some uniform law of large numbers (ULLN)
on the majorizing functions at the point.
This part is further divided into two cases:
a constant sequence with $\gamma_{\nu} = \underline{\gamma}$ for all $\nu$,
or a diminishing sequence with $\gamma_{\nu} \downarrow 0$.  The ULLN needed in the
former case is fairly straightforward.  The second
part requires more care as we need to deal with the limits of the majorizing functions
$\wh{c}_k(\bullet,z;\gamma_{\nu};x^{\nu})$
as $\nu \to \infty$.

\gap

A key tool in the convergence proof of the SPSA is a uniform bound for
the errors:
\begin{equation} \label{eq:errors}
\begin{array}{l}
\mathbb{E}\, \Bigg[\, \Bigg| \, \mathbb{E}[ c_0(x^{\nu},\tilde{z} ) ] -
\displaystyle{
\frac{1}{N_{\nu-1}}
} \, \displaystyle{
\sum_{s=1}^{N_{\nu-1}}
} \, c_0(x^{\nu},z^s) \, \Bigg| \,\Bigg], \\ [0.2in]
\mathbb{E}\, \Bigg[\, \Bigg| \, \mathbb{E}[ c_k(x^{\nu},\tilde{z};\gamma_{\nu-1} ) ] -
\displaystyle{
\frac{1}{N_{\nu-1}}
} \, \displaystyle{
\sum_{s=1}^{N_{\nu-1}}
} \, c_k(x^{\nu},z^s;\gamma_{\nu-1}) \, \Bigg| \,\Bigg], \epc k \in [ K ].
\end{array} \end{equation}
We derive
these bounds following the approach in \cite{ErmolievNorkin13} that is based on the
concept of Rademacher averages defined below.

\begin{definition} \label{df:Rademacher averages} \rm
For a given family of points $\boldsymbol{\xi}^N \triangleq \{ \xi_1, \cdots, \xi_N \}$
with each $\xi_i  \in \Xi$ and a sequence
of functions  $\{ \, f(\bullet, \xi_i) : X \to \mathbb{R} \, \}_{i=1}^N$,
the {\sl Rademacher average} $\mathbf{R}_N(f,\boldsymbol{\xi}^N)$ is defined as
\[
\mathbf{R}_N(f,\boldsymbol{\xi}^N) \, \triangleq \, \mathbb{E}_{\boldsymbol{\sigma}}
\left[ \, \sup_{x\in X} \, \Big| \, \frac{1}{N} \sum_{i=1}^N \sigma_i f(x,\xi_i) \, \Big|
\, \right],
\]
where $\sigma_i$ are i.i.d.\ random numbers such that $\sigma_i \in \{ \, +1, -1\, \}$
each with the probability 1/2 and $\mathbb{E}_{\boldsymbol{\sigma}}$ denotes the
expectation
over the random vector $\boldsymbol{\sigma} = (\sigma_1, \ldots, \sigma_N)$.
For the family of Carath\'eodory
functions $\{ f(\bullet,\xi): X \to \mathbb{R} \}_{\xi\in \Xi}$, the Rademacher average
is defined as
\[
\mathbf{R}_N(f,\Xi) \, \triangleq \, \sup_{\boldsymbol{\xi}^N \, \in \, \Xi^N} \,
\mathbf{R}_N(f,\boldsymbol{\xi}^N).
\]	
\end{definition}
The following simple lemma \cite[Theorem~3.1]{ErmolievNorkin13} facilitates the bound
of (\ref{eq:errors}) given upper bounds on the Rademacher averages.  The proof follows
from a straightforward application of the symmetrization
lemma \cite[Lemma~2.3.1]{vanderVaartWellner96}; see \cite[Appendix~C]{ErmolievNorkin13}.

\begin{lemma} \label{lm:Rademacher range implies Nemirovsky} \rm
Let $\{ \, f(\bullet, z^s) : X \to \mathbb{R} \, \}_{s=1}^N$ be arbitrary
Carath\'eodory functions.
For any $N > 0$ and any family
$Z^N \triangleq \{ z^s \}_{s=1}^N$ of i.i.d.\  samples of the
random variable $\tilde{z}$,
\[
\mathbb{E}\left[ \, \sup_{x \in X} \, \left| \displaystyle{
\frac{1}{N}
} \, \displaystyle{
\sum_{s=1}^N
} \, f(x,z^s) - \mathbb{E}[ f(x,\tilde{z}) ] \, \right| \, \right] \, \leq \, 2 \,
\mathbf{R}_N(f,\Xi).
\]
\end{lemma}

\fbox{
\parbox{6.5in}{
\begin{center}
{\bf Blanket assumptions on (\ref{eq:new SP}) for convergence of SPSA}
\end{center}

\textbf{Basic B-differentiability and other properties}: as described for the
problem (\ref{eq:new SP}), including   {the boundedness of $X$} and the
nonnegativity of the objective $c_0(\bullet,z)$; thus all the $V$-functions
defined in (\ref{eq:Vfunction}) are nonnegative.

\gap

\textbf{Growth of Rademacher averages}: there exist positive constants $W_0$,
$W_1$ and $W_2$ such that the Rademacher averages of the objective function
$c_0(\bullet,\tilde{z} )$ and the
constraint functions $c_k(\bullet,\tilde{z};\bullet)$ satisfy, for all integers $N > 0$
and all exponents $\beta \in ( 0, 1/2 )$

\gap

$\bullet $ $\mathbf{R}_N(c_0,\Xi) \leq
\displaystyle{
\frac{W_0}{N^{\beta}}
}$; and

\gap

$\bullet $ $\displaystyle{
\max_{k \in [ K ]}
} \, \mathbf{R}_N(c_k(\bullet,\bullet;\gamma),\Xi) \leq \displaystyle{
\frac{W_1}{N^{\beta}}
}
+ \displaystyle{
\frac{W_2}{\gamma}
} \, \displaystyle{
\frac{1}{\sqrt{N}}
}$ for all $\gamma > 0$.  \hfill $\Box$
}}

\gap

The growth conditions of the Rademacher averages imposed above are
essentially assumption B(iii) in \cite{ErmolievNorkin13} where there is a discussion
with proofs of various common cases for the satisfaction of the conditions.  Most
relevant to us is Lemma~B.2 therein that explains both the exponent $\beta$ and
the fraction $1/\gamma$.  In particular, $\beta$ is used to upper bound a term
$\sqrt{\ln N/N}$ and thus can be somewhat flexible.  Nevertheless, the ``constants''
in the numerators of the bounds of the Rademacher averages
$\mathbf{R}_N(c_k(\bullet,\bullet;\gamma),\Xi)$ depend on two things:
(i) the uniform boundedness of the functions $c_k(\bullet,\bullet;\gamma)$ on $X \times \Xi$
by a constant independent of $\gamma$ (Assumption B(i) in \cite{ErmolievNorkin13}) and (ii) the linear dependence on
the Lipschitz modulus of the function $c_k(\bullet,z;\gamma)$, among other constants (Lemma B.4 in the reference).
In the context of the relaxed/restricted
functions $c_k^{\rm rlx/rst}(x,z;\gamma) = \displaystyle{
\sum_{\ell=1}^L
} \, c_{k\ell}^{\rm rlx/rst}(x,z;\gamma)$, these summands are affine combinations of
$\phi_{\rm lb/ub}({\cal Z}_{\ell}(\bullet,z),\gamma)$ which are bounded between 0 and 1;
moreover, by their definitions, the functions $c_k^{\rm rlx/rst}(\bullet,z;\gamma)$,
are Lipschitz continuous with modulus $\mbox{Lip}_c(z)/\gamma$; cf.\ (\ref{eq:Lip in x}).
This explains the term $1/\gamma$ in the numerator of the bound of $\mathbf{R}_N(c_k(\bullet,\bullet;\gamma),\Xi)$.
The reason to expose this fraction is for the analysis of the case where the sequence
$\{ \gamma_{\nu} \} \downarrow 0$.  Knowing
how the Rademacher bound depends on $\gamma$ leads to conditions on the decay of this
sequence to ensure convergence of the SPSA; see the proof of Proposition~\ref{pr:successive iterates}
that makes use of Lemma~\ref{lm:condition on sample sizes}.

\gap

Before moving to the next subsection, we state a (semi)continuous convergence result
of random functionals.
This result is drawn from \cite[Theorem 2.3]{ArtsteinWets95}; see also
\cite[Theorem~7.48]{ShapiroDentchevaRuszczynski09} where continuity is assumed.
For ease of reference, we state the result pertaining to a given vector $\bar{x}$.

\begin{proposition} \label{pr:uniform pointwise ULLN} \rm
Let $c(\bullet,z) :  Y \to \mathbb{R}$ be semicontinuous in a neighborhood ${\cal N}$
of a vector $\bar{x}$ in the open set $Y \subseteq \mathbb{R}^n$.
Suppose that $c(x,\bullet)$ is dominated by an integrable function for any
$x \in {\cal N}$.  For any
sequence $\{ x^N \}_{N=1}^{\infty}$ converging to $\bar{x}$, and any i.i.d.\ samples
$\{ z^s \}_{s=1}^N$, it holds that

\gap

$\bullet $ if $c(\bullet,z)$ is lower semicontinuous in ${\cal N}$, then
\[ \displaystyle{
\liminf_{N \to \infty}
} \, \, \displaystyle{
\frac{1}{N}
} \, \displaystyle{
\sum_{s=1}^N
} \, c(x^N,z^s) - \mathbb{E}[ c(\bar{x},\tilde{z} ) ] \, \geq  \, 0 \epc
\mbox{almost surely};
\]
$\bullet $ if $c(\bullet,z)$ is upper semicontinuous in ${\cal N}$, then
\[ \displaystyle{
\limsup_{N \to \infty}
} \, \, \displaystyle{
\frac{1}{N}
} \, \displaystyle{
\sum_{s=1}^N
} \, c(x^N,z^s) - \mathbb{E}[ c(\bar{x},\tilde{z} ) ] \, \leq  \, 0 \epc
\mbox{almost surely}.
\]
\end{proposition}

\subsection{Convergence analysis: preliminary results}

We are now ready to begin the proof of convergence of the SPSA.
We first establish a lemma that
provides a practical guide for the selection of the sample sizes $N_{\nu}$.

\begin{lemma} \label{lm:condition on sample sizes} \rm
For the sequence of positive integers $\{ N_{\nu} \}$, a scalar $\beta \in ( 0, 1/2 )$,
and the positive sequence $\{ \gamma_{\nu} \}$,
suppose that there exist
a positive integer $\bar{\nu}$ and positive scalars $\delta$ and $\{ c_i \}_{i=1}^4$
with $c_3 < \bar{\nu}$ and $\beta ( 1 + c_1 ) > 1 + \delta$ such that
\[
c_2 \, \nu^{\, 1 + c_1} 
\, \leq \, N_{\nu} \, \leq \, \displaystyle{
\frac{N_{\nu-1}}{\left( 1 - \displaystyle{
\frac{c_3}{\nu}
} \, \right)}
} \ \mbox{ and } \ \gamma_{\nu} \, \geq \, \displaystyle{
\frac{c_4}{\nu^{\, \delta}}
} \epc \epc \forall \, \nu \, \geq \, \bar{\nu}.
\]
Then the following six series are finite:
\[ \begin{array}{lll}
S_1 \, \triangleq \, \displaystyle{
\sum_{\nu=1}^{\infty}
} \, \displaystyle{
\frac{N_{\nu} - N_{\nu-1}}{N_{\nu}}
} \, \displaystyle{
\frac{1}{N_{\nu-1}^{\, \beta}}
};  &
S_2 \, \triangleq \, \displaystyle{
\sum_{\nu=1}^{\infty}
} \, \displaystyle{
\frac{1}{N_{\nu}^{\beta}}
}; &
S_3 \, \triangleq \, \displaystyle{
\sum_{\nu=1}^{\infty}
} \, \displaystyle{
\frac{( \, N_{\nu} - N_{\nu-1} \, )^{1 - \beta}}{N_{\nu}}
}; \\ [0.3in]
S_4\, \triangleq \,
\displaystyle{
\sum_{\nu=1}^{\infty}
} \, \displaystyle{
\frac{N_{\nu} - N_{\nu-1}}{N_{\nu}}
} \, \displaystyle{
\frac{1}{N_{\nu-1}^{\beta}}
} \, \displaystyle{
\frac{1}{\gamma_{\nu-1}}
}; &
S_5 \, \triangleq \, \displaystyle{
\sum_{\nu=1}^{\infty}
} \, \displaystyle{
\frac{1}{N_{\nu}^{\beta} \, \gamma_{\nu}}
}; &
S_6 \, \triangleq \, \displaystyle{
\sum_{\nu=1}^{\infty}
} \, \displaystyle{
\frac{( \, N_{\nu} - N_{\nu-1} \, )^{1 - \beta}}{N_{\nu}}
} \, \displaystyle{
\frac{1}{\gamma_{\nu-1}}
} \, .
\end{array}
\]
are all finite.
\end{lemma}

\begin{proof}  For any $\nu \geq \bar{\nu} + 1$, we have
\[
\textcolor{black}{N_{\nu} \, \leq \, \displaystyle{
\frac{\nu \, N_{\nu-1}}{\nu - c_3}
} \, \leq \, \displaystyle{
\frac{\nu \, N_{\nu-1}}{1 + \bar{\nu} - c_3}
} \, \leq \, \nu \, N_{\nu-1}}.
\]
Hence, $N_{\nu} - N_{\nu-1} \leq c_3 \, \displaystyle{
\frac{N_{\nu}}{\nu}
} \, \leq \, c_3 \, N_{\nu-1}$ for any $\nu \geq \bar{\nu}+1$.  Thus,
\[
S_1 \, \leq \, S_3 \epc \mbox{and} \epc S_4 \, \leq \, S_6.
\]
Since
\[
\begin{array}{rll}
\displaystyle{
\frac{( \, N_{\nu} - N_{\nu-1} \, )^{1 - \beta}}{N_{\nu}}
} & = & \left( \displaystyle{
\frac{N_{\nu} - N_{\nu-1}}{N_{\nu}}
} \right)^{1 - \beta} \, \displaystyle{
\frac{1}{N_{\nu}^{\, \beta}}
} \\ [0.2in]
& \leq & \left( \displaystyle{
\frac{c_3}{\nu}
} \right)^{1 - \beta} \, \displaystyle{
\frac{1}{( \, c_2 \, \nu^{1 + c_1} \, )^{\beta}}
} \, = \, \displaystyle{
\frac{c_3^{1 - \beta}}{c_2^{\, \beta}}
} \, \displaystyle{
\frac{1}{\nu^{\, 1 + c_1 \beta}}
}; \\ [0.2in]
\mbox{and} \epc
\displaystyle{
\frac{1}{\gamma_{\nu-1}}
} \, \displaystyle{
\frac{( \, N_{\nu} - N_{\nu-1} \, )^{1 - \beta}}{N_{\nu}}
} & \leq & \displaystyle{
\frac{c_3^{1 - \beta}}{c_4 \, c_2^{\, \beta}}
} \, \displaystyle{
\frac{1}{\nu^{\, 1 + c_1 \beta -\delta}}
}, \epc \mbox{because $\gamma_{\nu-1} \geq \gamma_{\nu} \geq \displaystyle{
\frac{c_4}{\nu^{\, \delta}}
}$}
\end{array}
\]
and by assumption, $\beta ( 1 + c_1 ) > 1 + \delta$, which implies $1 + c_1 \beta -\delta > 2 - \beta > 1.5$, it follows that
the sums $S_1$, $S_3$, $S_4$, and $S_6$ are finite.  Finally, we have
\[
\displaystyle{
\frac{1}{N_{\nu}^{\beta}}
} \, \leq \, \displaystyle{
\frac{1}{c_2^{\beta} \, \nu^{\beta ( 1 + c_1 )}}
} \epc \mbox{and} \epc \displaystyle{
\frac{1}{N_{\nu}^{\beta} \, \gamma_{\nu}}
} \, \leq \, \displaystyle{
\frac{1}{c_2^{\beta} \, \nu^{\beta ( 1 + c_1 ) - \delta}}
} .
\]
Thus the remaining two sums $S_2$ and $S_5$ are finite too.
\end{proof}

Based on the above lemma, we next prove a preliminary result for the sequence
$\{ x^{\nu} \}$ of iterates produced by the SPSA.  Notice that the
proposition does not assume any limiting condition on the
sequence of penalty parameters $\{ \lambda_{\nu} \}$.

\begin{proposition} \label{pr:successive iterates} \rm
Under the blanket assumptions set forth above for the
problem \eqref{eq:new SP} and the assumptions on
$\{ N_{\nu}, \rho_{\nu}, \gamma_{\nu}, \lambda_{\nu} \}$,
  {including those in Lemma~\ref{lm:condition on sample sizes}},
if $\{ x^{\nu} \}$ is any sequence produced by the SPSA, then
the sum $\displaystyle{
\sum_{\nu=1}^{\infty}
} \, \displaystyle{
\frac{\rho_{\nu}}{\lambda_{\nu}}
} \, \| \, x^{\nu+1} - x^{\nu} \, \|^2_2$
is finite with probability one.  \hfill $\Box$
\end{proposition}

\begin{proof}  Based on the main iteration in the SPSA, we have
\begingroup
\allowdisplaybreaks
\begin{align*}
& \;\displaystyle{
\frac{1}{\lambda_{\nu}}
} \, V_{\lambda_{\nu}}(x^{\nu+1}, Z^{N_{\nu}};\gamma_{\nu}) + \displaystyle{
\frac{\rho_{\nu}}{2 \, \lambda_{\nu}}
} \, \| \, x^{\nu+1} - x^{\nu} \, \|^2_2 \\[0.05in]
\leq &\;\, \displaystyle{
\frac{1}{\lambda_{\nu}}
} \, \wh{V}_{\lambda_{\nu}}(x^{\nu+1},Z^{N_{\nu}};\gamma_{\nu};x^{\nu}) + \displaystyle{
\frac{\rho_{\nu}}{2 \, \lambda_{\nu}}
} \, \| \, x^{\nu+1} - x^{\nu} \, \|^2 \hspace{0.3in} \mbox{by majorization} \\[0.05in]
\leq &\,\; \displaystyle{
\frac{1}{\lambda_{\nu}}
} \, \wh{V}_{\lambda_{\nu}}(x^{\nu},Z^{N_{\nu}};\gamma_{\nu};x^{\nu}) \hspace{0.8in}
\mbox{by the optimality of $x^{\nu+1}$} \\[0.05in]
= &\;\, \displaystyle{
\frac{1}{\lambda_{\nu}}
} \, V_{\lambda_{\nu}}(x^{\nu},Z^{N_{\nu}};\gamma_{\nu}) \hspace{0.7in}
\mbox{by the touching property of the majorization} \\
= &\;\, \displaystyle{
\frac{1}{N_{\nu}}
} \, \displaystyle{
\frac{1}{\lambda_{\nu}}
} \, \, \displaystyle{
\sum_{s=1}^{N_{\nu}}
} \, c_0(x^{\nu},z^s) + \textcolor{black}{\displaystyle{
\sum_{k=1}^K
}} \, \max\left\{ \, \displaystyle{
\frac{1}{N_{\nu}}
} \, \displaystyle{
\sum_{s=1}^{N_{\nu}}
} \, c_k(x^{\nu},z^s;\gamma_{\nu}) - \zeta_k, \, 0 \, \right\} \hspace{0.2in}
\mbox{by definition} \\
\leq &\;\, \displaystyle{
\frac{1}{N_{\nu-1} \, \lambda_{\nu-1}}
} \, \displaystyle{
\sum_{s=1}^{N_{\nu-1}}
} \, c_0(x^{\nu},z^s) + \, \delta_{1, \nu} + \textcolor{black}{\displaystyle{
\sum_{k=1}^K
}} \, \max\left\{ \, \displaystyle{
\frac{1}{N_{\nu-1}}
} \, \displaystyle{
\sum_{s=1}^{N_{\nu-1}}
} \, c_k(x^{\nu},z^s;\gamma_{\nu-1}) - \zeta_k, \, 0 \, \right\} \\
\epc & \;\, + \textcolor{black}{\displaystyle{
\sum_{k=1}^K
}} \, \max\left\{ \, \displaystyle{
\frac{1}{N_{\nu}}
} \, \displaystyle{
\sum_{s=1}^{N_{\nu}}
} \, c_k(x^{\nu},z^s;\gamma_{\nu}) - \zeta_k, \, 0 \, \right\} -
\textcolor{black}{\displaystyle{
\sum_{k=1}^K
}} \, \max\left\{ \, \displaystyle{
\frac{1}{N_{\nu-1}}
} \, \displaystyle{
\sum_{s=1}^{N_{\nu-1}}
} \, c_k(x^{\nu},z^s;\gamma_{\nu-1}) - \zeta_k, \, 0 \, \right\} \\
& \epc \textcolor{black}{
\mbox{by   {$\lambda_{\nu} \geq \lambda_{\nu-1}$}, nonnegativity of $c_0$, definition of $\delta_{1,\nu}$
and adding and subtracting terms}} \\
= & \;\, \displaystyle{
\frac{1}{\lambda_{\nu-1}}
} \, V_{\lambda_{\nu-1}}(x^{\nu},Z^{N_{\nu-1}};\gamma_{\nu-1}) 
+ \delta_{1,\nu} + \delta_{2,\nu} + \delta_{3,\nu}, \\
& \epc \textcolor{black}{
\mbox{by definition of $V_{\lambda_{\nu-1}}(x^{\nu},Z^{N_{\nu-1}};\gamma_{\nu-1})$
and the definitions of the $\delta$-terms below}}
\end{align*}
\endgroup
where
\[
\begin{array}{ll}
\delta_{1,\nu} \triangleq   {\displaystyle{
\frac{1}{\lambda_{\nu-1}}
} \, \Bigg( \, \displaystyle{
\frac{1}{N_{\nu}}
} \, \displaystyle{
\sum_{s=1}^{N_{\nu}}
} \, c_0(x^{\nu},z^s) -  \displaystyle{
\frac{1}{N_{\nu-1}}
} \, \displaystyle{
\sum_{s=1}^{N_{\nu-1}}
} \, c_0(x^{\nu},z^s) \, \Bigg)},  \\ [0.25in]
\delta_{2,\nu} \triangleq \textcolor{black}{\displaystyle{
\sum_{k=1}^K
}} \, \left[ \, \max\Bigg\{ \, \displaystyle{
\frac{1}{N_{\nu}}
} \, \displaystyle{
\sum_{s=1}^{N_{\nu}}
} \, c_k(x^{\nu},z^s;\gamma_{\nu}) - \zeta_k, \, 0 \, \Bigg\} -
\max\Bigg\{ \, \displaystyle{
\frac{1}{N_{\nu}}
} \, \displaystyle{
\sum_{s=1}^{N_{\nu}}
} \, c_k(x^{\nu},z^s;\gamma_{\nu-1}) - \zeta_k, \, 0 \, \Bigg\} \, \right] \\ [0.25in]
\delta_{3,\nu} \triangleq \textcolor{black}{\displaystyle{
\sum_{k=1}^K
}} \, \left[ \, \max\Bigg\{ \, \displaystyle{
\frac{1}{N_{\nu}}
} \, \displaystyle{
\sum_{s=1}^{N_{\nu}}
} \, c_k(x^{\nu},z^s;\gamma_{\nu-1}) - \zeta_k, \, 0 \, \Bigg\} -
\max\Bigg\{ \, \displaystyle{
\frac{1}{N_{\nu-1}}
} \, \displaystyle{
\sum_{s=1}^{N_{\nu-1}}
} \, c_k(x^{\nu},z^s;\gamma_{\nu-1}) - \zeta_k, \, 0 \, \Bigg\} \, \right].
\end{array} \]
Therefore taking conditional expectation with respect to the
$\sigma$-algebra ${\cal F}^{\, \nu-1}$
generated by the family $Z^{N_{\nu-1}}$ of random samples up to iteration $\nu-1$,
we have
\[ \begin{array}{l}
\mathbb{E}\left[ \, \displaystyle{
\frac{1}{\lambda_{\nu}}
} \, V_{\lambda_{\nu}}(x^{\nu+1},Z^{N_{\nu}};\gamma_{\nu}) + \displaystyle{
\frac{\rho_\nu}{2 \lambda_{\nu}}
} \, \| \, x^{\nu+1} - x^{\nu} \, \|^2 \, | \, {\cal F}^{\nu-1} \, \right] \\ [0.2in]
\leq \, \displaystyle{
\frac{1}{\lambda_{\nu-1}}
} \,  V_{\lambda_{\nu-1}}(x^{\nu},Z^{N_{\nu-1}};\gamma_{\nu-1}) +
\abs{ \, \mathbb{E}\left[ \, \delta_{1,\nu} \, | \, {\cal F}^{\nu-1} \, \right]\, }+
\abs{\, \mathbb{E}\left[ \, \delta_{2,\nu} \, | \, {\cal F}^{\nu-1} \, \right] \,}+
\abs{\, \mathbb{E}\left[ \, \delta_{3,\nu} \, | \, {\cal F}^{\nu-1} \, \right]\,}.  
\end{array} \]
We next evaluate each error term individually.  Since
  {
\[ \begin{array}{l}
\displaystyle{
\frac{1}{N_{\nu}}
} \, \displaystyle{
\sum_{s=1}^{N_{\nu}}
} \, c_0(x^{\nu},z^s) -  \displaystyle{
\frac{1}{N_{\nu-1}}
} \, \displaystyle{
\sum_{s=1}^{N_{\nu-1}}
} \, c_0(x^{\nu},z^s) \\ [0.2in]
= \, \displaystyle{
\frac{1}{N_{\nu}}
} \, \displaystyle{
\sum_{s=1}^{N_{\nu-1}}
} \, c_0(x^{\nu},z^s) + \displaystyle{
\frac{1}{N_{\nu}}
} \, \displaystyle{
\sum_{s=N_{\nu-1}+1}^{N_{\nu}}
} \, c_0(x^{\nu},z^s) -  \displaystyle{
\frac{1}{N_{\nu-1}}
} \, \displaystyle{
\sum_{s=1}^{N_{\nu-1}}
} \, c_0(x^{\nu},z^s) \\ [0.2in]
= \, \left( \, 1 - \displaystyle{
\frac{N_{\nu-1}}{N_{\nu}}
} \, \, \right) \, \left( \, \displaystyle{
\frac{1}{N_{\nu} - N_{\nu-1}}
} \, \displaystyle{
\sum_{s=N_{\nu-1}+1}^{N_{\nu}}
} \, c_0(x^{\nu},z^s) - \displaystyle{
\frac{1}{N_{\nu-1}}
} \, \displaystyle{
\sum_{s=1}^{N_{\nu-1}}
} \, c_0(x^{\nu},z^s) \, \right),
\end{array} \]}
and $\{ z^s \}$ are i.i.d.\ samples of $\wt{z}$, we deduce
\[ \begin{array}{lll}
\mathbb{E}\left[ \, \left| \, \mathbb{E}\left[ \, \delta_{1,\nu} \, | \,
{\cal F}^{\nu-1} \, \right] \, \right| \, \right]
& = & \displaystyle{
\frac{N_{\nu} - N_{\nu-1}}{N_{\nu} \, \lambda_{\nu-1}}
} \, \, \mathbb{E}\, \Bigg[\, \Bigg| \, \mathbb{E}[ c_0(x^{\nu},\tilde{z} ) ] -
\displaystyle{
\frac{1}{N_{\nu-1}}
} \, \displaystyle{
\sum_{s=1}^{N_{\nu-1}}
} \, c_0(x^{\nu},z^s) \, \Bigg| \,\Bigg]\\ [0.2in]
& \leq & \displaystyle{
\frac{N_{\nu} - N_{\nu-1}}{N_{\nu} \, \lambda_{\nu-1}}
} \, \, \mathbb{E}\, \Bigg[\, \displaystyle{
\sup_{x \in X}
} \, \Bigg| \, \mathbb{E}[ c_0(x,\tilde{z} ) ] -
\displaystyle{
\frac{1}{N_{\nu-1}}
} \, \displaystyle{
\sum_{s=1}^{N_{\nu-1}}
} \, c_0(x,z^s) \, \Bigg| \,\Bigg]\\ [0.25in]
& \leq & \displaystyle{
\frac{N_{\nu} - N_{\nu-1}}{N_{\nu} \, \lambda_{\nu-1}}
} \, \displaystyle{
\frac{2 \, W_0}{N_{\nu-1}^{\, \beta}}
} \, \leq \, \displaystyle{
\frac{N_{\nu} - N_{\nu-1}}{N_{\nu}}
} \, \displaystyle{
\frac{2 \, W_0}{N_{\nu-1}^{\, \beta}}
}\, ,
\end{array} \]
where the   {second} inequality follows from the growth assumption of the Rademacher
averages of the objective
function  and the   {third} inequality holds because $\lambda_{\nu-1} \geq 1$.
For the second error term $\delta_{2, \nu}$,  we have
\[ \begin{array}{lll}
\left| \, \delta_{2,\nu} \, \right| & \leq & \displaystyle{
\sum_{k=1}^K
} \, \left| \, \begin{array}{l}
\max\Bigg\{ \, \displaystyle{
\frac{1}{N_{\nu}}
} \, \displaystyle{
\sum_{s=1}^{N_{\nu}}
} \, c_k(x^{\nu},z^s;\gamma_{\nu}) - \zeta_k, \, 0 \, \Bigg\} \\ [0.25in]
- \, \max\Bigg\{ \, \displaystyle{
\frac{1}{N_{\nu}}
} \, \displaystyle{
\sum_{s=1}^{N_{\nu}}
} \, c_k(x^{\nu},z^s;\gamma_{\nu-1}) - \zeta_k, \, 0 \, \Bigg\}
\end{array} \right| \\ [0.55in]
& \leq & \displaystyle{
\sum_{k=1}^K
} \ \Bigg| \, \displaystyle{
\frac{1}{N_{\nu}}
} \, \displaystyle{
\sum_{s=1}^{N_{\nu}}
} \, c_k(x^{\nu},z^s;\gamma_{\nu}) - \displaystyle{
\frac{1}{N_{\nu}}
} \, \displaystyle{
\sum_{s=1}^{N_{\nu}}
} \, c_k(x^{\nu},z^s;\gamma_{\nu-1}) \, \Bigg| \\ [0.3in]
& \leq & \displaystyle{
\sum_{k=1}^K
} \, \left[ \, \begin{array}{l}
\Bigg| \, \mathbb{E}[ c_k(x^{\nu},\tilde z;\gamma_{\nu}) ] - \displaystyle{
\frac{1}{N_{\nu}}
} \, \displaystyle{
\sum_{s=1}^{N_{\nu}}
} \, c_k(x^{\nu},z^s;\gamma_{\nu}) \, \Bigg| \, + \\ [0.2in]
\Bigg| \,  \mathbb{E}[ c_k(x^{\nu},\tilde z;\gamma_{\nu-1}) ] - \displaystyle{
\frac{1}{N_{\nu}}
} \, \displaystyle{
\sum_{s=1}^{N_{\nu}}
} \, c_k(x^{\nu},z^s;\gamma_{\nu-1}) \, \Bigg| \, + \\ [0.2in]
\Bigg| \, \mathbb{E}[ c_k(x^{\nu},\tilde z;\gamma_{\nu}) ] -
\mathbb{E}[ c_k(x^{\nu},\tilde z;\gamma_{\nu-1}) ] \, \Bigg|
\end{array} \right].
\end{array} \]
Consequently, since $1/\gamma_{\nu} \geq 1/\gamma_{\nu-1}$, we deduce
\[ \begin{array}{l}
\mathbb{E}\left[ \, \left| \, \mathbb{E}\left[ \, \delta_{2,\nu} \,
| \, {\cal F}^{\nu-1} \, \right] \, \right| \, \right] \, \leq \,
\displaystyle{
\sum_{k=1}^K
} \, \mathbb{E}\left[ \, \left[ \, \begin{array}{l}
\displaystyle{
\sup_{x \in X}
} \, \Bigg| \, \mathbb{E}[ c_k(x,\tilde z;\gamma_{\nu}) ] - \displaystyle{
\frac{1}{N_{\nu}}
} \, \displaystyle{
\sum_{s=1}^{N_{\nu}}
} \, c_k(x,z^s;\gamma_{\nu}) \, \Bigg| \, + \\ [0.2in]
\displaystyle{
\sup_{x \in X}
} \, \Bigg| \,  \mathbb{E}[ c_k(x,\tilde z;\gamma_{\nu-1}) ] - \displaystyle{
\frac{1}{N_{\nu}}
} \, \displaystyle{
\sum_{s=1}^{N_{\nu}}
} \, c_k(x,z^s;\gamma_{\nu-1}) \, \Bigg| \, + \\ [0.2in]
\displaystyle{
\sup_{x \in X}
} \,\Bigg| \, \mathbb{E}[ c_k(x,\tilde z;\gamma_{\nu}) ] -
\mathbb{E}[ c_k(x,\tilde z;\gamma_{\nu-1}) ] \, \Bigg|
\end{array} \right] \, \right] \\ [0.8in]
\epc \leq \, \displaystyle{
\frac{2 K}{N_{\nu}^{\beta}}
} \, \left( \, W_1 + \displaystyle{
\frac{W_2}{\gamma_{\nu}}
} \, \right) + \displaystyle{
\sum_{k=1}^K
} \ \displaystyle{
\sup_{x \in X}
} \, \Bigg| \,  \mathbb{E}[ c_k(x,\tilde z;\gamma_{\nu}) - c_k(x,\tilde z;\gamma_{\nu-1}) ] \, \Bigg|,
\end{array} \]
where for simplicity, we have used the fact that $\sqrt{N_{\nu}} \geq N_{\nu}^{\beta}$.
Regarding the third error term $\delta_{3, \nu}$,  we have
\[ \begin{array}{l}
\left| \, \delta_{3,\nu} \, \right| \, \leq \, \textcolor{black}{\displaystyle{
\sum_{k=1}^K
}} \, \left| \, \begin{array}{l}
\max\Bigg\{ \, \displaystyle{
\frac{1}{N_{\nu}}
} \, \displaystyle{
\sum_{s=1}^{N_{\nu}}
} \, c_k(x^{\nu},z^s;\gamma_{\nu-1}) - \zeta_k, \, 0 \, \Bigg\} \\ [0.25in]
- \, \max\Bigg\{ \, \displaystyle{
\frac{1}{N_{\nu-1}}
} \, \displaystyle{
\sum_{s=1}^{N_{\nu-1}}
} \, c_k(x^{\nu},z^s;\gamma_{\nu-1}) - \zeta_k, \, 0 \, \Bigg\}
\end{array} \right| \\ [0.55in]
\, \leq \, \textcolor{black}{\displaystyle{
\sum_{k=1}^K
}} \ \Bigg| \, \displaystyle{
\frac{1}{N_{\nu}}
} \, \displaystyle{
\sum_{s=1}^{N_{\nu}}
} \, c_k(x^{\nu},z^s;\gamma_{\nu-1}) - \displaystyle{
\frac{1}{N_{\nu-1}}
} \, \displaystyle{
\sum_{s=1}^{N_{\nu-1}}
} \, c_k(x^{\nu},z^s;\gamma_{\nu-1}) \, \Bigg| \\ [0.3in]
\, \leq \, \displaystyle{
\frac{N_{\nu} - N_{\nu-1}}{N_{\nu}}
} \, \textcolor{black}{\displaystyle{
\sum_{k=1}^K
}} \, \left[ \, \begin{array}{l}
\Bigg| \, \mathbb{E}[ c_k(x^{\nu},\tilde z;\gamma_{\nu-1}) ] - \displaystyle{
\frac{1}{N_{\nu-1}}
} \, \displaystyle{
\sum_{s=1}^{N_{\nu-1}}
} \, c_k(x^{\nu},z^s;\gamma_{\nu-1}) \, \Bigg| \\ [0.2in]
+ \ \Bigg| \, \displaystyle{
\frac{1}{N_{\nu} - N_{\nu-1}}
} \, \displaystyle{
\sum_{s=N_{\nu-1}+1}^{N_{\nu}}
} \, c_k(x^{\nu},z^s;\gamma_{\nu-1}) - \mathbb{E}[ c_k(x^{\nu},\tilde{z};
\gamma_{\nu-1}) ] \, \Bigg| 
\end{array} \, \right]. 
\end{array} \]
Consequently,
\[ \begin{array}{l}
\mathbb{E}\left[ \, \left| \, \mathbb{E}\left[ \, \delta_{3,\nu} \,
| \, {\cal F}^{\nu-1} \, \right] \, \right| \, \right] \\ [0.2in]
\leq \, \displaystyle{
\frac{N_{\nu} - N_{\nu-1}}{N_{\nu}}
} \, \textcolor{black}{\displaystyle{
\sum_{k=1}^K
}} \ \mathbb{E}\left[ \, \begin{array}{l}
\displaystyle{
\sup_{x \in X}
} \, \Bigg| \, \mathbb{E}[ c_k(x,\tilde z;\gamma_{\nu-1}) ] - \displaystyle{
\frac{1}{N_{\nu-1}}
} \, \displaystyle{
\sum_{s=1}^{N_{\nu-1}}
} \, c_k(x,z^s;\gamma_{\nu-1}) \, \Bigg| \\ [0.2in]
+ \ \displaystyle{
\sup_{x \in X}
} \, \Bigg| \, \displaystyle{
\frac{1}{N_{\nu} - N_{\nu-1}}
} \, \displaystyle{
\sum_{s=N_{\nu-1}+1}^{N_{\nu}}
} \, c_k(x,z^s;\gamma_{\nu-1}) - \mathbb{E}[ c_k(x,\tilde{z};
\gamma_{\nu-1}) ] \, \Bigg| 
\end{array} \, \right] \\ [0.5in]
\leq \, \displaystyle{
\frac{N_{\nu} - N_{\nu-1}}{N_{\nu}}
} \, \left[ \, \displaystyle{
\frac{1}{N_{\nu-1} ^{\, \beta}}
} + \displaystyle{
\frac{1}{( \, N_{\nu} - N_{\nu-1} \, )^{\, \beta}}
} \, \right] \, (  2 K ) \, \left( \, W_1 + \displaystyle{
\frac{W_2}{\gamma_{\nu-1}}
} \, \right).
\end{array}
\]
By Lemma~\ref{lm:condition on sample sizes}, 
we can show that
\[ \displaystyle{
\sum_{\nu=1}^{\infty}
} \ \left( \, \Big| \, \mathbb{E}\left[ \, \delta_{1,\nu} \mid \, {\cal F}^{\nu-1}
\right] \, \Big| +
\Big| \, \mathbb{E} \left[ \, \delta_{2,\nu} \mid \, {\cal F}^{\nu-1} \right] \, \Big| +
\Big| \, \mathbb{E} \left[ \, \delta_{3,\nu} \mid \, {\cal F}^{\nu-1} \right] \, \Big|
\, \right)
\]
is finite with probability 1.
By the Robbins-Siegmund nonnegative almost supermartingale convergence
lemma (see \cite[Theorem~1]{robbins1971convergence} and
\cite[Lemma~11, Chapter~2]{polyak1987optimization}),
it follows that the sum
\[
\displaystyle{
\sum_{\nu=1}^{\infty}
} \, \displaystyle{
\frac{\rho_{\nu}}{\lambda_{\nu}}
} \, \mathbb{E}\left[ \, \| \, x^{\nu+1} - x^{\nu} \, \|^2 \, \mid \, {\cal F}^{\nu-1}\, \right]
\]
is finite with probability one. Thus so is the sum $\displaystyle{
\sum_{\nu=1}^{\infty}
} \,  \displaystyle{
\frac{\rho_{\nu}}{\lambda_{\nu}}
} \, \| \, x^{\nu+1} - x^{\nu} \, \|^2$ with probability one, by a similar
argument as in \cite[Theorem 1]{LiuCuiPang20}.
\end{proof}

\subsection{Feasibility and stationarity of a limit point}

Define the family ${\cal K}$ of infinite index subsets $\kappa$ of
$\{ 1, 2, \cdots, \infty \}$ such that $\displaystyle{
\lim_{\nu (\in \kappa) \to \infty}
} \, \| \, x^{\nu+1}- x^{\nu} \, \| \, = \, 0$ with probability 1.
This family is nonempty because otherwise, $\displaystyle{
\liminf_{\nu \to \infty}
} \, \| \, x^{\nu+1}- x^{\nu} \, \|$ would be positive, contradicting the
combined consequences:
$\displaystyle{
\sum_{\nu=1}^{\infty}
} \,  \displaystyle{
\frac{\rho_{\nu}}{\lambda_{\nu}}
} \, = \, \infty$ (see (\ref{eq:sum_rho_lambda})) and $\displaystyle{
\sum_{\nu=1}^{\infty}
} \, \displaystyle{
\frac{\rho_{\nu}}{\lambda_{\nu}}
} \, \| \, x^{\nu+1} - x^{\nu} \, \|^2 \, < \, \infty$
(Proposition~\ref{pr:successive iterates}) under the given assumptions.
Let $x^{\infty}$ be any accumulation point
  {(which must exist by the boundedness assumption of $X$)} of the subsequence
$\{ x^{\nu} \}_{\nu \in \kappa}$ produced by the SPSA for any
$\kappa \in {\cal K}$.   For simplicity, we assume that
$x^{\infty} = \displaystyle{
\lim_{\nu (\in \kappa) \to \infty}
} \ x^{\nu}$.  Hence $x^{\infty} = \displaystyle{
\lim_{\nu (\in \kappa) \to \infty}
} \ x^{\nu+1}$.  We wish to establish certain
feasibility and stationarity property of such a limit point.  We will divide the
analysis into two major cases: (i) the sequence
$\{ \gamma_{\nu} \}$ is a constant, and (ii) $\{ \gamma_{\nu} \} \downarrow 0$.
By the majorization property of the surrogation functions and the global optimality
of the iterates, we have
\[ \begin{array}{l}
\displaystyle{
\frac{\rho_{\nu}}{2 \lambda_{\nu}}
} \, \| \, x^{\nu+1} - x^{\nu} \, \|^2 + \displaystyle{
\frac{1}{N_{\nu} \lambda_{\nu}}
} \, \displaystyle{
\sum_{s=1}^{N_{\nu}}
} \, c_0(x^{\nu+1},z^s) + \displaystyle{
\sum_{k=1}^K
} \, \max\left\{ \, \displaystyle{
\frac{1}{N_{\nu}}
} \, \displaystyle{
\sum_{s=1}^{N_{\nu}}
} \, c_k(x^{\nu+1},z^s;\gamma_{\nu}) - \zeta_k, \, 0 \, \right\} \\ [0.3in]
\leq \, \displaystyle{
\frac{\rho_{\nu}}{2 \lambda_{\nu}}
} \, \| \, x^{\nu+1} - x^{\nu} \, \|^2 + \displaystyle{
\frac{1}{N_{\nu} \lambda_{\nu}}
} \, \displaystyle{
\sum_{s=1}^{N_{\nu}}
} \, \wh{c}_0(x^{\nu+1},z^s;x^{\nu}) + \displaystyle{
\sum_{k=1}^K
} \, \max\left\{ \, \displaystyle{
\frac{1}{N_{\nu}}
} \, \displaystyle{
\sum_{s=1}^{N_{\nu}}
} \, \wh{c}_k(x^{\nu+1},z^s;\gamma_{\nu};x^{\nu}) - \zeta_k, \, 0 \, \right\} \\ [0.3in]
\leq \, \displaystyle{
\frac{\rho_{\nu}}{2 \lambda_{\nu}}
} \, \| \, x - x^{\nu} \, \|^2 + \displaystyle{
\frac{1}{N_{\nu} \lambda_{\nu}}
} \, \displaystyle{
\sum_{s=1}^{N_{\nu}}
} \, \wh{c}_0(x,z^s;x^{\nu}) + \displaystyle{
\sum_{k=1}^K
} \, \max\left\{ \, \displaystyle{
\frac{1}{N_{\nu}}
} \, \displaystyle{
\sum_{s=1}^{N_{\nu}}
} \, \wh{c}_k(x,z^s;\gamma_{\nu};x^{\nu}) - \zeta_k, \, 0 \, \right\},\
\forall \, x \in X .
\end{array} \]

\subsubsection{Fixed approximation parameter $\gamma_\nu$}

Let $\gamma_{\nu} = \underline{\gamma} > 0$ for all $\nu$.
We then have the following inequality, which is  the cornerstone of the
remaining arguments. For all $x \in X$,
\begin{equation} \label{eq:LRHS}
\begin{array}{l}
\underbrace{\displaystyle{
\frac{\rho_{\nu}}{2 \lambda_{\nu}}
} \, \| \, x^{\nu+1} - x^{\nu} \, \|^2 + \displaystyle{
\frac{1}{N_{\nu} \lambda_{\nu}}
} \, \displaystyle{
\sum_{s=1}^{N_{\nu}}
} \, c_0(x^{\nu+1},z^s) + \displaystyle{
\sum_{k=1}^K
} \, \max\left\{ \, \displaystyle{
\frac{1}{N_{\nu}}
} \, \displaystyle{
\sum_{s=1}^{N_{\nu}}
} \, c_k(x^{\nu+1},z^s;\underline{\gamma}) - \zeta_k, \, 0 \, \right\}}_{
\mbox{denoted LHS$_{\, \nu}$}} \\ [0.35in]
\leq \,  \underbrace{\displaystyle{
\frac{\rho_{\nu}}{2 \lambda_{\nu}}
} \, \| \, x - x^{\nu} \, \|^2 + \displaystyle{
\frac{1}{N_{\nu} \lambda_{\nu}}
} \, \displaystyle{
\sum_{s=1}^{N_{\nu}}
} \, \wh{c}_0(x,z^s;x^{\nu}) + \displaystyle{
\sum_{k=1}^K
} \, \max\left\{ \, \displaystyle{
\frac{1}{N_{\nu}}
} \, \displaystyle{
\sum_{s=1}^{N_{\nu}}
} \, \wh{c}_k(x,z^s;\underline{\gamma};x^{\nu}) - \zeta_k, \, 0 \, \right\}}_{
\mbox{denoted RHS$_{\, \nu}$}}.
\end{array} \end{equation}

\gap

The following theorem presents the main convergence result for the case of a fixed
$\underline{\gamma}$.  In particular, the first assertion gives a sufficient
condition
for the feasibility of a limit point to the $\underline{\gamma}$-approximation
problem (\ref{eq:new SP}), under which
the B-stationarity of the point to the same problem can be established with
a further constraint qualification.  Notice that since $\underline{\gamma}$ stays
positive, one cannot expect the feasibility to the limiting constraint
in (\ref{eq:focus sp_dc_constraint}) to be recovered.  Thus, this result addresses
basically the $\underline\gamma$-approximation of the chance-constraint optimization
problem  with an arbitrarily prescribed scaling parameter $\underline\gamma > 0$.

\begin{theorem} \label{th:convergence finite gamma} \rm
In the setting of Proposition~\ref{pr:successive iterates}  for  the problem
\eqref{eq:new SP}, let $\{ x^{\nu} \}$ be a sequence of iterates
produced by the Algorithm with $\gamma_{\nu}$ equal to the constant
$\underline{\gamma}$ for all $\nu$.  For any infinite index set $\kappa \in {\cal K}$,
the following two statements (a) and (b) hold for any accumulation point
$x^{\infty}$ of the subsequence $\{ x^{\nu} \}_{\nu \in \kappa}$:

\gap

(a) If the following \textcolor{black}{surrogate}
$\underline{\gamma}$-problem at $x^{\infty}$:
\begin{equation} \label{eq:problem at limit}
\begin{array}{ll}
\displaystyle{
\operatornamewithlimits{\mbox{\bf minimize}}_{x \in X}
} & 
\mathbb{E}\left[ \, \wh{c}_0(x,\tilde{z};x^{\infty}) \, \right] \\ [0.1in]
\mbox{\bf subject to} & \wh{c}_k(x;\underline{\gamma};x^{\infty}) \, \triangleq \,
\mathbb{E}\left[ \, \wh{c}_k(x,\tilde{z};\underline{\gamma};x^{\infty}) \, \right]
\, \leq \, \zeta_k \epc \forall \, k \, \in \, [ K ]
\end{array} \end{equation}
has a feasible solution $\wh{x}$ satisfying
\[
\displaystyle{
\frac{1}{\lambda_{\infty}}
} \, \mathbb{E}\left[ \, \wh{c}_0(\wh{x},\tilde{z};x^{\infty}) \, \right] \, \leq \,
\displaystyle{
\frac{1}{\lambda_{\infty}}
} \,\mathbb{E}\left[ \, \wh{c}_0(x^{\infty},\tilde{z};x^{\infty}) \, \right],
\]
[the latter condition is trivially satisfied if $\lambda_{\infty} = +\infty$],
then $x^{\infty}$ is feasible to the $\underline{\gamma}$-problem
(\ref{eq:new SP}), or equivalently,
feasible to the problem (\ref{eq:problem at limit}).

\gap

\textcolor{black}{(b) Assume that the vector $\wh{x}$ in (a) exists.  Then, under
the constraint closure condition:
\begin{equation} \label{eq:closure condition finite gamma}
\emptyset \, \neq \, \displaystyle{
\bigcap_{k=1}^K
} \, \left\{ \, x \, \in \, X \, \mid \, \wh{c}_k(x;\underline{\gamma};x^{\infty})
\, \leq \, \zeta_k \, \right\}
\, \subseteq \, \mbox{cl} \left( \, \displaystyle{
\bigcap_{k=1}^K
} \, \left\{ \, x \, \in \, X \, \mid \, \wh{c}_k(x;\underline{\gamma};x^{\infty})
\, < \, \zeta_k \, \right\} \right),
\end{equation}
it holds that
$x^{\infty}$ is a B-stationary point of the problem (\ref{eq:problem at limit});
%
%
if additionally,
$x^{\infty}$, which
is feasible to the $\underline{\gamma}$-problem (\ref{eq:new SP}),
satisfies the directional Slater condition for the feasible set of this problem,
\textcolor{black}{i.e., if the following inclusion holds:
\begin{equation} \label{eq:dd-Slater}
\begin{array}{l}
\displaystyle{
\bigcap_{k \in {\cal A}(x^{\infty};\underline{\gamma})}
} \, \left\{ \, v \, \in \, {\cal T}(x^{\infty};X) \, \mid \,
\bar{c}_k(\bullet;\underline{\gamma})^{\prime}(x^{\infty};v) \, \leq 0 \, \right\}
\\ [0.25in]
\epc \subseteq \mbox{cl}\left( \, \displaystyle{
\bigcap_{k \in {\cal A}\left( x^{\infty};\underline{\gamma} \right)}
} \, \left\{ \, v \, \in \, {\cal T}(x^{\infty};X) \, \mid \,
\bar{c}_k(\bullet;\underline{\gamma})^{\prime}(x^{\infty};v) \, < 0 \, \right\}
\, \right),
\end{array} \end{equation}
where ${\cal A}\left( x^{\infty};\underline{\gamma} \right) \, \triangleq \,
\left\{ \, k \in [ K ] \, \mid \,
\bar{c}_k(x^{\infty};\underline{\gamma}) \, = \, \zeta_k \, \right\}$,}
\noindent then $x^{\infty}$ is a B-stationary point of (\ref{eq:new SP}).}
\end{theorem}

\begin{proof}   (a) We have
\[ \begin{array}{l}
\mbox{LHS}_{\, \nu} \, = \, \displaystyle{
\frac{\rho_{\nu}}{2 \, \lambda_{\nu}}
} \, \| \, x^{\nu+1} - x^{\nu} \, \|^2 + \displaystyle{
\frac{1}{N_{\nu} \, \lambda_{\nu}}
} \, \displaystyle{
\sum_{s=1}^{N_{\nu}}
} \, c_0( x^{\nu+1},z^s) \ + \\ [0.2in]
\, \displaystyle{
\sum_{k=1}^K
} \, \max\left\{ \, \left( \, \displaystyle{
\frac{1}{N_{\nu}}
} \, \displaystyle{
\sum_{s=1}^{N_{\nu}}
} \, c_k(x^{\nu+1},z^s;\underline{\gamma}) -
\mathbb{E}\left[ c_k(x^{\nu+1},\tilde{z};\underline{\gamma}) \right] \, \right) +
\left( \, \mathbb{E}\left[ c_k(x^{\nu+1},\tilde{z};\underline{\gamma}) \right] -
\zeta_k \, \right), \, 0 \, \right\}.
\end{array} \]
Hence with probability 1,
\[
\displaystyle{
\lim_{\nu (\in \kappa) \to \infty}
} \, \mbox{LHS}_{\, \nu} \, = \, \displaystyle{
\frac{1}{\lambda_{\infty}}
} \, \mathbb{E}\left[ c_0(x^{\infty},\tilde{z} ) \right] + \displaystyle{
\sum_{k=1}^K
} \, \max\left\{ \, \mathbb{E}\left[ c_k(x^{\infty},\tilde{z};\underline{\gamma}) \right]
- \zeta_k, \, 0 \, \right\}.
\]
Substituting the feasible vector $\wh{x}$ into RHS$_{\, \nu}$, we have
\[ \begin{array}{l}
\mbox{RHS}_{\, \nu} \, = \, \displaystyle{
\frac{\rho_{\nu}}{2 \lambda_{\nu}}
} \, \| \, \wh{x} - x^{\nu} \, \|^2 + \displaystyle{
\frac{1}{N_{\nu} \lambda_{\nu}}
} \, \displaystyle{
\sum_{s=1}^{N_{\nu}}
} \, \wh{c}_0(\wh{x},z^s;x^{\nu}) + \displaystyle{
\sum_{k=1}^K
} \, \max\left\{ \, \displaystyle{
\frac{1}{N_{\nu}}
} \, \displaystyle{
\sum_{s=1}^{N_{\nu}}
} \, \wh{c}_k(\wh{x},z^s;\underline{\gamma};x^{\nu}) - \zeta_k, \, 0 \, \right\}
\\ [0.2in]
\leq \, \displaystyle{
\frac{\rho_{\nu}}{2 \lambda_{\nu}}
} \, \| \, \wh{x} - x^{\nu} \, \|^2 + \displaystyle{
\frac{1}{N_{\nu} \lambda_{\nu}}
} \, \displaystyle{
\sum_{s=1}^{N_{\nu}}
} \, \wh{c}_0(\wh{x},z^s;x^{\nu}) +  \displaystyle{
\sum_{k=1}^K
} \, \max\left\{ \,
\mathbb{E}\left[ \wh{c}_k(\wh{x},\tilde{z};\underline{\gamma};x^{\infty}) \, \right] -
\zeta_k, \, 0 \, \right\} \ + \\ [0.2in]
\hspace{0.4in}  \, \displaystyle{
\sum_{k=1}^K
} \, \max\left\{ \, \displaystyle{
\frac{1}{N_{\nu}}
} \, \displaystyle{
\sum_{s=1}^{N_{\nu}}
} \, \wh{c}_k(\wh{x},z^s;\underline{\gamma};x^{\nu}) -
\mathbb{E}\left[ \wh{c}_k(\wh{x},\tilde{z};\underline{\gamma};x^{\infty}) \, \right], \,
0 \, \right\} \\ [0.2in]
\leq \, \displaystyle{
\frac{\rho_{\nu}}{2 \lambda_{\nu}}
} \, \| \, \wh{x} - x^{\nu} \, \|^2 + \displaystyle{
\frac{1}{\lambda_{\nu}}
} \, \left\{ \displaystyle{
\frac{1}{N_{\nu}}
} \, \displaystyle{
\sum_{s=1}^{N_{\nu}}
} \, \wh{c}_0(\wh{x},z^s;x^{\nu}) - \mathbb{E}\left[
\wh{c}_0(\wh{x},\tilde{z};x^{\infty}) \right] \right\} +  \displaystyle{
\frac{1}{\lambda_{\nu}}
} \, \mathbb{E}\left[ \wh{c}_0(\wh{x},\tilde{z};x^{\infty}) \right] \ + \\ [0.2in]
\hspace{0.4in}  \, \displaystyle{
\sum_{k=1}^K
} \, \max\left\{ \, \displaystyle{
\frac{1}{N_{\nu}}
} \, \displaystyle{
\sum_{s=1}^{N_{\nu}}
} \, \wh{c}_k(\wh{x},z^s;\underline{\gamma};x^{\nu}) -
\mathbb{E}\left[ \wh{c}_k(\wh{x},\tilde{z};\underline{\gamma};x^{\infty}) \, \right],
\, 0 \, \right\}.
\end{array} \]
Therefore by Proposition~\ref{pr:uniform pointwise ULLN} for the constraint surrogation functions
$\wh{c}_k(\wh{x},z;{\underline \gamma};\bullet)$
and a similar inequality for the objective
surrogation function $\wh{c}_0$, and the fact
that $\rho_{\nu}/\lambda_{\nu} \to 0$ and $\lambda_{\nu} \to \lambda_{\infty}$,
we deduce with probability 1,
\[
\displaystyle{
\lim_{\nu (\in \kappa) \to \infty}
} \, \mbox{RHS}_{\nu} \, \leq \, \displaystyle{
\frac{1}{\lambda_{\infty}}
} \, \mathbb{E}\left[ \wh{c}_0(\wh{x},\tilde{z};x^{\infty}) \right] \, \leq \,  \displaystyle{
\frac{1}{\lambda_{\infty}}
} \, \mathbb{E}\left[ \wh{c}_0(x^{\infty},\tilde{z};x^{\infty}) \right] = \displaystyle{
\frac{1}{\lambda_{\infty}}
} \, \mathbb{E}\left[{c}_0(x^{\infty},\tilde{z} ) \right].
\]
Consequently, we deduce
$\max\left\{ \, \mathbb{E}\left[ c_k(x^{\infty},\tilde{z};\underline{\gamma}) \right]
- \zeta_k, \, 0 \, \right\} \leq 0$ for all $k \in [ K ]$ with probability 1.  Therefore,
$x^{\infty}$ is feasible to the
$\underline{\gamma}$-problem (\ref{eq:new SP}) with probability 1.

\gap

(b) Suppose (\ref{eq:closure condition finite gamma}) holds.  By (a),
$x^{\infty}$ is feasible to
the problem (\ref{eq:problem at limit}).  We first show that $x^{\infty}$ is
a global minimizer of the same problem with an additional proximal
term; i.e., the problem
\begin{equation} \label{eq:limiting approximate problem}
\begin{array}{ll}
\displaystyle{
\operatornamewithlimits{\mbox{\bf minimize}}_{x \in X}
} & 
\mathbb{E}\left[ \, \wh{c}_0(x,\tilde{z};x^{\infty}) \, \right] + \displaystyle{
\frac{\rho_{\infty}}{2}
} \, \| \, x - x^{\infty} \, \|^2 \\ [0.1in]
\mbox{\bf subject to} & \mbox{same constraints as (\ref{eq:problem at limit}).}
\end{array} \end{equation}
Let $\wh{x}$ be a feasible solution to (\ref{eq:limiting approximate problem}).
By (\ref{eq:closure condition finite gamma}), there exists a sequence
$\{ \wh{x}^{\, \mu} \}_{\mu=1}^{\infty} \subseteq X$ converging to $\wh{x}$ such that
for each $\mu$,
$\mathbb{E}\left[ \,
\wh{c}_k(\wh{x}^{\, \mu},\tilde{z};\underline{\gamma};x^{\infty}) \, \right] < \zeta_k$
for all $k \in [ K ]$.
Consider any such vector $\wh{x}^{\, \mu}$.
It follows from Proposition \ref{pr:uniform pointwise ULLN} of the constraint functions
that for all
$\nu (\in \kappa)$ sufficiently large (dependent on $\mu$), we have $\displaystyle{
\frac{1}{N_{\nu}}
} \, \displaystyle{
\sum_{s=1}^{N_{\nu}}
} \, \wh{c}_k(\wh{x}^{\, \mu},z^s;\underline{\gamma};x^{\nu}) \leq \zeta_k$
almost surely.  Therefore, with $x = \wh{x}^{\, \mu}$,
we obtain from (\ref{eq:LRHS}), after justifiably dropping the max terms on the left
and right sides and then $\lambda_{\nu}$,
\[ \begin{array}{l}
\displaystyle{
\frac{\rho_{\nu}}{2}
} \, \| \, x^{\nu+1} - x^{\nu} \, \|^2 + \displaystyle{
\frac{1}{N_{\nu}}
} \, \displaystyle{
\sum_{s=1}^{N_{\nu}}
} \, c_0(x^{\nu+1},z^s) \, \leq \, \displaystyle{
\frac{\rho_{\nu}}{2}
} \, \| \, \wh{x}^{\, \mu} - x^{\nu} \, \|^2 + \displaystyle{
\frac{1}{N_{\nu}}
} \, \displaystyle{
\sum_{s=1}^{N_{\nu}}
} \, \wh{c}_0(\wh{x}^{\, \mu},z^s;x^{\nu}) \\ [0.2in]
\leq \,  \displaystyle{
\frac{\rho_{\nu}}{2}
} \, \| \, \wh{x}^{\, \mu} - x^{\nu} \, \|^2 + \left\{ \, \displaystyle{
\frac{1}{N_{\nu}}
} \, \displaystyle{
\sum_{s=1}^{N_{\nu}}
} \, \wh{c}_0(\wh{x}^{\, \mu},z^s;x^{\nu}) - \mathbb{E}\left[ \wh{c}_0(\wh{x}^{\, \mu},
\tilde{z};x^{\infty}) \right] \, \right\} +
\mathbb{E}\left[ \wh{c}_0(\wh{x}^{\, \mu},\tilde{z};x^{\infty}) \right].
\end{array} \]
By Proposition \ref{pr:uniform pointwise ULLN} applied to the objective function
$\wh c_0(\bullet,\tilde{z};\bullet)$, and by taking the limits
$\mu \to \infty$ and $\nu (\in \kappa) \to \infty$, we deduce with probability 1,
\[
\mathbb{E}\left[ c_0(x^{\infty},\tilde{z} ) \right] \, \leq \, \displaystyle{
\frac{\rho_{\infty}}{2}
} \, \| \, \wh{x} - x^{\infty} \, \|^2 + \mathbb{E}\left[ \wh{c}_0(\wh{x},\tilde{z};x^{\infty}) \right],
\]
which establishes the minimizing claim about $x^{\infty}$.
By the first-order optimality condition (\ref{eq:limiting approximate problem}),
the claimed B-stationarity of
$x^{\infty}$ with reference to the problem (\ref{eq:problem at limit}) follows readily.

\gap

With $x^{\infty}$ being feasible to the $\underline{\gamma}$-problem
(\ref{eq:new SP}), we can justifiably assume
the directional Slater condition (\ref{eq:dd-Slater}).
It remains to show, by using the latter condition, that
$c_0^{\, \prime}(x^{\infty};v) \geq 0$ for all $v$ satisfying
$\bar{c}_k(\bullet;\underline{\gamma})^{\prime}(x^{\infty};v) < 0$ for all
$k \in {\cal A}(x^{\infty})$.
For such a vector $v$, we have
$\wh{c}_k(\bullet;\underline{\gamma};x^{\infty})^{\prime}(x^{\infty};v) < 0$
for all $k \in {\cal A}(x^{\infty})$,
by the directional derivative consistency condition  of the surrogate functions; thus
$\wh{c}_k(x^{\infty} + \tau v;\underline{\gamma};x^{\infty}) < \zeta$
for all $\tau > 0$ sufficiently small.  By the above proof,
$x^{\infty}$ is an optimal solution of (\ref{eq:limiting approximate problem}); thus
\[
\mathbb{E}\left[ \, \wh{c}_0(x^{\infty} + \tau v,\tilde{z};x^{\infty}) \, \right] +
\displaystyle{
\frac{\rho_{\infty}}{2}
} \ \tau ^2 \, \| \,\, v \, \|^2 \, \geq \, \mathbb{E}\left[ \,
\wh{c}_0(x^{\infty},\tilde{z};x^{\infty}) \, \right], \epc
\forall \, \tau > 0 \mbox{ sufficiently small}.
\]
Dividing $\tau > 0$ and letting $\tau \downarrow 0$, we obtain,
by the directional derivative consistency condition of the surrogate function
$\wh{c}_0(\bullet,z;x^{\infty})$ for $c_0(\bullet,z)$
\[
\left[ \, \mathbb{E}\left[ \, \wh{c}_0(\bullet,\tilde{z};x^{\infty}) \, \right]
\, \right]^{\, \prime}(x^{\infty};v) \, = \,
\bar{c}_0^{\, \prime}(x^{\infty};v) \, \geq \, 0,
\]
establishing the desired B-stationarity of $x^{\infty}$.
\end{proof}

We make several remarks about the above theorem.  First, in addition to the basic
set-up, the theorem relies on two key assumptions:
(i) the existence of the feasible vector $\wh{x}$, and (ii) the constraint closure
condition.  Both are reasonable: for the former
condition, we need to keep in mind that the algorithm encompasses a penalty idea by
softening the hard $\gamma$-expectation constraints.
In order to recover the feasibility of such constraints, the two main exact penalization
results in
Section~\ref{sec:exact penalization}---Propositions~\ref{pr:stationary penalization}
and \ref{pr:approx stationary penalization}---impose certain global directional
derivative conditions on all infeasible points; whereas the feasibility
assumption in Theorem~\ref{th:convergence finite gamma} pertains to the limit point
on hand; if we desire, the assumption can certainly be globalized to all
infeasible points.  Another salient point about
this pointwise assumption is that it exploits the construction that leads to the limit.
Since the function $\wh{c}_k(\bullet,z;\underline{\gamma};x^{\infty})$
is continuous, the left-hand set in (\ref{eq:closure condition finite gamma}) is
closed; thus equality holds between the two sets.  We write
this condition as an inclusion to be consistent with the subsequent condition
(\ref{eq:closure condition diminishing gamma}) for the case of diminishing
$\gamma_{\nu} \downarrow 0$ where the left-hand set may not be closed.
These closure conditions are constraint qualifications; needless
to say, for general non-affine problems, such a condition is a must in order to
establish any kind of sharp stationarity properties of the point of interest.

\subsubsection{Diminishing approximation parameter $\gamma_\nu$} \label{subsec:diminishing approx}

Let $\displaystyle{
\lim_{\nu \to \infty}
} \, \gamma_{\nu} = 0$.  The general result in this case requires the use of a
limiting function to play the role of the fixed $\underline{\gamma}$-functions
$\wh{c}_k(\bullet;\underline{\gamma};x^{\infty})$ for $k \in [ K ]$ in the previous case.
These alternative functions are required to satisfy the limit conditions
(\ref{eq:Bunderline}) and (\ref{eq:Bupper})   {given below}.

\gap

$\bullet $ For each $k \in [ K ]$, there exists a function
$\wt{c}_k(\bullet,\bullet;\bullet) : X \times \Xi \times X \to \mathbb{R}$ such that
for all $x \in X$, the function
$\wt{c}_k(x,\bullet;x^{\infty})$ is measurable with
$\mathbb{E}\,[ \, \wt{c}_k(x,\tilde{z};x^{\infty})\, ] < \infty$, and for all
i.i.d.\ samples $\{ z^s \}_{s=1}^{\infty}$, it holds that

\gap

-- for all $\{ ( y^{\nu},x^{\nu} ) \} \subseteq X \times X$ converging to
$( x^\infty,x^\infty )$,
\begin{equation} \label{eq:Bunderline}
\mathbb{E}[ \, \wt{c}_k(x^\infty,\tilde{z};x^{\infty}) \, ] \, \leq \,
\displaystyle{
\liminf_{\nu (\in \kappa) \to \infty}
} \, \displaystyle{
\frac{1}{N_{\nu}}
} \, \displaystyle{
\sum_{s=1}^{N_{\nu}}
} \, \wh{c}_k(y^\nu,z^s;\gamma_{\nu};x^{\nu}), \epc \mbox{almost surely};
\end{equation}
-- for all $\{ ( y^{\nu},x^{\nu} ) \} \subseteq X \times X$ converging to
$( y^\infty,x^\infty )$ for some $y^{\infty} \in X$,
\begin{equation} \label{eq:Bupper}
\displaystyle{
\limsup_{\nu (\in \kappa) \to \infty}
} \, \displaystyle{
\frac{1}{N_{\nu}}
} \, \displaystyle{
\sum_{s=1}^{N_{\nu}}
} \, \wh{c}_k(y^{\nu},z^s;\gamma_{\nu};x^{\nu}) \, \leq \, \mathbb{E}[ \,
\wt{c}_k(y^{\infty},\tilde{z};x^{\infty}) \, ], \epc \mbox{almost surely}.
\end{equation}

Subsequently, we will discuss the choice of the functions
$\wt{c}_k(\bullet,\bullet;\bullet)$  in the context of the restricted
(\ref{eq:restricted sp_dc_constraint}) and
relaxed (\ref{eq:relaxed sp_dc_constraint}) problems; for now,
we establish the following analogous convergence result for the case of a
diminishing sequence $\{ \gamma_{\nu} \} \downarrow 0$.

\begin{theorem} \label{th:convergence diminishing gamma} \rm
In the setting of Proposition~\ref{pr:successive iterates},
let $\{ x^{\nu} \}$ be a sequence of iterates produced by the
Algorithm with $\gamma_{\nu} \downarrow 0$.  Let $x^{\infty}$ be an accumulation point
of the
subsequence $\{ x^{\nu} \}_{\nu \in \kappa}$ corresponding to any infinite index set
$\kappa \in {\cal K}$.
Assume that $x^{\infty}$ satisfies the conditions \eqref{eq:Bunderline}
and \eqref{eq:Bupper}
for some functions 
$\{ \wt{c}_k(\bullet,\bullet;\bullet) \}_{k=1}^K$.
Then the following two statements hold for $x^{\infty}$:

\gap

(a) if there exists $\wh{x} \in X$ satisfying $\mathbb{E}\left[ \,
\wt{c}_k(\wh{x},\tilde{z};x^{\infty}) \, \right]
\, \leq \, \zeta_k$ for all $k \in [ K ]$ and also
\begin{equation} \label{eq:objective less}
\displaystyle{
\frac{1}{\lambda_{\infty}}
} \, \mathbb{E}\left[ \, \wh{c}_0(\wh{x},\tilde{z};x^{\infty}) \, \right]
\, \leq \displaystyle{
\frac{1}{\lambda_{\infty}}
} \,\mathbb{E}\left[ \, \wh{c}_0(x^{\infty},\tilde{z};x^{\infty}) \, \right];
\end{equation}
then $x^{\infty}$ satisfies
$\mathbb{E}[ \, \wt{c}_k(x^\infty,\tilde{z};x^{\infty}) \, ]  \, \leq \, \zeta_k$ for
all $k \in [ K ]$;

\gap

(b) if in addition the closure condition holds:
\begin{equation} \label{eq:closure condition diminishing gamma}
\emptyset \, \neq \, \displaystyle{
\bigcap_{k=1}^K
} \, \left\{ \, x \, \in \, X \, \mid \, \mathbb{E}[ \,
\wt{c}_k(x,\tilde{z};x^{\infty})  \, ] \, \leq \, \zeta_k \, \right\}
\, \subseteq  \, \mbox{cl}\left( \, \displaystyle{
\bigcap_{k=1}^K
} \, \left\{ x \, \in \, X \, \mid \,
\mathbb{E}[ \, \wt{c}_k(x,\tilde{z};x^{\infty}) \, ] \, < \, \zeta_k  \, \right\}
\, \right),
\end{equation}
then $x^{\infty}$ is a B-stationary point of the problem:
\begin{equation} \label{eq:problem at limit diminishing gamma}
\begin{array}{ll}
\displaystyle{
\operatornamewithlimits{\mbox{\bf minimize}}_{x \in X}
} & \mathbb{E}\left[ \, \wh{c}_0(x,\tilde{z};x^{\infty}) \, \right] + \displaystyle{
\frac{\rho_{\infty}}{2}
} \, \| \, x - x^{\infty} \, \|^2 \\ [0.15in]
\mbox{\bf subject to} & \mathbb{E}\left[ \, \wt{c}_k(x,\tilde{z};x^{\infty}) \, \right]
\, \leq \, \zeta_k, \epc \forall \, k \, \in \, [ K ].
\end{array} \end{equation}
\end{theorem}

\begin{proof}  Instead of (\ref{eq:LRHS}), we have \textcolor{black}{for all $x \in X$,}
\[ 
\begin{array}{l}
\underbrace{\displaystyle{
\frac{\rho_{\nu}}{2 \lambda_{\nu}}
} \, \| \, x^{\nu+1} - x^{\nu} \, \|^2 + \displaystyle{
\frac{1}{N_{\nu} \lambda_{\nu}}
} \, \displaystyle{
\sum_{s=1}^{N_{\nu}}
} \, c_0(x^{\nu+1},z^s) + \displaystyle{
\sum_{k=1}^K
} \, \max\left\{ \, \displaystyle{
\frac{1}{N_{\nu}}
} \, \displaystyle{
\sum_{s=1}^{N_{\nu}}
} \, \wh{c}_k(x^{\nu+1},z^s;\gamma_{\nu};x^\nu) - \zeta_k, \, 0 \, \right\}}_{
\mbox{denoted LHS$_{\, \nu}$}} \\ [0.35in]
\leq \, \underbrace{\displaystyle{
\frac{\rho_{\nu}}{2 \lambda_{\nu}}
} \, \| \, x - x^{\nu} \, \|^2 + \displaystyle{
\frac{1}{N_{\nu} \lambda_{\nu}}
} \, \displaystyle{
\sum_{s=1}^{N_{\nu}}
} \, \wh{c}_0(x,z^s;x^{\nu}) + \displaystyle{
\sum_{k=1}^K
} \, \max\left\{ \, \displaystyle{
\frac{1}{N_{\nu}}
} \, \displaystyle{
\sum_{s=1}^{N_{\nu}}
} \, \wh{c}_k(x,z^s;\gamma_{\nu};x^{\nu}) - \zeta_k, \, 0 \, \right\}}_{
\mbox{denoted RHS$_{\, \nu}$}}.
\end{array} \]
Thus
\[ \begin{array}{l}
\mbox{LHS}_{\nu} \, = \, \displaystyle{
\frac{\rho_{\nu}}{2 \lambda_{\nu}}
} \, \| \, x^{\nu+1} - x^{\nu} \, \|^2 + \displaystyle{
\frac{1}{N_{\nu} \lambda_{\nu}}
} \, \displaystyle{
\sum_{s=1}^{N_{\nu}}
} \, c_0(x^{\nu+1},z^s) \ + \\ [0.1in]
\displaystyle{
\sum_{k=1}^K
 } \, \max\left\{ \, \left( \, \displaystyle{
\frac{1}{N_{\nu}}
} \, \displaystyle{
\sum_{s=1}^{N_{\nu}}
} \, \wh{c}_k(x^{\nu+1},z^s;\gamma_\nu;x^\nu) - \,
\mathbb{E}[ \, \wt{c}_k(x^\infty,\tilde{z};x^{\infty}) \, ]  \, \right) +
\mathbb{E}[ \, \wt{c}_k(x^\infty,\tilde{z};x^{\infty}) \, ] - \zeta_k, \, 0 \, \right\},
\end{array} \]
which yields, upon taking the liminf as $\nu (\in \kappa) \to \infty$,
with probability 1,
\[
\displaystyle{
\liminf_{\nu (\in \kappa) \to \infty}
} \, \mbox{LHS}_{\nu} \, \geq \, \displaystyle{
\frac{1}{\lambda_{\infty}}
} \, \mathbb{E}\left[ c_0(x^{\infty}, \tilde z) \right] + \displaystyle{
\sum_{k=1}^K
} \, \max\left\{ \, \mathbb{E}[ \, \wt{c}_k(x^\infty,\tilde{z};x^{\infty}) \, ] -
\zeta_k, \, 0 \, \right\}.
\]
Letting $x = \wh{x}$, we deduce
\[ \begin{array}{l}
\left( \mbox{ RHS}_{\nu} \mbox{ at $x = \wh{x}$} \, \right) \, \leq \, \displaystyle{
\frac{\rho_{\nu}}{2 \lambda_{\nu}}
} \, \| \, \wh{x} - x^{\nu} \, \|^2 + \displaystyle{
\frac{1}{N_{\nu} \lambda_{\nu}}
} \, \displaystyle{
\sum_{s=1}^{N_{\nu}}
} \, \wh{c}_0(\wh{x},z^s;x^{\nu}) \ + \\ [0.2in]
\displaystyle{
\sum_{k=1}^K
} \, \max\left\{ \, \left( \, \displaystyle{
\frac{1}{N_{\nu}}
} \, \displaystyle{
\sum_{s=1}^{N_{\nu}}
} \, \wh{c}_k(\wh{x},z^s;\gamma_{\nu};x^{\nu}) - \mathbb{E}\left[ \,
\wt{c}_k(\wh{x},\tilde{z};x^{\infty}) \, \right] \, \right) +
\mathbb{E}\left[ \, \wt{c}_k(\wh{x},\tilde{z};x^{\infty}) \, \right] - \zeta_k,
\,  0 \, \right\}.
\end{array} \]
By the same argument as in the proof of Theorem~\ref{th:convergence finite gamma},
the proof of the two assertions about $x^{\infty}$ from this point on
is similar, except that instead of the functions
$\wh{c}_k(\bullet,\bullet;\underline{\gamma};x^{\infty})$, we replace them by
the functions
$\wt{c}_k(\bullet,\bullet;x^{\infty})$ and employ the two limit assumptions
\eqref{eq:Bunderline} and \eqref{eq:Bupper}.   We do not repeat the details.
\end{proof}

\begin{remark}  \label{rm:theorem diminishing} \rm
To be consistent with Theorem~\ref{th:convergence finite gamma},
Theorem~\ref{th:convergence diminishing gamma} involves only one function
$\wt{c}$ satisfying the two limits
(\ref{eq:Bunderline}) and (\ref{eq:Bupper}).  When specialized to the relaxed
(\ref{eq:relaxed sp_dc_constraint}) and
restricted (\ref{eq:restricted sp_dc_constraint}) problems to be discussed momentarily,
this necessitates a zero-probability assumption at the limit
point $x^{\infty}$.  A more general version of
Theorem~\ref{th:convergence diminishing gamma} can be proved wherein
we employ two separate functions $\underline{c}_k$ and $\wt{c}_k$, the former for the
liminf inequality (\ref{eq:Bunderline})
and the latter for the limsup inequality (\ref{eq:Bupper}).  In this generalized
version, the conclusion of part (a) in
Theorem~\ref{th:convergence diminishing gamma} would be that
$x^{\infty}$ satisfies $\mathbb{E}[ \,
\underline{c}_k(x^\infty,\tilde{z};x^{\infty}) \, ]  \leq \zeta$, while the feasibility
of $x^{\infty}$ to (\ref{eq:problem at limit diminishing gamma}) in part (b) needs to
be made an assumption, instead of being
a consequence of part (a) as in the single-function version of the theorem.
See also Remark~\ref{rm:theorem diminishing continued}.  \hfill $\Box$
\end{remark}


\subsubsection{Returning to the relaxed (\ref{eq:relaxed sp_dc_constraint}) and
restricted (\ref{eq:restricted sp_dc_constraint}) problems: $\gamma_{\nu} \downarrow 0$}
\label{subsec:gamma to zero}

Under the setting in Sections~\ref{sec:piecewise probabilistic constraints}
and \ref{sec:ACCSP},
the specialization of Theorem~\ref{th:convergence finite gamma} (for finite
$\underline{\gamma}$) to the restricted
and relaxed problems with the surrogation functions derived in Appendix~1 is
fairly straightforward.  \textcolor{black}{See also Appendix~2 for the discussion of
the practical implementation of the SPSA with these surrogate functions.}
In what follows, we address the specialization of
Theorem~\ref{th:convergence diminishing gamma} to these two problems
\textcolor{black}{since it involves
the auxiliary functions $\wt{c}_k(\bullet,\bullet;\bullet)$ that remain fairly
abstract up to this point.}
For this purpose, we need to introduce a particular majorization so that the theorem
is applicable.  The focus is on the approximation of the constraint:
\begin{equation} \label{eq:single cc}
\displaystyle{
\sum_{\ell=1}^L
} \, ( \, e^+_{k\ell} - e^-_{k\ell} \, ) \,
\mathbb{P}({\cal Z}_{\ell}(x,\tilde{z} ) \geq 0) \, \leq \, \zeta_k, \epc \mbox{where} \epc
{\cal Z}_{\ell}(x,z) \, = \, \displaystyle{
\max_{1 \leq i \leq I_{\ell}}
} \, g_{i\ell}(x,z) - \displaystyle{
\max_{1 \leq j \leq J_{\ell}}
}\, h_{j\ell}(x,z),
\end{equation}
by the relaxed and restricted constraints
\[
\bar c^{\rm rlx}_k(x;\gamma) \, \triangleq \, \mathbb{E}\left[ \, \displaystyle{
\sum_{\ell=1}^L
} \, c^{\, \rm rlx}_{k\ell}(x,\tilde{z};\gamma) \, \right] \, \leq \, \zeta_k
\epc \mbox{and} \epc
\bar c^{\rm rst}_k(x;\gamma) \, \triangleq \, \mathbb{E}\left[ \, \displaystyle{
\sum_{\ell=1}^L
} \, c_{k\ell}^{\, \rm rst}(x,\tilde{z};\gamma) \, \right] \, \leq \, \zeta_k,
\]
where, with $t_{\ell}$ being the shorthand for ${\cal Z}_{\ell}(x,z)$,
\[ \begin{array}{l}
c^{\, \rm rlx}_{k\ell}(x,z;\gamma) \, \triangleq \,
e_{k\ell}^+ \phi_{\rm lb}(t_{\ell},\gamma) - e_{k\ell}^-
\phi_{\rm ub}(t_{\ell},\gamma) \\ [0.1in]
\hspace{0.8in} = e_{k\ell}^+ \max\left\{ \,
\min\left( \, \wh{\theta}_{\rm cve}\left( \displaystyle{
\frac{t_{\ell}}{\gamma}
} \right), \, 1 \, \right), \ 0 \, \right\} - e_{k\ell}^-
\min\left\{ \, \max\left( \, \wh{\theta}_{\rm cvx}\left( 1 + \displaystyle{
\frac{t_{\ell}}{\gamma}
} \right), \, 0 \, \right), \, 1 \, \right\} \\ [0.25in]
c^{\, \rm rst}_{k\ell}(x,z;\gamma) \, \triangleq \,
e_{k\ell}^+ \phi_{\rm ub}(t_{\ell},\gamma) - e_{k\ell}^- \phi_{\rm lb}(t_{\ell},\gamma)
\\ [0.1in]
\hspace{0.8in} = e_{k\ell}^+ \min\left\{ \, \max\left( \,
\wh{\theta}_{\rm cvx}\left( 1 + \displaystyle{
\frac{t_{\ell}}{\gamma}
} \right), \, 0 \, \right), \, 1 \, \right\} - e_{k\ell}^- \max\left\{ \,
\min\left( \, \wh{\theta}_{\rm cve}\left( \displaystyle{
\frac{t_{\ell}}{\gamma}
} \right), \, 1 \, \right), \ 0 \, \right\}.
\end{array} \]
By Proposition~\ref{pr:usc and lsc of phi}, it follows that
$c^{\, \rm rlx}_{k\ell}(\bullet,z;\bullet)$ and
$c^{\, \rm rst}_{k\ell}(\bullet,z;\bullet)$ are lower and upper semicontinuous on
$X \times \mathbb{R}_+$, respectively.
Given $( x,z,\bar{x} ) \in X \times\Xi \times X$, let
\[ \begin{array}{l}
Lg_{i\ell}(x,z;\bar{x}) \, \triangleq \, g_{i\ell}(\bar{x},z) +
\nabla_x g_{i\ell}(\bar{x},z)^{\top}( \, x - \bar{x} \, ) \, \leq \, g_{i\ell}(x,z)
\\ [0.1in]
Lh_{j\ell}(x,z;\bar{x}) \, \triangleq \, h_{j\ell}(\bar{x},z) +
\nabla_x h_{j\ell}(\bar{x},z)^{\top}( \, x - \bar{x} \, ) \, \leq \, h_{j\ell}(x,z)
\end{array} \]
be the linearizations at $\bar{x}$ of the functions
$g_{i\ell}(\bullet,z)$ and $h_{j\ell}(\bullet,z)$ evaluated at $x$,
and define
\[
\left\{ \begin{array}{ll}
L{\cal Z}_{\ell}^{\, h}(x,z;\bar{x}) \,\triangleq \, \displaystyle{
\max_{1 \leq i \leq I_{\ell}}
} \, g_{i\ell}(x,z) - \displaystyle{
\max_{1 \leq j \leq J_{\ell}}
}\, Lh_{j\ell}(x,z;\bar{x}) \\[0.15in]
L{\cal Z}_{\ell}^{\, g}(x,z;\bar{x}) \,\triangleq \, \displaystyle{
\max_{1 \leq i \leq I_{\ell}}
} \, Lg_{i\ell}(x,z;\bar{x}) - \displaystyle{
\max_{1\leq j \leq J_{\ell}}
} \, h_{j\ell}(x,z).
\end{array}\right.
\]
Note that the functions $L{\cal Z}_{\ell}^{\, g}(\bullet,z;\bullet)$ and
$L{\cal Z}_{\ell}^{\, h}(\bullet,z;\bullet)$ are both continuous.
Obviously we have
\begin{equation} \label{eq:L-ctilde rlx}
L{\cal Z}_{\ell}^{\, g}(x,z;\bar{x}) \,\leq \, \left[\, \underbrace{\displaystyle{
\max_{1 \leq i \leq I_{\ell}}
} \, g_{i\ell}(x,z) - \displaystyle{
\max_{1 \leq j \leq J_{\ell}}
} \, h_{j\ell}(x,z)}_{\mbox{$= {\cal Z}_{\ell}(x,z)$}} \,\right] \, \leq \,
L{\cal Z}_{\ell}^{\, h}(x,z;\bar{x});
\end{equation}
these inequalities yield
\begin{equation} \label{eq:ctilde rlx}
\left\{ \begin{array}{ll}
c^{\, \rm rlx}_k(x,z;\gamma) \, \leq \,\wh{c}_k^{\, \rm rlx}(x,z;\gamma;\bar{x})
\,\triangleq \, \displaystyle{
\sum_{\ell=1}^L
} \, \left[ \, \begin{array}{l}
e_{k\ell}^+ \phi_{\rm lb}( \bullet,\gamma) \circ
L{\cal Z}_{\ell}^{\, h}(x,z;\bar{x}) \ - \\ [0.1in]
e^-_{k\ell} \phi_{\rm ub}( \bullet,\gamma)
\circ L{\cal Z}_{\ell}^{\, g}(x,z;\bar{x})
\end{array} \right] \\ [0.3in]
c^{\, \rm rst}_k(x,z;\gamma) \, \leq \, \wh{c}_k^{\, \rm rst}(x,z;\gamma;\bar{x})
\, \triangleq \, \displaystyle{
\sum_{\ell=1}^L
} \, \left[ \, \begin{array}{l}
e^+_{k\ell} \, \phi_{\rm ub}( \bullet,\gamma) \circ
L{\cal Z}_{\ell}^{\, h}(x,z;\bar{x}) \ - \\ [0.1in]
e^-_{k\ell} \,\phi_{\rm lb}(\bullet,\gamma)
\circ L{\cal Z}_{\ell}^{\, g}(x,z;\bar{x}) \,
\end{array} \right].
\end{array} \right.
\end{equation}
  {We note that the inequalities in (\ref{eq:L-ctilde rlx}) and (\ref{eq:ctilde rlx}) all
hold as equalities for $x = \bar{x}$.}
The two majorizing functions $\wh{c}^{\, \rm rlx}_k(x,z;\gamma;\bar{x})$
and $\wh{c}^{\, \rm rst}_k(x,z;\gamma;\bar{x})$ lead to two
majorized subproblems being solved at each iteration:
\begin{equation} \label{eq:tilde relax surrogate}
\displaystyle{
\operatornamewithlimits{\mbox{\bf minimize}}_{x \in X}
} \ \displaystyle{
\frac{1}{N_{\nu}}
} \, \displaystyle{
\sum_{s=1}^{N_{\nu}}
} \, \wh{c}_0(x,z^s;x^{\nu})+ \lambda_{\nu} \displaystyle{
\sum_{k=1}^K
} \, \max\left( \displaystyle{
\frac{1}{N_{\nu}}
} \, \displaystyle{
\sum_{s=1}^{N_{\nu}}
} \, \wh{c}_k^{\, \rm rlx}(x,z^s;\gamma_{\nu};x^{\nu}) - \zeta_k, 0 \right)
+ \displaystyle{
\frac{\rho_{\nu}}{2}
} \, \| \, x - x^{\nu} \, \|^2
\end{equation}
for the relaxed problem (\ref{eq:relaxed sp_dc_constraint}), and
\begin{equation} \label{eq:tilde surrogate}
\displaystyle{
\operatornamewithlimits{\mbox{\bf minimize}}_{x \in X}
} \ \displaystyle{
\frac{1}{N_{\nu}}
} \, \displaystyle{
\sum_{s=1}^{N_{\nu}}
} \, \wh{c}_0(x,z^s;x^{\nu})+ \lambda_{\nu} \displaystyle{
\sum_{k=1}^K
} \, \max\left( \displaystyle{
\frac{1}{N_{\nu}}
} \, \displaystyle{
\sum_{s=1}^{N_{\nu}}
} \, \wh{c}_k^{\, \rm rst}(x,z^s;\gamma_{\nu};x^{\nu}) - \zeta_k, 0 \right)
+ \displaystyle{
\frac{\rho_{\nu}}{2}
} \, \| \, x - x^{\nu} \, \|^2
\end{equation}
for the restricted problem (\ref{eq:restricted sp_dc_constraint}), respectively.
Postponing the discussion of the practical solution of the above two subproblems in
Appendix~2, we remark that
similar to $c^{\, \rm rlx}_{k\ell}(\bullet,z;\bullet)$ and
$c^{\, \rm rst}_{k\ell}(\bullet,z;\bullet)$, for fixed $z$, the functions
$\wh{c}^{\, \rm rlx}_k(x,z;\gamma;\bar{x})$ and
$\wh{c}^{\, \rm rst}_k(x,z;\gamma;\bar{x})$
are, respectively, lower and upper semicontinuous at every
$(x,\gamma,\bar{x}) \in X \times \mathbb{R}_+ \times X$ including $\gamma = 0$ for which
\[ \begin{array}{l}
\wh{c}^{\, \rm rlx}_k(x,z;0;\bar{x}) = \displaystyle{
\sum_{\ell=1}^L
} \, \left[ \, e_{k\ell}^+ \, \onebld_{( 0,\infty )}(\bullet) \circ
L{\cal Z}_{\ell}^{\, h}(x,z;\bar{x}) -
e^-_{k\ell} \, \onebld_{[ 0,\infty )}(\bullet) \circ
L{\cal Z}_{\ell}^{\, g}(x,z;\bar{x}) \, \right]
\, \triangleq \, c_{k,{\rm lb}}(x,z;\bar{x}) \\ [0.2in]
\wh{c}^{\, \rm rst}_k(x,z;0;\bar{x}) = \displaystyle{
\sum_{\ell=1}^L
} \, \left[ \, e^+_{k\ell} \, \onebld_{[ 0,\infty )}(\bullet) \circ
L{\cal Z}_{\ell}^{\, h}(x,z;\bar{x}) -
e_{k\ell}^- \onebld_{( 0,\infty )}(\bullet) \circ
L{\cal Z}_{\ell}^{\, g}(x,z;\bar{x}) \, \right] \, \triangleq \,
c_{k,{\rm ub}}(x,z;\bar{x}).
\end{array} \]
Notice that $c_{k,{\rm lb}}(x,z;\bar{x}) \leq c_{k,{\rm ub}}(x,z;\bar{x})$
and for all $x \in X$,
\[ \begin{array}{lll}
c_{k,{\rm lb}}(x,z;x) & = & \displaystyle{
\sum_{\ell=1}^L
} \, \left[ \, \left( \, e^+_{k\ell} \, \onebld_{( 0,\infty )}(\bullet) - e^-_{k\ell} \,
\onebld_{[ 0,\infty )}(\bullet) \, \right) \, \right] \circ {\cal Z}_{\ell}(x,z)
\\ [0.25in]
& \leq & \displaystyle{
\sum_{\ell=1}^L
} \, \left[ \, \left( \, e^+_{k\ell} \, \onebld_{[ 0,\infty )}(\bullet) - e^-_{k\ell} \,
\onebld_{( 0,\infty )}(\bullet) \, \right) \, \right] \circ {\cal Z}_{\ell}(x,z)
\, = \, c_{k,{\rm ub}}(x,z;x).
\end{array} \]
The following lemma shows how these functions satisfy
the required inequalities (\ref{eq:Bunderline}) and (\ref{eq:Bupper}).

\begin{lemma} \label{lm:ULLN diminishing usc1} \rm
Let $x^{\infty} \in X$ satisfy the zero-probability condition:
$\mathbb{P}( {\cal Z}_{\ell}(x^{\infty},\tilde{z}) = 0 ) = 0$ for all $\ell \in [ L ]$.
Then the inequalities (\ref{eq:Bunderline}) and (\ref{eq:Bupper}) hold at $x^{\infty}$
with $\wt{c}_k(x,z;\bullet) \triangleq c_{k,{\rm ub}}(x,z;\bullet)$ for both the
restricted $\wh{c}_k^{\, \rm rst}(\bullet,z;\gamma;\bullet)$ and relaxed
$\wh{c}_k^{\, \rm rlx}(\bullet,z;\gamma;\bullet)$
functions.
\end{lemma}

\begin{proof}
Let $\{ ( y^{\nu},x^{\nu} ) \} \subseteq X \times X$ converge to
$( x^\infty,x^\infty )$.  We have, almost surely,
\[
\begin{array}{l}
\displaystyle{
\liminf_{\nu (\in \kappa) \to \infty}
} \,  \displaystyle{
\frac{1}{N_{\nu}}
} \, \displaystyle{
\sum_{s=1}^{N_{\nu}}
} \, \wh{c}_k^{\, \rm rst}(y^\nu,z^s;\gamma_{\nu};x^{\nu})
\, \geq \, \displaystyle{
\liminf_{\nu (\in \kappa) \to \infty}
} \,  \displaystyle{
\frac{1}{N_{\nu}}
} \, \displaystyle{
\sum_{s=1}^{N_{\nu}}
} \, \wh{c}^{\, \rm rlx}_k(y^\nu,z^s;\gamma_{\nu};x^{\nu}) \\ [0.2in]
\epc \geq \, \mathbb{E}[\,c_{k,{\rm lb}}(x^{\infty},\tilde{z};x^{\infty})\,] \
\left\{ \mbox{\begin{tabular}{l}
by Proposition \ref{pr:uniform pointwise ULLN} and the \\ [3pt]
lower semicontinuity of $\wh{c}^{\, {\rm rlx}}_k(\bullet,z;\bullet;\bullet)$
\end{tabular}} \right. \\ [0.2in]
\epc = \, \mathbb{E}[\,c_{k,{\rm ub}}(x^{\infty},\tilde{z};x^{\infty})\,] \quad
\mbox{ by the zero-probability assumption at $x^{\infty}$}.
\end{array}
\]
This establishes (\ref{eq:Bunderline}) for the restricted function
$\wh{c}_k^{\, \rm rst}(\bullet,z;\gamma;\bullet)$.
To prove (\ref{eq:Bupper}), let $\{ ( y^{\nu},x^{\nu} ) \} \subseteq X \times X$
converge to $( y^\infty,x^\infty )$.
By the upper semicontinuity of the function
$\wh{c}_k^{\, \rm rst}(\bullet,z;\bullet;\bullet)$ mentioned above, we immediately have
\begin{equation} \label{eq:rstlimsup}
\displaystyle{
\limsup_{\nu (\in \kappa) \to \infty}
} \,  \displaystyle{
\frac{1}{N_{\nu}}
} \, \displaystyle{
\sum_{s=1}^{N_{\nu}}
} \, \wh{c}_k^{\, \rm rst}(y^{\nu},z^s;\gamma_{\nu};x^{\nu}) \, \leq \,
\mathbb{E}[ \, c_{k,{\rm ub}}(y^{\infty},\tilde{z};x^{\infty}) \, ],
\epc \mbox{almost surely}.
\end{equation}
To prove (\ref{eq:Bunderline}) for the relaxed function
$\wh{c}_k^{\, \rm rlx}(\bullet,z;\gamma;\bullet)$, let
$\{ ( y^{\nu},x^{\nu} ) \} \subseteq X \times X$ converge to $( x^\infty,x^\infty )$.
We have already noted the following inequality in the above proof:
\[
\displaystyle{
\liminf_{\nu (\in \kappa) \to \infty}
} \,  \displaystyle{
\frac{1}{N_{\nu}}
} \, \displaystyle{
\sum_{s=1}^{N_{\nu}}
} \, \wh{c}_k^{\, \rm rlx}(y^\nu,z^s;\gamma_{\nu};x^{\nu}) \geq \,
\mathbb{E}[\,c_{k,{\rm ub}}(x^{\infty},\tilde{z};x^{\infty})\,], \epc
\mbox{almost surely},
\]
under the zero-probability assumption.
To prove (\ref{eq:Bupper}), let $\{ ( y^{\nu},x^{\nu} ) \} \subseteq X \times X$
converge to $( y^\infty,x^\infty )$.  We have almost surely
\[
\begin{array}{lll}
\displaystyle{
\limsup_{\nu (\in \kappa) \to \infty}
} \,  \displaystyle{
\frac{1}{N_{\nu}}
} \, \displaystyle{
\sum_{s=1}^{N_{\nu}}
} \, \wh{c}_k^{\, \rm rlx}(y^{\nu},z^s;\gamma_{\nu};x^{\nu}) & \leq &
\displaystyle{
\limsup_{\nu (\in \kappa) \to \infty}
} \,  \displaystyle{
\frac{1}{N_{\nu}}
} \, \displaystyle{
\sum_{s=1}^{N_{\nu}}
} \, \wh{c}_k^{\, \rm rst}(y^{\nu},z^s;\gamma_{\nu};x^{\nu}) \\ [0.25in]
& \leq & \mathbb{E}[ \, c_{k,{\rm ub}}(y^{\infty},\tilde{z};x^{\infty}) \, ],
\end{array} \]
as in (\ref{eq:rstlimsup}).
\end{proof}

We can easily get the following corollary of
Theorem \ref{th:convergence diminishing gamma} based on the above lemma.

\begin{corollary} \label{co:convergence restriced diminishing gamma} \rm
Let the blanket assumptions for the functions $c^{\, \rm rst}_{k\ell}$ and
$c^{\, \rm rlx}_{k\ell}$ be valid.
Let $\gamma_{\nu} \downarrow 0$ and $\{ x^{\nu} \}$ be a sequence of iterates
produced by the SPSA with the majorization functions
$\{ \wh{c}_k^{\, \rm rst} \}_{k \in [ K ]}$ or
$\{ \wh{c}_k^{\, \rm rlx} \}_{k \in [ K ]}$ defined by \eqref{eq:ctilde rlx}.
Let $x^{\infty}$ be an accumulation point of the
subsequence $\{ x^{\nu} \}_{\nu \in \kappa}$ corresponding to any infinite index set
$\kappa \in {\cal K}$.  Suppose
$\mathbb{P}( {\cal Z}_{\ell}(x^{\infty},\tilde{z}) = 0 ) = 0$ for all $\ell \in [ L ]$,
then the following two statements (i) and (ii) hold for $x^{\infty}$ almost surely:

\gap

(i)   {the following three statements (ia), (ib), and (ic) are equivalent:}

\gap

(ia) there exists $\wh{x} \in X$ satisfying
\[
\displaystyle{
\sum_{\ell=1}^L
} \, \left[ \, e^+_{k\ell} \, \mathbb{P}\left(
L{\cal Z}_{\ell}^{\, h}(\wh{x},\tilde{z};x^{\infty}) \geq 0 \right) -
e^-_{k\ell} \, \mathbb{P}\left(
L{\cal Z}_{\ell}^{\, g}(\wh{x},\tilde{z};x^{\infty}) > 0 \right) \, \right] \, \leq \,
\zeta_k, \epc \forall \, k \, \in [ K ],
\]
and also
\[
\displaystyle{
\frac{1}{\lambda_{\infty}}
} \, \mathbb{E}\left[ \, \wh{c}_0(\wh{x},\tilde{z};x^{\infty}) \, \right] \, \leq \,
\displaystyle{
\frac{1}{\lambda_{\infty}}
} \,\mathbb{E}\left[ \, \wh{c}_0(x^{\infty},\tilde{z};x^{\infty}) \, \right],
\]
(ib) $x^{\infty}$ is feasible to the problem:
\begin{equation} \label{eq:restricted problem at limit diminishing gamma}
\begin{array}{ll}
\displaystyle{
\operatornamewithlimits{\mbox{\bf minimize}}_{x \in X}
} & \mathbb{E}\left[ \, \wh{c}_0(x,\tilde{z};x^{\infty}) \, \right] \\ [0.15in]
\mbox{\bf subject to} & \displaystyle{
\sum_{\ell=1}^L
} \, \left[ \, \begin{array}{l}
e^+_{k\ell} \, \mathbb{P}\left(
L{\cal Z}_{\ell}^{\, h}(x,\tilde{z};x^{\infty}) \geq 0 \right) \ - \\ [0.1in]
e^-_{k\ell} \, \mathbb{P}\left(
L{\cal Z}_{\ell}^{\, g}(x,\tilde{z};x^{\infty}) > 0 \right)
\end{array} \right] \, \leq \, \zeta_k, \epc \forall \, k \, \in [ K ],
\end{array} \end{equation}
(ic) $x^{\infty}$ is feasible to the CCP (\ref{eq:focus sp_dc_constraint});

\gap

(ii) if in addition the closure condition holds:
\[
\emptyset \, \neq \, \displaystyle{
\bigcap_{k=1}^K
} \, \left\{ \, x \, \in \, X \, \mid \,\mathbb{E}\left[ \,
c_{k, {\rm ub}}(x,\tilde{z};x^{\infty}) \,\right] \, \leq \, \zeta_k \, \right\}
\, \subseteq  \, \mbox{cl}\left( \, \displaystyle{
\bigcap_{k=1}^K
} \, \left\{ \, x \, \in \, X \, \mid \, \mathbb{E}\left[ \,
c_{k,{\rm ub}}(x,\tilde{z};x^{\infty}) \, \right] \, < \, \zeta_k \, \right\} \, \right),
\]
then $x^{\infty}$ is a B-stationary point to the problem
(\ref{eq:restricted problem at limit diminishing gamma}).  \hfill $\Box$
\end{corollary}

\begin{proof} It suffices to observe that (ic) implies (ia) if $\wh{x} = x^{\infty}$.
\end{proof}

\begin{remark} \label{rm:theorem diminishing continued} \rm
Note that
\[ \begin{array}{l}
e^+_{k\ell} \, \mathbb{P}\left( {\cal Z}_{\ell}(x,\tilde{z}) \geq 0 \right) -
e^-_{k\ell} \, \mathbb{P}\left( {\cal Z}_{\ell}(x,\tilde{z}) \geq 0 \right)
\, \leq \,
e^+_{k\ell} \, \mathbb{P}\left( {\cal Z}_{\ell}(x,\tilde{z}) \geq 0 \right) -
e^-_{k\ell} \, \mathbb{P}\left( {\cal Z}_{\ell}(x,\tilde{z}) > 0 \right) \\ [0.15in]
\, \leq \, e^+_{k\ell} \mathbb{P}\left( L{\cal Z}_{\ell}^{\, h}(x,\tilde{z};x^{\infty})
\geq 0 \right) - e^-_{k\ell} \,
\mathbb{P}\left( L{\cal Z}_{\ell}^{\, g}(x,\tilde{z};x^{\infty}) > 0 \right).
\end{array} \]
Thus the feasible set of (\ref{eq:restricted problem at limit diminishing gamma})
is a subset of the feasible set of the CCP (\ref{eq:focus sp_dc_constraint}), which is:
\begin{equation} \label{eq:original cc}
\displaystyle{
\bigcap_{k=1}^K
} \, \left\{ \, x \, \in \, X \, \mid \, \displaystyle{
\sum_{\ell=1}^L
} \, e_{k\ell} \, \mathbb{P}\left( {\cal Z}_{\ell}(x,\tilde{z})
\geq 0 \right) \, \leq \, \zeta_k \, \right\}.
\end{equation}
As such, the B-stationarity of $x^{\infty}$ for the problem
(\ref{eq:restricted problem at limit diminishing gamma}) is weaker than the
B-stationarity of $x^{\infty}$ for the original CCP (\ref{eq:focus sp_dc_constraint}).
This is regrettably the best we can do at this time in this case of
$\gamma_{\nu} \downarrow 0$.  A noteworthy remark about the setting of
Corollary~\ref{co:convergence restriced diminishing gamma} is that it is rather broad;
in particular, there is no restriction on the sign of the coefficients $e_{k\ell}$
and pertains to a fairly
general class of difference-of-convex random functionals ${\cal Z}_{\ell}(\bullet,z)$.

\gap

Continuing from Remark~\ref{rm:theorem diminishing}, we note that without the
zero-probability assumption at $x^{\infty}$, this limit point would satisfy
\[
\displaystyle{
\sum_{\ell=1}^L
} \, \left[ \, e^+_{k\ell} \, \mathbb{P}\left( {\cal Z}_{\ell}(x,\tilde{z}) > 0 \right)
- e^-_{k\ell} \, \mathbb{P}\left( {\cal Z}_{\ell}(x,\tilde{z}) \geq 0 \right)
\, \right] \, \leq \, \zeta_k, \epc \forall \, k \, \in \, [ K ],
\]
which is a relaxation of the chance constraint in (\ref{eq:original cc}); moreover,
for the B-stationarity of $x^{\infty}$ in part (ii) of
Corollary~\ref{co:convergence restriced diminishing gamma} to be valid, one needs
to assume that $x^{\infty}$ is feasible to
(\ref{eq:restricted problem at limit diminishing gamma}).  Hence there is a gap in
the two conclusions of the corollary.  With the zero-probability assumption in place,
this gap disappears.  A noteworthy final remark is that the latter
zero-probability assumption is made only at a limit point of the sequence
$\{ x^{\nu} \}_{\nu \in \kappa}$.  \hfill $\Box$
\end{remark}

\subsection{\textcolor{black}{A summary of the SPSA and its convergence}}

\textcolor{black}{$\bullet $ {\bf Blanket assumptions:}}

--- $X$ is a polytope on which the objective $c_0(\bullet,z)$ is nonnegative for all
$z \in \Xi$;
each scalar $e_{k\ell}$ has the signed decomposition
$e_{k\ell}^+ - e_{k \ell}^-$ for all $(k,\ell) \in [ K ] \times [ L ]$;

--- ($\boldsymbol{\cal Z}$) in
Section~\ref{sec:piecewise probabilistic constraints} for the bivariate functions
${\cal Z}_{\ell}$; and
($\boldsymbol{\Theta}$) in Subsection~\ref{subsec:approx Heaviside} for
the functions $\wh{\theta}_{\rm cvx/cve}$;

--- the objective function $c_0(\bullet,\tilde{z})$ and the
constraint functions $\left\{ c_k(\bullet,\tilde{z};\gamma) \triangleq \displaystyle{
\sum_{\ell=1}^L
} \, c_{k\ell}(\bullet,\tilde{z};\gamma) \right\}_{k \in [ K ]}$
satisfy the blanket assumptions ({\bf A}$_{\rm o}$), ({\bf A}$_{\rm c}$) and
the interchangeability of directional derivatives ({\bf I}$_{\rm dd}$);
moreover, the functions $c_0(\bullet,\tilde{z})$ and
$\left\{ c_k(\bullet,\tilde{z};\bullet) \right\}_{k=1}^K$ have a uniform finite
variance on $X$.

\gap

\textcolor{black}{$\bullet $ {\bf Set-up:}}

Combining the structural assumptions of the random functionals
${\cal Z}_{\ell}(x,\tilde{z})$
with the $\gamma$-approximations of the Heaviside functions
by the truncation of the functions $\wh{\theta}_{\, \rm cvx/cve}$,
we obtain the restricted/relaxed
approximations $c_{k\ell}^{\, \rm rst/rlx}(x,\tilde{z};\gamma)$ of the
probability function $e_{k\ell} \, \mathbb{P}({\cal Z}_{\ell}(x,\tilde{z}) \geq 0)$.
Index-set or subgradient based surrogation functions
$\wh{c}_{k\ell}^{\, \rm rst/rlx}(\bullet,z;\gamma;\bar{x})$
at an arbitrary $\bar{x}$ are
derived for $c_{k\ell}^{\, \rm rst/rlx}(\bullet,\tilde{z};\gamma)$
in terms of the pointwise minima of finitely many convex functions; the surrogation of
the relaxed function $c_{k\ell}^{\, \rm rlx}(\bullet,\tilde{z};\gamma)$ requires
additionally that $\wh{\theta}_{\, \rm cvx/cve}$ be either differentiable or
piecewise affine.

\gap

\textcolor{black}{$\bullet $ {\bf Main computations:}}

Iteratively solve the subproblems
(\ref{eq:proximal Vsubproblem}) for a sequence of parameters
\[
\{ \, N_{\nu}; \lambda_{\nu}; \rho_{\nu}; \gamma_{\nu} \, \}_{\nu=1}^{\infty}
\]
that obey the conditions specified before the description of the SPSA (in particular
(\ref{eq:ratio of rho and lambda})) and also those in
Lemma~\ref{lm:condition on sample sizes}.  Let $\displaystyle{
\lim_{\nu \to \infty}
} \, \lambda_{\nu} = \lambda_{\infty} \in [ 1,\infty ]$.

\gap

\textcolor{black}{$\bullet $ {\bf Convergence:}}

Specialized to the restricted and relaxed problems,
the two general convergence results, Theorem~\ref{th:convergence finite gamma}
(finite $\underline{\gamma}$) and
Theorem~\ref{th:convergence diminishing gamma} (diminishng $\gamma_{\nu} \downarrow 0)$,
assert the feasibility/stationarity of any limit point
$x^{\infty}$ of the sequence $\{ x^{\nu} \}$ such that
\[
x^{\infty} \, = \, \displaystyle{
\lim_{\nu (\in \kappa) \to \infty}
} \, x^{\nu} \, = \, \displaystyle{
\lim_{\nu (\in \kappa) \to \infty}
} \, x^{\nu+1}
\]
for some infinite subset of iteration counters $\kappa \subseteq \{ 1, 2, \cdots \}$.

\gap

--- Fixed parameter $\underline{\gamma}$:
Suppose $\gamma_{\nu} = \underline{\gamma}$ for all $\nu$.
Under the following three conditions:

\gap

$\bullet $ there exists $\wh{x} \in X$ such that
$\displaystyle{
\frac{1}{\lambda_{\infty}}
} \, \mathbb{E}\left[ \, \wh{c}_0(\wh{x},\tilde{z};x^{\infty}) \, \right] \, \leq \,
\displaystyle{
\frac{1}{\lambda_{\infty}}
} \,\mathbb{E}\left[ \, \wh{c}_0(x^{\infty},\tilde{z};x^{\infty}) \, \right]$ and
\[
\wh{c}_k^{\, \rm rst/rlx}(\wh{x};\underline{\gamma};x^{\infty})
\, \leq \, \zeta_k, \epc \forall \, k \, \in \, [ K ];
\]
$\bullet $
$\emptyset \, \neq \, \displaystyle{
\bigcap_{k=1}^K
} \, \left\{ \, x \, \in \, X \, \mid \,
\wh{c}_k^{\, \rm rst/rlx}(x;\underline{\gamma};x^{\infty})
\, \leq \, \zeta_k \, \right\}
\, \subseteq \, \mbox{cl} \left( \, \displaystyle{
\bigcap_{k=1}^K
} \, \left\{ \, x \, \in \, X \, \mid \,
\wh{c}_k^{\, \rm rst/rlx}(x;\underline{\gamma};x^{\infty})
\, < \, \zeta_k \, \right\} \right)$,
which implies that $x^{\infty}$ is feasible to (\ref{eq:new SP}), and

\gap

$\begin{array}{l}
\bullet \displaystyle{
\bigcap_{k \in {\cal A}(x^{\infty})}
} \, \left\{ \, v \, \in \, {\cal T}(x^{\infty};X) \, \mid \,
\bar{c}_k^{\, \rm rst/rlx}(\bullet;\underline{\gamma})^{\prime}(x^{\infty};v)
\, \leq 0 \, \right\} \\ [0.25in]
\hspace{0.2in} \subseteq \mbox{cl}\left( \, \displaystyle{
\bigcap_{k \in {\cal A}\left( x^{\infty};\underline{\gamma} \right)}
} \, \left\{ \, v \, \in \, {\cal T}(x^{\infty};X) \, \mid \,
\bar{c}_k^{\, \rm rst/rlx}(\bullet;\underline{\gamma})^{\prime}(x^{\infty};v)
\, < 0 \, \right\} \, \right)
\end{array}$

\gap

then $x^{\infty}$ is a B-stationary point of (\ref{eq:new SP}).

\gap

--- Diminishing parameter:
Suppose $\mathbb{P}( {\cal Z}_{\ell}(x^{\infty},\tilde{z}) = 0 ) = 0$ for all
$\ell \in [ L ]$.  With the surrogation
functions $\wh{c}_{k\ell}^{\, \rm rst/rlx}(\bullet,z;\gamma;\bar{x})$
defined by (\ref{eq:restricted problem at limit diminishing gamma}),
Corollary~\ref{co:convergence restriced diminishing gamma} holds for the case
$\gamma_{\nu} \downarrow 0$.

\section{Conclusions}

In this paper, we have provided a thorough variational analysis for the affine
chance-constrained stochastic program with nonconvex and nondifferentiable
random functionals.  The discontinuous indicator functions in the probabilistic
constraints are approximated by a general class of parameterized
difference-of-convex functions that are not necessarily smooth.    A practically
implementable convex programming based sampling schemes with incremental
sample batches combined with exact penalization and upper surrogation
is proposed to solve the problem.  Subsequential convergence of the generated
sequences are established under both fixed parametric approximations
and diminishing ones.


\gap

\underline{\bf Appendix 1: Derivation of majorization functions for
$c^{\, \rm rlx}_{k\ell}(\bullet,z;\gamma)$ and
$c^{\, \rm rst}_{k\ell}(\bullet,z;\gamma)$.}

\gap

We show how the structures of the relaxed $c^{\, \rm rlx}_{k\ell}(\bullet,z;\gamma)$
and the restricted $c^{\, \rm rst}_{k\ell}(\bullet,z;\gamma)$ functions
can be used to define a pointwise-minimum of convex functions majorizing these
functions to be used in the sequential sampling algorithm.
\textcolor{black}{These surrogation functions have their origin in the
reference \cite{PangRazaAlvarado16}
for deterministic problems with difference-of-max-convex functions of the
kind (\ref{eq:g and h}); this
initial work is subsequently extended in \cite{CuiPangSen18} to problems with
convex composite such dc functions.  In particular, numerical results in the
latter reference demonstrate the practical viability of the solution method.
Further numerical results with similar surrogation functions for solving related
problems can be found in reference \cite{CuiHePang20} for multi-composite nonconvex
optimization problems arising from deep neural networks with piecewise activation
functions, and in \cite{CuiHePang20b} for solving certain
robustified nonconvex optimization problems.  While the problems in these references
are all deterministic, the paper \cite{LiuPang20} has some numerical
results for a stochastic difference-of-convex algorithm for solving certain
nonconvex risk minimization problems with expectation objectives but not the
difference-of-max-convex structure.}


\gap

{\bf Step 1:}  We start by (a) letting
$t_{\ell} \triangleq {\cal Z}_{\ell}(x,z) = g_{\ell}(x,z) - h_{\ell}(x,z)$,
(b)
substituting this expression in the truncation functions
$\phi_{\rm ub}$ and $\phi_{\rm lb}$, and (c) using the increasing property of
the functions $\wh{\theta}_{\rm cvx}$ and $\wh{\theta}_{\rm cve}$; this yields
\[ \begin{array}{l}
c^{{\, \rm rst}}_{k\ell}(x,z;\gamma) \, \triangleq \, e_{k \ell}^+ \, \phi_{\rm ub}({\cal Z}_{\ell}(x,z ), \gamma) -
e_{k \ell}^- \, \phi_{\rm lb}({\cal Z}_{\ell}(x,z ), \gamma) \\ [0.15in]
= \, e_{k \ell}^+ \, \wh{\theta}_{\rm cvx}\left( \, \min\left\{ \, \max\left( \, 1 + \displaystyle{
\frac{t_{\ell}}{\gamma}
}, \, 0 \, \right), \, 1 \, \right\} \, \right) - e_{k \ell}^- \, \wh{\theta}_{\rm cve}\left( \, \max\left\{ \, \min\left( \, \displaystyle{
\frac{t_{\ell}}{\gamma}
}, \, 1 \, \right), \, 0 \, \right\} \, \right),
\end{array} \]
Similarly, we have
\[ \begin{array}{l}
c^{{\, \rm rlx}}_{k\ell}(x,z;\gamma) \, \triangleq \, e_{k \ell}^+ \,
\phi_{\rm lb}({\cal Z}_{\ell}(x,z ), \gamma) -
e_{k \ell}^- \, \phi_{\rm ub}({\cal Z}_{\ell}(x,z ), \gamma) \\ [0.1in]
= \, -\left[ \, e_{k \ell}^- \, \wh{\theta}_{\rm cvx}\left( \, \min\left\{ \, \max\left( \, 1 + \displaystyle{
\frac{t_{\ell}}{\gamma}
}, \, 0 \, \right), \, 1 \, \right\} \, \right) - e_{k \ell}^+ \, \wh{\theta}_{\rm cve}\left( \, \max\left\{ \, \min\left( \, \displaystyle{
\frac{t_{\ell}}{\gamma}
}, \, 1 \, \right), \, 0 \, \right\} \, \right) \, \right]
\end{array} \]

{\bf Step 2.}  Using the difference-of-convex decomposition of the truncation operator
\[ 
T_{[ \, 0,1 \, ]}(t) \, \triangleq \, \min\left\{ \, \max( t,0 ), \, 1 \, \right\}
\, = \, \max\left\{ \, \min( t,1 ), \, 0 \, \right\} 
\, = \, \max( t,0 ) - \max( t-1, 0), 
\]
we may obtain
\begin{equation} \label{eq:cgamma}
\begin{array}{l}
\left( \begin{array}{l}
c^{\, \rm rst}_{k\ell}(x,z;\gamma) \\ [0.1in]
c^{\, \rm rlx}_{k\ell}(x,z;\gamma)
\end{array} \right) \\ [0.3in]
= \, \pm \, e_{k \ell}^{\pm} \, \wh{\theta}_{\rm cvx}\left( \,
\underbrace{\max\left\{ \, 1 + \displaystyle{
\frac{g_{\ell}(x,z)}{\gamma}
}, \ \displaystyle{
\frac{h_{\ell}(x,z)}{\gamma}
} \, \right\}}_{\mbox{denoted $g_{\ell}^1(x,z;\gamma)$}} - \underbrace{
\max\left\{ \, \displaystyle{
\frac{g_{\ell}(x,z)}{\gamma}
}, \ \displaystyle{
\frac{h_{\ell}(x,z)}{\gamma}
} \, \right\}}_{\mbox{denoted $gh_{\ell}(x,z;\gamma)$}} \right) \\ [0.5in]
\epc \mp \, e_{k \ell}^{\mp} \, \wh{\theta}_{\rm cve}\left( \,
\max\left\{ \, \displaystyle{
\frac{g_{\ell}(x,z)}{\gamma}
}, \ \displaystyle{
\frac{h_{\ell}(x,z)}{\gamma}
} \, \right\} -  \underbrace{\max\left\{ \, \displaystyle{
\frac{g_{\ell}(x,z)}{\gamma}
} - 1, \ \displaystyle{
\frac{h_{\ell}(x,z)}{\gamma}
} \, \right\}}_{\mbox{denoted $h_{\ell}^1(x,z;\gamma)$}} \, \right).
\end{array}
\end{equation}
{\bf Step 3.} By the difference-of-max definition of $g_{\ell}(\bullet,z)$ and
$h_{\ell}(\bullet,z)$ in (\ref{eq:g and h}),
there are two ways to obtain the majorizations, termed {\sl index-set based} and
{\sl subgradient-based}, respectively.  As the terms suggest, the former makes use
of the  pointwise maximum
structure of these functions, whereas the latter uses only the subgradients
$\partial g_{\ell}(\bullet,z)$
and  $\partial h_{\ell}(\bullet,z)$ of these functions.

\gap

{\bf Index-set based majorization}:  First employed in \cite{PangRazaAlvarado16} for
deterministic difference-of-convex programs and
later extended in \cite{CuiPangSen18} to convex composite difference-max programs,
this approach is based on several index sets defined at a given pair $(x,z)$:
\[ \begin{array}{lll}
{\cal A}^{\, g}_{\ell}(x,z) & \triangleq & \displaystyle{
\operatornamewithlimits{\mbox{\bf argmax}}_{1 \leq i \leq I_{\ell}}
} \, g_{i\ell}(x,z) \, = \, \left\{ \, i \, \mid \, g_{i\ell}(x,z) \, = \,
g_{\ell}(x,z) \, \right\}, \
\mbox{where } \ g_{\ell}(x,z) \, \triangleq \,  \displaystyle{
\operatornamewithlimits{\mbox{\bf max}}_{1 \leq i \leq I_{\ell}}
} \, g_{i\ell}(x,z) \\ [0.1in]
{\cal A}^{\, h}_{\ell}(x,z) & \triangleq & \displaystyle{
\operatornamewithlimits{\mbox{\bf argmax}}_{1 \leq j \leq J_{\ell}}
} \, h_{j\ell}(x,z) \, = \, \left\{ \, j \, \mid \, h_{j\ell}(x,z) \, = \,
h_{\ell}(x,z) \, \right\}, \
\mbox{where } \ h_{\ell}(x,z) \, \triangleq \,  \displaystyle{
\operatornamewithlimits{\mbox{\bf max}}_{1 \leq j \leq J_{\ell}}
} \, h_{j\ell}(x,z);
\end{array} \]
moreover let $\wh{\cal A}^{\, g}_{\ell}(x,z)$ and $\wh{\cal A}^{\, h}_{\ell}(x,z)$ be
any subset of
$\left[ \, I_{\ell} \, \right] \times \left[ \, J_{\ell} \, \right]$ such that
\begin{equation} \label{eq:choice of index sets}
\wh{\cal A}^{\, g}_{\ell}(x,z) \, \cap \, {\cal A}^{\, g}_{\ell}(x,z) \, \neq \,
\emptyset \epc \mbox{and} \epc
\wh{\cal A}^{\, h}_{\ell}(x,z) \, \cap \, {\cal A}^{\, h}_{\ell}(x,z) \, \neq \,
\emptyset.
\end{equation}
Let $\wh{\cal A}^{\, gh}_{\ell}(x,z) \triangleq \wh{\cal A}^{\, g}_{\ell}(x,z) \times
\wh{\cal A}^{\, h}_{\ell}(x,z)$.
Notice that all these index sets do not depend on the scalar $\gamma$.  Several
noteworthy choices of such index sets include:
(a) singletons, (b) an $\varepsilon$-argmax for a given $\varepsilon \geq 0$:
\[ \begin{array}{lll}
{\cal A}^{\, g}_{\ell;\varepsilon}(x,z) & \triangleq &\displaystyle{
\operatornamewithlimits{\varepsilon-\mbox{\bf argmax}}_{1 \leq i \leq I_{\ell}}
} \, g_{i\ell}(x,z) \, = \, \left\{ \, i \, \mid \, g_{i\ell}(x,z) \, \geq \,
g_{\ell}(x,z) - \varepsilon \, \right\} \\ [0.1in]
{\cal A}^{\, h}_{\ell;\varepsilon}(x,z) & \triangleq & \displaystyle{
\operatornamewithlimits{\varepsilon-\mbox{\bf argmax}}_{1 \leq j \leq J_{\ell}}
} \, h_{j\ell}(x,z) \, = \, \left\{ \, j \, \mid \, h_{j\ell}(x,z) \, \geq \,
h_{\ell}(x,z) - \varepsilon \, \right\},
\end{array} \]
and (c) the full sets: $\wh{\cal A}^{\, g}_{\ell}(x,z) = [ I_{\ell} ]$ and
$\wh{\cal A}^{\, h}_{\ell}(x,z) = [ J_{\ell} ]$.  The last choice was used in
Subsection~\ref{subsec:gamma to zero}.
The two families
$\left\{ \, {\cal A}^{\, g}_{\ell;\varepsilon}(x,z) \, \right\}_{\varepsilon \geq 0}$
and $\left\{ \, {\cal A}^{\, h}_{\ell;\varepsilon}(x,z) \, \right\}_{\varepsilon \geq 0}$
are nondecreasing in $\varepsilon$ and each member therein
contains the respective sets ${\cal A}^{\, g}_{\ell}(x,z)$ and
${\cal A}^{\, h}_{\ell}(x,z)$ that correspond to $\varepsilon = 0$.
For any fixed but arbitrary vector  $\bar{x}$ and any pair
$(i,j)$ of indices in $[ I_{\ell} ] \times [ J_{\ell} ]$, let
\[ \begin{array}{l}
Lg_{i\ell}(x,z;\bar{x}) \, \triangleq \, g_{i\ell}(\bar{x},z) + \nabla_x g_{i\ell}(\bar{x},z)^{\top}( \, x - \bar{x} \, )
\, \leq \, g_{i\ell}(x,z) \\ [0.1in]
Lh_{j\ell}(x,z;\bar{x}) \, \triangleq \, h_{j\ell}(\bar{x},z) + \nabla_x h_{j\ell}(\bar{x},z)^{\top}( \, x - \bar{x} \, )
\, \leq \, h_{j\ell}(x,z)
\end{array} \]
be the linearizations of $g_{i\ell}(\bullet,z)$ and $h_{j\ell}(\bullet,z)$ at $\bar{x}$ evaluated at $x$, respectively.
It can be shown that
{\small
\begin{equation} \label{eq:bound for restricted}
\begin{array}{lll}
c^{{\, \rm rst}}_{k\ell}(x,z;\gamma) & \leq &
\left\{ \begin{array}{l}
\displaystyle{
\min_{(i,j) \in \wh{\cal A}^{\, gh}_{\ell}(\bar{x},z)}
} \, \min\left( \, \begin{array}{l}
\underbrace{e_{k \ell}^+ \, \wh{\theta}_{\rm cvx}\left( \, g_{\ell}^1(x,z;\gamma) -
\displaystyle{
\frac{Lg_{i\ell}(x,z;\bar{x})}{\gamma}
} \, \right)}_{\mbox{convex in $x$}}, \\ [0.4in]
\underbrace{e_{k \ell}^+ \, \wh{\theta}_{\rm cvx}\left( \, g_{\ell}^1(x,z;\gamma) -
\displaystyle{
\frac{Lh_{j\ell}(x,z;\bar{x})}{\gamma}
} \, \right)}_{\mbox{convex in $x$}}
\end{array} \right) \ + \\ [0.8in]
\epc \displaystyle{
\min_{(i,j) \in \wh{\cal A}^{\, gh}_{\ell}(\bar{x},z)}
} \, \min\left( \, \begin{array}{l}
\underbrace{\left[ -e_{k \ell}^- \, \wh{\theta}_{\rm cve} \right] \circ \left( \,
\displaystyle{
\frac{Lg_{i\ell}(x,z;\bar{x})}{\gamma}
} - h_{\ell}^1(x,z;\gamma) \, \right)}_{\mbox{convex in $x$}}, \\ [0.4in]
\underbrace{\left[ -e_{k \ell}^- \, \wh{\theta}_{\rm cve} \right] \circ \left( \,
\displaystyle{
\frac{Lh_{j\ell}(x,z;\bar{x})}{\gamma}
} - h_{\ell}^1(x,z;\gamma) \, \right)}_{\mbox{convex in $x$}}
\end{array} \right)
\end{array} \right\} \\ [1.4in]
& \triangleq & \wh{c}^{{\, \rm rst}}_{k\ell}(x,z;\gamma;\bar{x}) \\ [5pt]
& = & \mbox{pointwise mimimum of finitely many convex (albeit not necessarily} \\
& & \mbox{differentiable) functions}.
\end{array} \end{equation}
}

\noindent We note that the right-hand bounding function coincides with
$c^{{\, \rm rst}}_{k\ell}(x,z;\gamma)$ at the reference vector $x = \bar{x}$.

\gap

The derivation of a similar pointwise minimum-convex majorization of
$c^{{\, \rm rlx}}_{k\ell}(\bullet,z;\gamma)$
requires the base functions $\wh{\theta}_{\rm cvx}$ and $\wh{\theta}_{\rm cve}$
to be either (continuously) differentiable
or piecewise affine.  Lemma~~\ref{lm:PA of PA} takes care of latter case.
Consider the differentiable case.   We have
{\small
\[ \begin{array}{l}
c^{{\, \rm rlx}}_{k\ell}(x,z;\gamma) \, \leq \underbrace{-e_{k \ell}^- \,
\wh{\theta}_{\rm cvx}\left( \, g_{\ell}^1(\bar{x},z;\gamma) -
gh_{\ell}(\bar{x},z;\gamma) \, \right) + e_{k \ell}^+ \,
\wh{\theta}_{\rm cve}\left( \,gh_{\ell}(\bar{x},z;\gamma) -
h_{\ell}^{\, 1}(\bar{x},z;\gamma) \, \right)}_{\mbox{a constant given
$(\bar{x},z,\gamma)$}} \\ [0.15in]
\epc - \, e_{k \ell}^- \left[ \underbrace{\wh{\theta}_{\rm cvx}^{\, \prime}\left( \,
g_{\ell}^1(\bar{x},z;\gamma) -
gh_{\ell}(\bar{x},z;\gamma) \, \right)}_{\mbox{a nonnegative constant given $\bar{x}$}}
\right] \,
\left[ \underbrace{\left( g_{\ell}^1(x,z;\gamma) - g_{\ell}^1(\bar{x},z;\gamma) \right) -
\left( \, gh_{\ell}(x,z;\gamma) - gh_{\ell}(\bar{x},z;\gamma) \, \right)}_{\mbox{
diff-ptwise max of finitely many cvx fncs in $x$}} \right] \\ [0.4in]
\epc + \, e_{k \ell}^+ \left[ \underbrace{\wh{\theta}_{\rm cve}^{\, \prime}\left( \,
gh_{\ell}(\bar{x},z;\gamma) - h_{\ell}^1(\bar{x},z;\gamma) \right)}_{\mbox{
a nonnegative constant given $\bar{x}$}} \right] \,
\left[ \underbrace{\left( \, gh_{\ell}(x,z;\gamma) - gh_{\ell}(\bar{x},z;\gamma) \right)
- \left( h_{\ell}^1(x,z;\gamma) - h_{\ell}^1(\bar{x},z;\gamma) \right)}_{
\mbox{diff-ptwise max of finitely many cvx fncs in $x$}} \right],
\end{array} \]
}

\noindent which shows that, with $\bar{x}$ given,
$c^{{\, \rm rlx}}_{k\ell}(\bullet,z;\gamma)$ can be upper bounded by a difference
of pointwise maxima of finitely many convex
functions.  By substituting the expressions for the functions
$g_{\ell}^1(\bullet,z;\gamma)$ and $h_{\ell}^1(\bullet,z;\gamma)$,
it can be shown that the latter bounding function can be further upper bounded
by the pointwise minimum of finite many convex functions.
The end result is a bounding function $\wh{c}^{\, \rm rlx}_{k\ell}(x,z;\gamma;\bar{x})$
similar to (\ref{eq:bound for restricted}) obtained by
\[ \begin{array}{l}
\mbox{replacing } g_{\ell}^1(\bullet,z;\gamma) \mbox{ by }
Lg_{\ell}^1(\bullet,z;\gamma;\bar{x}) \, \triangleq \,
\max\left\{ \, 1 + \displaystyle{
\frac{\displaystyle{
\max_{i \in \wh{\cal A}^g_{\ell}(\bar{x},z)}
} \, Lg_{i\ell}(\bullet,z;\bar{x})}{\gamma}
}, \ \displaystyle{
\frac{\displaystyle{
\max_{j \in \wh{\cal A}^h_{\ell}(\bar{x},z)}
} \, Lh_{j\ell}(\bullet,z;\bar{x})}{\gamma}
} \, \right\} \\ [0.4in]
\mbox{replacing } h_{\ell}^1(\bullet,z;\gamma) \mbox{ by }
Lh_{\ell}^1(\bullet,z;\gamma;\bar{x}) \, \triangleq \,
\max\left\{ \, -1 + \displaystyle{
\frac{\displaystyle{
\max_{i \in \wh{\cal A}^g_{\ell}(\bar{x},z)}
} \, Lg_{i\ell}(\bullet,z;\bar{x})}{\gamma}
}, \ \displaystyle{
\frac{\displaystyle{
\max_{j \in \wh{\cal A}^h_{\ell}(\bar{x},z)}
} \, Lh_{j\ell}(\bullet,z;\bar{x})}{\gamma}
} \, \right\}
\end{array} \]
and keeping the function $gh_{\gamma;\ell}(x,z)$ in the expression without upper
bounding it.

\gap

To close the discussion of the index-set based majorization, we make an important
remark when the pair of index sets
$\left( \, \wh{\cal A}^{\, g}_{\ell}(x,z),\wh{\cal A}^{\, h}_{\ell}(x,z) \, \right)$
is chosen to be
$ \left( \, {\cal A}^{\, g}_{\ell;\varepsilon}(x,z),{\cal A}^{\, h}_{\ell;\varepsilon}
(x,z) \, \right)$ for a given $\varepsilon \geq 0$.
Namely, for any such $\varepsilon$, the resulting majorization for the restricted
functions satisfies the directional derivative consistency condition; that is,
\[
\left[ \, c^{\, \rm rst}_{k\ell}(\bullet,z;\gamma) \, \right]^{\, \prime}(\bar{x};v)
\, = \,
\left[ \, \wh{c}^{\, \rm rst}_{k\ell}(\bullet,z;\gamma;\bar{x}) \, \right]^{\, \prime}
(\bar{x};v)
\epc \forall \, (\bar{x},z,v) \in X \times \Xi \times \mathbb{R}^n;
\]
moreover, if $\wh{\theta}_{\rm cvx}$ and $\wh{\theta}_{\rm cve}$ are differentiable,
then the same holds for the relaxed functions; that is,
\[
\left[ \, c^{\, \rm rlx}_{k\ell}(\bullet,z;\gamma) \, \right]^{\, \prime}(\bar{x};v)
\, = \,
\left[ \, \wh{c}^{\, \rm rlx}_{k\ell}(\bullet,z;\gamma;\bar{x}) \, \right]^{\, \prime}
(\bar{x};v)
\epc \forall \, (\bar{x},z,v) \in X \times \Xi \times \mathbb{R}^n.
\]
Nevertheless, the majorization functions
$\wh{c}^{\, \rm rst}_{k\ell}(\bullet,z;\gamma;\bullet)$ and
$\wh{c}^{{\, \rm rlx}}_{k\ell}(\bullet,z;\gamma;\bullet)$ are upper semicontinuous
if $\varepsilon > 0$ and
may not be so if
$\left( \, \wh{\cal A}^{\, g}_{\ell}(x,z),\wh{\cal A}^{\, h}_{\ell}(x,z) \, \right) =
\left( \, {\cal A}^{\, g}_{\ell}(x,z),{\cal A}^{\, h}_{\ell}(x,z) \, \right)$.

\gap

{\bf Subgradient-based majorization}:  This approach has its origin from the early
days of deterministic difference-of-convex (dc)
programming \cite{LeThiPham05}; it is most recently extended in
the study of compound stochastic programs with multiple expectation functions
\cite{LiuCuiPang20}.
The approach provides a generalization to the choice of a single index in defining
the sets
$\wh{\cal A}^g_{\ell}(\bar{x},z)$ and/or $\wh{\cal A}^h_{\ell}(\bar{x},z)$;
it has the computational advantage of
avoiding the pointwise-minimum surrogation when these sets are not singletons.
Specifically, we choose a single surrogation function
from each of the following families:
\[ \begin{array}{lll}
{\cal G}_{\ell}(\bar{x},z) & \triangleq & \left\{ \, G_{\ell}(x,z;\bar{x}) \,
\triangleq \, g_{\ell}(\bar{x},z) +
( \eta_{\ell}^{\, g} )^{\top}( \, x - \bar{x} \, ) - h_{\ell}(x,z) \, : \,
\eta_{\ell}^{\, g} \, \in \, \partial_x g_{\ell}(\bar{x},z) \, \right\} \\ [0.1in]
{\cal H}_{\ell}(\bar{x},z) & \triangleq & \left\{ \, H_{\ell}(x,z;\bar{x}) \,
\triangleq \, g_{\ell}(x,z) - h_{\ell}(\bar{x},z) -
( \eta_{\ell}^{\, h} )^{\top}( \, x - \bar{x} \, ) \, : \, \eta_{\ell}^{\, h}
\, \in \, \partial_x h_{\ell}(\bar{x},z) \, \right\} .
\end{array}
\]
A member $G_{\ell}(x,z;\bar{x}) \in {\cal G}_{\ell}(\bar{x},z)$
will then replace the corresponding pointwise-maximum based surrogation $\displaystyle{
\max_{i \in \wh{\cal A}^g_{\ell}(\bar{x},z)}
} \, Lg_{i\ell}(x,z;\bar{x})$; similarly for the $h$-functions.  The end result is
that we will obtain a single convex function
$\wh{c}_{k\ell}^{\, \rm rst}(\bullet,z;\gamma;\bar{x})$ majorizing
$c_{k\ell}^{\, \rm rst}(\bullet,z;\gamma)$ at the reference vector $\bar{x}$;
and similarly for the relaxed function.
We omit the details of these other surrogation functions.

\gap

\underline{\bf Appendix~2: convex programming for the minimization
of (\ref{eq:proximal Vsubproblem}):}

\gap

\textcolor{black}{
The minimization problem (\ref{eq:proximal Vsubproblem}) is of the following form:
\begin{equation} \label{eq:Vsubproblem repeat}
\displaystyle{
\operatornamewithlimits{\mbox{\bf minimize}}_{x \in X}
} \ \underbrace{\displaystyle{
\frac{1}{N}
} \, \displaystyle{
\sum_{s=1}^N
} \,  \wh{c}_0(x,z^s;\bar{x}) + \displaystyle{
\frac{\rho}{2}
} \, \| x - \bar{x} \|^2}_{\mbox{denoted $\eta(x)$}} + \lambda \, \displaystyle{
\sum_{k=1}^K
} \, \max\left( \, \displaystyle{
\frac{1}{N}
} \, \displaystyle{
\sum_{s=1}^N
} \, \displaystyle{
\sum_{\ell=1}^L
} \, \wh{c}_{k\ell}(x,z^s;\gamma;\bar{x}) - \zeta_k, \, 0 \, \right),
\end{equation}
where, as derived above,
each function $\wh{c}_{k\ell}(\bullet,z^s;\gamma;\bar{x})$
is the pointwise minimum of finitely convex functions
(cf.\ e.g. (\ref{eq:bound for restricted})).
To simplify the discussion, we assume that
$\eta(x)$ is convex, so that we can focus on explaining how a global
minimizer of this problem can be obtained by solving finitely many convex programs, with
a proper manipulation of the second summation term.
For this purpose, we further assume that
\[
\wh{c}_{k\ell}(x,z^s;\gamma;\bar{x}) - \displaystyle{
\frac{\zeta_k}{L}
} \, = \, \displaystyle{
\min_{1 \leq i \leq M_{k\ell}^s}
} \, \chi_{k\ell i}^s(x), \epc ( \, k,\ell,s \, ) \, \in \, [ \, K \, ] \times
[ \, L \, ] \times [ \, S \, ],
\]
for some sample-dependent positive integers $M_{k\ell}^s$, with each $\chi_{k\ell i}^s$
being convex.  We have
\[ \begin{array}{l}
\max\left\{ \, \displaystyle{
\frac{1}{N}
} \, \displaystyle{
\sum_{s=1}^N
} \, \displaystyle{
\sum_{\ell=1}^L
} \, \wh{c}_{k\ell}(x,z^s;\gamma;\bar{x}) - \zeta_k, \, 0 \, \right\} \\ [0.2in]
= \, \max\left\{ \, \displaystyle{
\frac{1}{N}
} \, \min\left( \, \displaystyle{
\sum_{s=1}^N
} \, \displaystyle{
\sum_{\ell=1}^L
} \, \chi_{k\ell{i_{k\ell}^s}}^s(x) \ \left| \ \{ i_{k\ell}^s \}_{s=1}^N \in
\displaystyle{
\prod_{s=1}^N
} \, [ \, M_{k\ell}^s \, ] \, \right. \, \right), \, 0 \, \right\} \\ [0.3in]
= \, \displaystyle{
\frac{1}{N}
} \, \underbrace{\min\left\{ \, \max\left( \,  \displaystyle{
\sum_{s=1}^N
} \, \displaystyle{
\sum_{\ell=1}^L
} \, \chi_{k\ell{i_{k\ell}^s}}^s(x), \, 0 \, \right) \ \left| \
\{ i_{k\ell}^s \}_{s=1}^N \in \displaystyle{
\prod_{s=1}^N
} \, [ \, M_{k\ell}^{\, s} \, ] \, \right. \, \right\}}_{
\mbox{pointwise mininum of finitely many convex functions}}.
\end{array} \]
Hence the problem (\ref{eq:Vsubproblem repeat}) is equivalent to
\[
\underbrace{\min\left[ \, \underbrace{\displaystyle{
\operatornamewithlimits{\mbox{\bf minimize}}_{x \in X}
} \ \underbrace{\left\{ \, \begin{array}{l}
\eta(x) \ + \\ [0.1in]
\displaystyle{
\sum_{k=1}^K
} \, \max\left( \,  \displaystyle{
\sum_{s=1}^N
} \, \displaystyle{
\sum_{\ell=1}^L
} \, \chi_{k\ell{i_{k\ell}^s}}^s(x), \, 0 \, \right)
\end{array} \right\}}_{\mbox{convex function}}}_{\mbox{cvx program for given tuple
$\left\{ \, i_{k\ell}^{\, s} \in [ M_{k\ell}^s ] \, \right\}_{(k,\ell) \in
[ K ] \times [ L ]}^{s \in [ N ]}$}}
\ \left| \	\left\{ \, \{ i_{k\ell}^{\, s} \}_{s=1}^N \in \displaystyle{
\prod_{s=1}^N
} \, [ \, M_{k\ell}^{\, s} \, ] \, \right\}_{k \in [ K ]}^{
\ell \in [ L ]} \right. \, \right]}_{\mbox{finitely many
$\left( \, \displaystyle{
\sum_{k=1}^K
} \, \displaystyle{
\sum_{l=1}^L
} \, \displaystyle{
\sum_{s=1}^N
} \, M_{k\ell}^s \, \right)$ convex programs}}
\]
}

Based on the above derivation, it can be seen that the subgradient-based majorization
leads to simpler workload per iteration in an iterative method for solving the nonconvex
nondifferentiable CCP; nevertheless, the stationarity properties
of the limit points of the iterates produced are typically weaker than those of the
limit points produced by an index-set based surrogation
where multiple convex subprograms are solved.  So the
tradeoff between practical computational efforts and theoretical sharpness of the
computed solutions is something to be recognized in the numerical
solution of the relaxed and/or restricted formulations of the chance-constrained
stochastic programs.  Among the index-set surrogations, some choices
may not yield desirable convergence results while others, at the expense of more (yet
still finite) computational efforts per iteration, would
yield desirable properties of the accumulation points of the iterates produced.  This is
exemplified by the choices
$\wh{c}_k^{\, \rm rlx}(x,z;\gamma;\bar{x})$ and
$\wh{c}_k^{\, \rm rst}(x,z;\gamma;\bar{x})$ in (\ref{eq:ctilde rlx}) for the
convergence analysis of the case $\gamma \downarrow 0$, where
the full index sets $[ \, I \, ]$ and $[ \, J \, ]$ are employed in the linearizations.

\gap

{\bf Acknowledgements.}  The authors thank both referees for their very careful
reading of our manuscript and for their many constructive comments that have helped
to improve its quality.

\end{document}